\documentclass[11pt,a4paper]{amsart}
\usepackage{amsmath, amssymb, mathtools, amsthm, graphicx, hyperref, cleveref, mathrsfs, euscript, stmaryrd, leftidx, pb-diagram, tikz-cd, pigpen, bbm, multirow, wrapfig, epsfig, epsf, epic, pdfsync, tikz, color, xcolor, xypic, indentfirst, textcomp, tcolorbox, fancyhdr, blindtext, appendix, relsize, multicol, cutwin, stackengine, wasysym}
\usepackage[T1]{fontenc}
\usepackage[margin=2cm]{geometry}
\usepackage{enumitem}
\usepackage[export]{adjustbox}

\theoremstyle{plain}
\newtheorem{thm}{Theorem}[section]
\newtheorem{theorem}[thm]{Theorem}

\newtheorem{prop}[thm]{Proposition}
\newtheorem{corollary}[thm]{Corollary}
\newtheorem{cj}[thm]{Conjecture}
\theoremstyle{definition}
\newtheorem{definition}[thm]{Definition}
\newtheorem{remark}[thm]{Remark}
\numberwithin{equation}{section}

\newtheorem*{theorem*}{Theorem}

\newcommand{\set}{\EuScript{S}\mathsf{et}}

\newcommand{\sset}{\mathsf{s}\EuScript{S}\mathsf{et}}

\newcommand{\sps}{\mathsf{s}\EuScript{P}\mathsf{re}}

\newcommand{\ssh}{\mathsf{s}\EuScript{S}\mathsf{hv}}
\newcommand{\grp}{\EuScript{G}\mathsf{r}}
\newcommand{\ab}{\EuScript{A}\mathsf{b}}

\newcommand*{\pb}{\mbox{\LARGE{$\mathrlap{\cdot}\lrcorner$}}}

\providecommand{\bysame}{\leavevmode\hbox to3em{\hrulefill}\thinspace}
\providecommand{\MR}{\relax\ifhmode\unskip\space\fi MR }
\providecommand{\MRhref}[2]{%
	\href{http://www.ams.org/mathscinet-getitem?mr=#1}{#2}
}

\begin{document}
	
\title{Enumerating Non-Stable Vector Bundles}
\author{Peng DU}

\begin{abstract}
		In this article, we establish a motivic analog of an enumeration result of James-Thomas \cite{JTh65} on non-stable vector bundles in topological setting. Using this, we obtain results on enumeration of projective modules of rank $d$ over a smooth affine $k$-algebra $A$ of dimension $d$, recovering in particular results of Suslin and Bhatwadekar on cancellation of such vector bundles. Admitting a conjecture of Asok and Fasel, we prove cancellation of such vector bundles of rank $d-1$ if the base field $k$ is algebraically closed. We also explore the cancellation properties of symplectic vector bundles.
\end{abstract}

\maketitle
\makeatletter
\patchcmd{\@tocline}
{\hfil}
{\leaders\hbox{\,.\,}\hfil}{}{}
\makeatother
{\large {\tableofcontents}}

\section{Introduction}

We begin with some classical problems in commutative ring and module theory. Let $R$ be a commutative ring, and $P$ be a finitely generated projective $R$-module (of constant rank $n$). We say that the projective module $P$ is \emph{cancellative} if for any $m>0$, $P\oplus R^m\cong Q\oplus R^m$ for some $R$-module $Q$ implies $P\cong Q$. Note that if $R$ is a local ring, then any finitely generated projective $R$-module is cancellative (since such modules are always free and the isomorphism class of a finitely generated free module is uniquely determined by its rank in that case).

We have the \emph{Serre splitting theorem}, which says that for a finitely generated projective module $P$ over a commutative noetherian ring $R$ of Krull dimension $d$, if its rank is at least $d+1$, then it has a free summand of rank $1$; and \emph{Bass cancellation theorem}: any finitely generated projective module $P$ over $R$ of rank at least $d+1$ is cancellative.

So it is interesting to understand further cancellation properties of projective modules of lower rank. In order to better apply results in algebraic geometry, we restrict to the case when the commutative ring $A$ is a $k$-algebra, $k$ is a field. Even for $k$ algebraically closed ($k=\bar{k}$), Mohan Kumar \cite{Kum85} gives a negative answer in the case when the rank $n=d-2$ is prime ($d=\dim A$): there exists a rank $n=d-2$ stably free module which is not free. We are thus left to deal with $n=d$ and $n=d-1$. For the case $n=d$, Suslin \cite{Sus77a} confirms cancellation in the case $k=\bar{k}$. This result was later extended to the case where the base field is a $C_1$-field by Bhatwadekar \cite[Theorem 4.1 and Remark 4.2]{Bha03}. For the case $n=d-1$, Fasel-Rao-Swan \cite{FRS} confirm cancellation for stably free modules in the case $k=\bar{k}$ (with other mild conditions).

\vspace{3mm}

The goal of this article is to give some enumeration results on non-stable vector bundles (oriented, or symplectic) and restudy the cancellation problem for projective modules over algebras, using modern tools. Precisely, we establish some enumeration formulas for such non-stable vector bundles, and based on such formulas, we explore cancellation properties of projective modules using obstruction theory in $\mathbb{A}^1$-homotopy theory (hence we need to assume our affine $k$-algebra $A$ is smooth), following some ideas in \cite{AF14, AF15}.

Concretely, we extend a topological result of James-Thomas \cite{JTh65} to the motivic homotopy setting, namely, we identify certain set of homotopy lifting classes with nice source and target---the mathematical objects we want to enumerate---with the cokernel of a certain map of abelian groups associated to some $\mathbb{A}^1$-homotopy classes. We then give a formula for that map, following the method of \cite{JTh65}; similar methods were also explored in \cite{Sut69}. This is essentially a study of the derived mapping space in $\mathbb{A}^1$-homotopy theory (see \Cref{bgl-comp}), taking advantage of the fact that stable things are abelian group objects even in the \emph{unstable} motivic homotopy category.

Combining with the Suslin matrix construction and some known results on the cohomology of certain motivic spaces, we are able to show that the image of that map contains $n!$ multiple of the target abelian group, which is the cohomology of the variety ${\rm Spec}\ A$ with coefficients in some Milnor (or Milnor-Witt) K-theory sheaf, by manipulating some characteristic classes; this gives an upper bound of the set of the lifting classes we study. Hence the set of homotopy lifting classes is a singleton if the target abelian group is $n!$-divisible. It is worth remarking that Suslin's matrix construction plays a very important role here---it provides \emph{sufficiently many elements} in the target abelian group which lie in the image of the map in question.

Further, using recent results---the Rost-Schmid complex and the motivic Bloch-Kato conjecture, one can show that the target cohomology group is indeed $n!$-divisible in many cases (e.g. when all elements in the residue fields at all closed points of ${\rm Spec}\ A$ are $n!$-th power).

In the $n=d$ case, we will mostly focus on the odd rank case. For the even rank case, there are some further difficulties. On the one hand, as the $\mathbb{A}^1$-fundamental group of $\mathrm{BGL}_{n}$ is non-trivial (i.e. $\mathrm{BGL}_{n}$ is not $\mathbb{A}^1$-simply connected), the relevant $\mathbb{A}^1$-homotopy fiber sequence we get via $\mathbb{A}^1$-homotopical obstruction theory is not principal hence do not fall into our framework as in \cite{JTh65}; so we make the compromise that we restrict to enumeration for \emph{oriented} vector bundles. On the other hand, the first nontrivial homotopy sheaf of the relevant space is more difficult to study, even in the case of oriented bundles. Nevertheless, it is still quite useful to get results using $\mathrm{BSL}_{n}, \mathrm{BSp}_{2n}$ etc. instead of $\mathrm{BGL}_{n}$, hence studying enumeration problem for non-stable \emph{oriented} (or \emph{symplectic}) vector bundles. So we will mostly discuss enumeration problem for oriented vector bundles.

The main $\mathbb{A}^1$-homotopy foundation comes from Morel's work \cite{Mor} where the base fields are all assumed to be perfect (and infinite, but later, Hogadi-Kulkarni \cite{HK} gave a published version of Gabber's presentation lemma for finite fields, which confirms that Morel's results are indeed also true for finite fields), but every field $k$ contains a perfect (prime) subfield $k_0$, and since the $k$-group schemes $\mathrm{GL}_{n}/k, \mathrm{SL}_{n}/k$ are extended from $k_0$, our arguments in the text \emph{hold for a general base field $k$} as well; cf. \cite[Comments on the proof of Theorem 3.1.7]{AFHo19}. To make the statements more concise, we just assume everywhere that $k$ is perfect.

Our main result in the $n=d$ case is the following (for more details, see \Cref{enu-oddrk} and \Cref{compute-delta}).
\begin{theorem*}
	Let $k$ be a field, $A$ a smooth affine $k$-algebra of odd Krull dimension $d\geqslant 3$, and $X={\rm Spec}(A)$. Let $\xi$ be a stable vector bundle over $X$, whose classifying map is still denoted $\xi: X\to\mathrm{BGL}$. Let $\varphi_*: [X, \mathrm{BGL}_{d}]_{\mathbb{A}^1}\to[X, \mathrm{BGL}]_{\mathbb{A}^1}$ be the stabilizing map. Then there is a bijection
	\[\varphi_*^{-1}(\xi)\cong{\rm coker}\big({\rm K}_1(X)\xrightarrow{\Delta(c_{d+1}, \xi)}{\rm H}^{d}(X; \mathbf{K}^{\mathrm{M}}_{d+1})\big),\]
	where the homomorphism $\Delta(c_{d+1}, \xi)$ is given as follows: for $\beta\in{\rm K}_1(X)=[X, \Omega\mathrm{BGL}]_{\mathbb{A}^1}=[X, \mathrm{GL}]_{\mathbb{A}^1}$,
	\[\Delta(c_{d+1}, \xi)\beta=(\Omega c_{d+1})(\beta)+\sum_{r=1}^{d}((\Omega c_{r})(\beta))\cdot c_{d+1-r}(\xi).\]
	Here, $(\Omega c_{i})(\beta)\in {\rm H}^{i-1}(X; \mathbf{K}^{\mathrm{M}}_{i})$ for $i=1,\cdots, d+1$ are the Chern classes of $\beta$ and $c_{i}(\xi)$ are the ordinary Chern classes of $\xi$.
	
	In other words, given a vector bundle $\xi$ of rank $d$ over $X$, then the set of isomorphism classes vector bundles $\nu$ such that $\nu\oplus\mathscr{O}_X\cong\xi\oplus\mathscr{O}_X$ is in bijection with ${\rm coker}\Delta(c_{d+1}, \xi)$.
	
	In case $A$ is of even dimension $d\geqslant 2$ and the base field is of cohomological dimension at most $2$, the same results as above hold for oriented bundles (see \Cref{enu-evenrk}).
\end{theorem*}
As a corollary, we see that if $k$ is algebraically closed or of cohomological dimension at most $1$ then all rank $d$ vector bundles are cancellative (\Cref{divisability-cancel}), hence recovering the cancellation theorems of Suslin and Bhatwadekar.

In fact, the discussion in \Cref{d!-div} tells us that our result generalizes Bhatwadekar's in several ways: firstly, we relax Bhatwadekar's characteristic zero assumption on the base field (e.g. our result is valid if $k$ is a finite field, which was also confirmed by Kumar-Murthy-Roy \cite{KMR88}); secondly, we give an answer to Bhatwadekar's question at the end of the introduction section in \cite{Bha03}, at least if the affine algebra $A$ is smooth over $k$; thirdly, our result is applicable to all fields of cohomological dimension at most $1$, and such fields are strictly more than $C_1$-fields, see e.g. \cite{CT05}. Similar result also holds for $d$ even and oriented vector bundles (see \Cref{enu-evenrk} and \Cref{divisability-cancel-even}). In \Cref{d!-div}, we also give a class of examples for which both general criteria (Bhatwadekar's and ours here, which need cohomological dimension at most $1$ to guarantee cancellation in general) lose efficacy, yet we can still use \Cref{divisability-cancel} (2) in that situation.

It is also possible to explore the idea in this article to get further cancellation results in the rank $d-1$ case, which is more difficult. In \S 7, we obtain a conditional cancellation result by a study of a 2-stage Moore-Postnikov factorization, in which the first stage is very similar to the results in the rank $d$ case and gives a condition on the characteristic of the base field $k$ (namely, ${\rm char}(k)=0$ or ${\rm char}(k)\geqslant d$; see \Cref{divisability-cancel-d-1}) to ensure that the isomorphism class of a stable bundle has at most one lifting to the first stage in the tower; for the second stage, we invoke a conjecture of Asok-Fasel (\Cref{afcj}) describing $\pi_{d}^{\mathbb{A}^1}(\mathbb{A}^{d}\setminus 0)$, which gives a vanishing result on some cohomology group and ensures that the isomorphism class of the lifting map from the first stage to the second stage in the tower (exists and) is unique, which again involves studying the derived mapping space in $\mathbb{A}^1$-homotopy theory, taking advantage of the abelian group object structures.

The (conditional) cancellation result in the $n=d-1$ case is the following (see \Cref{cancel-d-1}).

\begin{theorem*}
	Assume that the base field $k$ is algebraically closed. Let $A$ be a smooth $k$-algebra of Krull dimension $d\geqslant 3$ and assume ${\rm char}(k)=0$ or ${\rm char}(k)\geqslant d$. Let $X={\rm Spec}\ A$. Assume that \Cref{afcj} holds, i.e. we have a sequence
	\begin{equation}
	\mathbf{K}^{\mathrm{M}}_{d+2}/24\to\pi_{d}^{\mathbb{A}^1}(\mathbb{A}^{d}\setminus 0)\to\mathbf{GW}^d_{d+1}\to0
	\end{equation}
	of homomorphisms which becomes exact after $d-3$-fold contraction, then every oriented rank $d-1$ vector bundle over $X$ is cancellative.
\end{theorem*}

\vspace{3mm}

 In \S 8, we prove a cancellation result for symplectic vector bundles, which we summarize as follows (see \Cref{borel-delta,enu-spbdl}).
 \begin{theorem*}
 	Let  $k$ be a perfect field with ${\rm char}(k)\neq 2$ and ${\rm c.d.}_2(k)\leqslant 2$, $X={\rm Spec}(A)$ be a smooth affine $k$-scheme of dimension $d=2n+1\geqslant 3$. Let $\xi$ be a stable symplectic vector bundle over $X$, whose classifying map is still denoted by $\xi: X\to\mathrm{BSp}$.
 	\begin{enumerate}[label=(\arabic*)]
 		\item There is an abelian group homomorphism
 		\[\Delta(b_{n+1}, \xi): {\rm KSp}_1(X)=[X, \Omega\mathrm{BSp}]_{\mathbb{A}^1}\to[X, \mathrm{K}(\mathbf{K}^{\mathrm{MW}}_{2n+2}, 2n+1)]_{\mathbb{A}^1}={\rm H}^{2n+1}(X; \mathbf{K}^{\mathrm{MW}}_{2n+2})\]
 		associated to the class $\xi\in[X, \mathrm{BSp}]_{\mathbb{A}^1}$ given by
 		\[\Delta(b_{n+1}, \xi)(\beta)=(\Omega b_{n+1})(\beta)+\sum_{r=1}^{n}((\Omega b_{r})(\beta))\cdot b_{n+1-r}(\xi),\]
 		whose cokernel is in bijection with $\EuScript{V}_{2n}^{\rm Sp}(X, \xi)$, the set of rank $d-1=2n$ symplectic vector bundles over $X$. Here $b_i\in{\rm H}^{2i}(\mathrm{BSp}; \mathbf{K}^{\mathrm{MW}}_{2i})$ is the $i$-th Borel class (as in \cite{PW10}).
 		\item We have $(2\cdot d!)\cdot{\rm H}^d(X; \mathbf{K}^{\mathrm{M}}_{d+1})\subset{\rm im}(\tau\Delta(b_{n+1}, \xi))$.
 		\item If ${\rm H}^d(X; \mathbf{K}^{\mathrm{M}}_{d+1})$ is $d!$-divisible, then ${\rm coker}\Delta(b_{n+1}, \xi)=0$. So in this case, any rank $2n=d-1$ symplectic vector bundle is cancellative. Moreover,  the map
 		\[(b_{n+1})_*: \pi_1({\rm RMap}(X, \mathrm{BSp}), \xi)\to\pi_1({\rm RMap}(X, \mathrm{K}(\mathbf{K}^{\mathrm{MW}}_{d+1}, d+1)), 0)={\rm H}^{d}(X; \mathbf{K}^{\mathrm{MW}}_{d+1})\]
 		is surjective for every $\xi\in[X, \mathrm{BSp}]_{\mathbb{A}^1}$.
 		\item If ${\rm c.d.}(k)\leqslant 1$, then every rank $2n=d-1$ symplectic vector bundle over $X$ is cancellative.
 	\end{enumerate}
 \end{theorem*}

In particular, for $d=2n+1\geqslant 3$ and a given symplectic vector bundle $\xi$ of rank $d-1$ over $X$, if ${\rm c.d.}(k)\leqslant 1$, then any symplectic vector bundle $\nu$ such that $\nu\oplus H^m\cong\xi\oplus H^m$ is in fact isomorphic to $\xi$, where $H$ is the trivial rank $2$ symplectic vector bundle over $X$.

\vspace{3mm}

The structure of this article is briefly as follows. Firstly, as a warm-up, we start in \S 2 with some fundamental constructions for simplicial sets, some already familiar from classical homotopy theory. In \S 3, we present one of the most important input of $\mathbb{A}^1$-homotopy theory, the motivic Moore-Postnikov tower (in a form slightly more general than that in the existing literature), which is our obstruction-theoretic tool; and (relative) cohomology groups in our motivic homotopy world, this is useful and important in analyzing various maps in the sequel. In \S 4, we identify the first non-trivial $\mathbb{A}^1$-homotopy sheaf of the $\mathbb{A}^1$-homotopy fiber of the map from non-stable to stable classifying spaces. In \S 5, we set up the framework of enumeration of lifting $\mathbb{A}^1$-homotopy classes, following ideas developed in James-Thomas \cite{JTh65} in the classical homotopy setting. In \S 6, we apply the previous results to enumeration problems for non-stable (oriented) vector bundles by some computations and identifying some characteristic classes (with mild conditions); along the way, we also describe (in \Cref{chern-delta}) more concretely and more conceptually the computing formula we give by studying the derived mapping space in $\mathbb{A}^1$-homotopy theory, which says that the formula we provide essentially describes the homomorphism of the fundamental group of a general component of the derived mapping space in question to another derived mapping space given by the $k$-invariant in question, surjectivity of which gives uniqueness of the lifting class. In \S 7, we give a (conditional) cancellation result in the $n=d-1$ case, provided the Asok-Fasel \Cref{afcj} holds. In the last section, we establish an (unconditional) cancellation result for symplectic vector bundles, in a similar fashion as before, taking great advantages of the results developed in the previous sections, with slightly more care.

\section{Some fundamental results and constructions for simplicial sets}

Let $\sset$ be the category of simplicial sets, with the usual (Kan-Quillen) model structure. In this section, we focus on the associated pointed category $\sset_*$ of pointed simplicial sets, with the induced model structure (see \cite{Hir}).

We first give a generalization of the usual path fibration to the relative case. Let $q: (E, v)\rightarrow (B, b)$ be a Kan fibration of pointed Kan complexes, with fiber $F$ over $b\in B_0$. Assume $q$ admits a section $s: B\rightarrow E$ (we do not require $s(b)=v$). We denote this by the diagram
\begin{equation}
\begin{tikzcd}
F \arrow[r] & E  \arrow[r,shift left,"q"]
\arrow[r,<-,shift right,swap,"s"]& B.
\end{tikzcd}
\label{splitfib}
\end{equation}
We define a map $E\rightarrow B^{\Delta^1}$ to be the composite $E\xrightarrow{q} B=B^{\Delta^0}\rightarrow B^{\Delta^1}$ (the latter arrow is given by taking ``constant paths'', i.e. the map $(\Delta^1\rightarrow\Delta^0)^*$); it appears as the left vertical arrow in the second of the following commutative diagrams
\[\begin{tikzcd}[column sep=4em]
E^{\Delta^1} \arrow[r, "(\partial\Delta^1\hookrightarrow\Delta^1)^*"] \arrow[d, "q^{\Delta^1}"']  & E^{\partial\Delta^1} \arrow[d, "q^{\partial\Delta^1}"] \\
B^{\Delta^1} \arrow[r, "(\partial\Delta^1\hookrightarrow\Delta^1)^*"']          & B^{\partial\Delta^1}
\end{tikzcd}
\qquad
\text{and}
\qquad
\begin{tikzcd}[column sep=4em]
E \arrow[r, "{(sq, 1)}"] \arrow[d]  & E^{\partial\Delta^1}=E\times E \arrow[d, "q^{\partial\Delta^1}=q\times q"] \\
B^{\Delta^1} \arrow[r, "(\partial\Delta^1\hookrightarrow\Delta^1)^*"']          & B^{\partial\Delta^1}=B\times B.
\end{tikzcd}\]
We obtain the induced maps
$$\alpha: E^{\Delta^1}\rightarrow B^{\Delta^1}\times_{B^{\partial\Delta^1}}E^{\partial\Delta^1} \text{\ and\ } \beta: E\rightarrow B^{\Delta^1}\times_{B^{\partial\Delta^1}}E^{\partial\Delta^1}.$$

We make $E^{\Delta^1}$ and $B^{\Delta^1}$ pointed by taking as base points the composites $\Delta^1\rightarrow\Delta^0\xrightarrow{v}E$ and  $\Delta^1\rightarrow\Delta^0\xrightarrow{b}B$ respectively; the other spaces also have similar base points. Then the map $\alpha: E^{\Delta^1}\rightarrow B^{\Delta^1}\times_{B^{\partial\Delta^1}}E^{\partial\Delta^1}$ is a fibration (see \cite[Proposition 9.3.8]{Hir}) with fibre over the above chosen base point the loop space $\Omega F$. Now we define a map $u$ via the pull-back diagram
\begin{equation}
\begin{tikzcd}
\mathrm{P}_BE \arrow[r] \arrow[d, "u"']\arrow[dr, phantom, "\pb", near start]  & E^{\Delta^1}  \arrow[d, "\alpha"] \\
E \arrow[r, "\beta"']          & B^{\Delta^1}\times_{B^{\partial\Delta^1}}E^{\partial\Delta^1},
\end{tikzcd}
\qquad
\text{yielding a fibre sequence}
\qquad
\begin{tikzcd}
\Omega F \arrow[r] & \mathrm{P}_BE  \arrow[r, "u"] & E.
\end{tikzcd}
\label{switchfib}
\end{equation}
Decoding the definition, we see that, geometrically, the \emph{relative path space} $\mathrm{P}_BE$ is the space whose vertices are pairs $(e, \sigma)$ with $e\in E_0$ and $\sigma$ a path in $E$ from the point $sq(e)$ to $e$, totally lying inside the fibre of $q: E\rightarrow B$ over $q(e)$; the $q^{\Delta^1}: E^{\Delta^1}\rightarrow B^{\Delta^1}$ component of $\alpha$ in the pull-back is taking care of the paths lying totally inside some fibers. The map $u(e, \sigma)=e$ is just the endpoint of such a path. We also have a map
$$\nu: B\rightarrow\mathrm{P}_BE$$
by taking ``constant paths'' at points of the image of $s: B\to E$ (rigorously, one defines it by a commutative diagram and using the universal property of the pull-back in diagram (\ref{switchfib})). One then checks that $u\nu=s$ and
\[(qu)\nu=\mathrm{id}_B, \nu(qu)\sim\mathrm{id}_{\mathrm{P}_BE},\]
hence $B$ is a retract of $\mathrm{P}_BE$ and there is a homotopy equivalence pair $qu: \mathrm{P}_BE\rightleftarrows B: \nu$; geometrically, the homotopy equivalence can be given by shrinking every path to its starting point.

\begin{definition}
	We define the \emph{relative loop space} $\Omega_BE$ via the cartesian square
	\begin{equation}
	\begin{tikzcd}
	\Omega_BE \arrow[r] \arrow[d, "\Omega q"']\arrow[dr, phantom, "\pb", near start]  & \mathrm{P}_BE \arrow[d, "u"] \\
	B \arrow[r, "s"']          & E.
	\end{tikzcd}
	\label{rellp}
	\end{equation}
\end{definition}
We obtain another fibration $\Omega q: \Omega_BE\rightarrow B$ with fibre $\Omega F$, the total space $\Omega_BE$ has points (vertices) the loops $\sigma$ in $E$ based at points in $B$, totally lying inside the fibre of $q: E\rightarrow B$; $\Omega q$ has a canonical section $s': B\rightarrow\Omega_BE$ (given by taking ``constant loops''). 

Note that if $B$ is just a point, then all these constructions specialize to the usual path fibration $\Omega E\rightarrow\mathrm{P}E\rightarrow E$.

\vspace{4mm}

Now we assume $G$ is a (discrete) group and acts on an abelian group $M$ via a group homomorphism $\rho: G\rightarrow\mathrm{Aut}(M)$. Then $G$ acts on the Eilenberg-Mac Lane space $\mathrm{K}(M, n), n\geqslant 1$. We have the \emph{twisted Eilenberg-Mac Lane space}
$$\mathrm{K}^G(M, n):=\mathrm{E}G\times_G\mathrm{K}(M, n),$$
fitting into a fibre sequence $\mathrm{K}(M, n)\rightarrow\mathrm{K}^G(M, n)\xrightarrow{q}\mathrm{B}G$ when $n\geqslant 2$, which admits a section
$$s=s_G: \mathrm{B}G\rightarrow\mathrm{K}^G(M, n)$$
induced by the inclusion of the base point of $\mathrm{K}(M, n)$ into $\mathrm{K}(M, n)$. It is easy to find that for $n\geqslant 2$,
\begin{equation*}
\pi_j(\mathrm{K}^G(M, n)) = \begin{cases}
\pi_1\mathrm{B}G=G,  & j=1; \\
\pi_j(\mathrm{K}(M, n))=M,  & j=n; \\
0,  & j\neq 1, n.
\end{cases}
\end{equation*} 
 We are thus in the situation of diagram (\ref{splitfib}):
 \begin{equation}
 \begin{tikzcd}
 \mathrm{K}(M, n) \arrow[r] & \mathrm{K}^G(M, n)  \arrow[r,shift left,"q"]
 \arrow[r,<-,shift right,swap,"s"]& \mathrm{B}G.
 \end{tikzcd}
 \label{twem}
 \end{equation}
The previous constructions thus yield a fibre sequence
\begin{equation}
\begin{tikzcd}
\Omega \mathrm{K}(M, n)\simeq\mathrm{K}(M, n-1) \arrow[r] & \mathrm{P}^G(M, n)  \arrow[r, "u"] & \mathrm{K}^G(M, n),
\end{tikzcd}
\label{switchtwem}
\end{equation}
where $\mathrm{P}^G(M, n):=\mathrm{P}_{\mathrm{B}G}\mathrm{K}^G(M, n)$ in our previous notation. Recall that there is a pair of homotopy equivalences $$qu: \mathrm{P}^G(M, n)\rightleftarrows \mathrm{B}G: \nu.$$
Thus we get a homotopy fibre sequence
\begin{equation}
\begin{tikzcd}
\mathrm{K}(M, n-1) \arrow[r] & \mathrm{B}G  \arrow[r, "s"] & \mathrm{K}^G(M, n),
\end{tikzcd}
\label{switchtwem'}
\end{equation}

The following result is a consequence of the discussions in \cite[Chapter VI, \S 5]{GJ}.
\begin{prop}\label{univemfib}
	For any $\mathrm{K}(M, n)$-fibration $p: E\rightarrow B$ with $p_*: \pi_1(E)\to\pi_1(B)\cong G$ an isomorphism, there is a unique element $[k]\in [B, \mathrm{K}^G(M, n+1)]$ such that the fibration $p: E\rightarrow B$ is equivalent to the pull-back of $u: \mathrm{P}^G(M, n+1)\rightarrow\mathrm{K}^G(M, n+1)$ along $k$. Thus we have \emph{homotopy cartesian squares}
	\begin{equation}
	\begin{tikzcd}
	E \arrow[r] \arrow[d, "p"']  & \mathrm{P}^G(M, n+1)  \arrow[d, "u"] \\
	B \arrow[r, "k"']          & \mathrm{K}^G(M, n+1)
	\end{tikzcd}
	\qquad
	\text{and}
	\qquad
	\begin{tikzcd}
	E \arrow[r] \arrow[d, "p"'] & \mathrm{B}G  \arrow[d, "s"] \\
	B \arrow[r, "k"']          & \mathrm{K}^G(M, n+1).
	\end{tikzcd}
	\label{univhopb}
	\end{equation}
\end{prop}
\begin{remark}\label{univbgfib}
	From this result, together with the existence result of Moore-Postnikov systems for maps in $\sset$, one deduces a complete Moore-Postnikov decomposition for $\sset$ as in \cite[ChapterVI, \S3]{GJ}. There is a similar result in the motivic homotopy category, see \Cref{MotMP} in the next section.
\end{remark}

\section{Tools from motivic homotopy theory}\label{motmod}

We fix a noetherian finite dimensional scheme $S$. Write $\EuScript{S}\mathsf{m}_S$ for the category of $S$-schemes which are finitely presented and smooth over $S$; it is essentially small. We make it a Grothendieck site by equipping it with the Nisnevich topology (see \cite{MV}), called the \emph{Nisnevich site} of $S$.

Let $\sps(\EuScript{S}\mathsf{m}_S)$ and $\ssh(\EuScript{S}\mathsf{m}_S)$ be respectively the category of simplicial presheaves and simplicial sheaves on this site. We use $*$ to denote the final object in $\sps(\EuScript{S}\mathsf{m}_S)$, which is represented by $S$ via Yoneda lemma. Both the sheaf and presheaf categories have the Joyal-Jardine model structure, which are Quillen equivalent; $\mathbb{A}^1$-localizing we get the Morel-Voevodsky \emph{(unstable) motivic model structure} on $\sps(\EuScript{S}\mathsf{m}_S)$ and $\ssh(\EuScript{S}\mathsf{m}_S)$. The weak equivalences in the motivic model categories are called $\mathbb{A}^1$-weak equivalences. The two model structures are again Quillen equivalent, hence lead to the same homoptoy category $\EuScript{H}^{\mathbb{A}^1}(S)$---the
\emph{$\mathbb{A}^1$-homotopy category} or the
\emph{unstable motivic homotopy category} of $S$, see \cite{MV}. The homotopy classes will be denoted by $[-, -]_{\mathbb{A}^1}$ (or even $[-, -]$).

The motivic model category has most of the nice properties one could imagine: it is a cofibrantly generated, (left and right) proper, cellular, combinatorial, simplicial model category (i.e. it is a module over the symmetric monodial simplicial model category $\sset$), etc., see \cite{MV, Horn}. In particular, the general nice results of \cite{Hir, Hov} apply here.

There is an obvious pointed version and we obtain the
\emph{pointed $\mathbb{A}^1$-homotopy category} $\EuScript{H}_{*}^{\mathbb{A}^1}(S)$ of $S$. The homotopy classes will be denoted by $[-, -]_{\mathbb{A}^1, *}$ (or even $[-, -]$). For any $X\in\sps(\EuScript{S}\mathsf{m}_S), Y\in\sps(\EuScript{S}\mathsf{m}_S)_*$, we have the canonical isomorphism $[X, Y]_{\mathbb{A}^1}\cong[X_+, Y]_{\mathbb{A}^1, *}$, where $X_+:=X\amalg *$ is $X$ with a disjoint base point added.

We fix an \emph{$\mathbb{A}^1$-fibrant replacement functor} $\mathrm{L}_{\mathbb{A}^1}$ with respect to the Morel-Voevodsky model structure on $\sps(\EuScript{S}\mathsf{m}_k, \mathrm{Nis})_*$.

If $S={\rm Spec}(k), k$ a commutative ring, we simply write $\EuScript{S}\mathsf{m}_k$ for $\EuScript{S}\mathsf{m}_S$ and write $\EuScript{H}_{(*)}^{\mathbb{A}^1}(k)$ for $\EuScript{H}_{(*)}^{\mathbb{A}^1}(S)$.

\vspace{3mm}

In fact, other than the Joyal-Jardine model structure in the first step (before performing $\mathbb{A}^1$-localization), there are other Quillen equivalent model structures---one can choose the local projective model structure for either the presheaf or sheaf category (or even some intermediate model structure, between the local projective model structure and the local injective model structure). They have equivalent homotopy categories and each has some advantages in certain situations. See e.g. \cite{DHI, Jar15}.

\vspace{3mm}

Now assume $k$ is a perfect field. We write $\grp_{k}^{\mathbb{A}^1}$ for the category of \emph{strongly $\mathbb{A}^1$-invariant sheaves of groups} and $\ab_{k}^{\mathbb{A}^1}$ for the category of \emph{strictly $\mathbb{A}^1$-invariant sheaves of abelian groups} on the Nisnevich site $(\EuScript{S}\mathsf{m}_k, \mathrm{Nis})$ (see Morel \cite[Definition 1.7]{Mor}), the latter being an abelian category. A strictly $\mathbb{A}^1$-invariant sheaf of abelian groups is a strongly $\mathbb{A}^1$-invariant sheaf of groups. Conversely, after quite a lot of effort, Morel shows that, if ${\bf A}$ is a sheaf of abelian groups on $(\EuScript{S}\mathsf{m}_k, \mathrm{Nis})$ and ${\bf A}\in\grp_{k}^{\mathbb{A}^1}$, then ${\bf A}\in\ab_{k}^{\mathbb{A}^1}$ (see \cite{Mor}). Here is a construction essentially due to Voevodsky.
\begin{definition}
	For a sheaf of abelian groups $A$, we define its \emph{contraction} $A_{-1}$ by $A_{-1}(U):=\ker(A(U\times\mathbb{G}_m)\to A(U))$ induced by the map $U=U\times\{1\}\hookrightarrow U\times\mathbb{G}_m$. And we define inductively the \emph{iterated contraction} $A_{-(n+1)}=(A_{-n})_{-1}, n\in\mathbb{N}$ (by convention, $A_{0}=A_{-0}=A$). 
\end{definition}
The contraction $A_{-1}$ is also a sheaf of abelian groups, it can be equivalently defined by $A_{-1}(U)={\rm coker}(A(U)\to A(U\times\mathbb{G}_m))$ induced by the projection $U\times\mathbb{G}_m\to U$ (since the projection has a section given by the unit of $\mathbb{G}_m$). The contraction defines an exact functor $(-)_{-1}: \ab_{k}^{\mathbb{A}^1}\to\ab_{k}^{\mathbb{A}^1}$ (\cite[Lemma 2.32]{Mor}.

See also \cite{AFHo19} for a systematic study of the group-theoretic properties of such sheaves.
\begin{definition}
	For any pointed motivic space $X\in\sps(\EuScript{S}\mathsf{m}_k)_*$, we have $\pi_{1}^{\mathbb{A}^1}(X)\in\grp_{k}^{\mathbb{A}^1}$ and $\pi_{n}^{\mathbb{A}^1}(X)\in\ab_{k}^{\mathbb{A}^1}$ for $n>1$, where for $n\in\mathbb{N}, \pi_{n}^{\mathbb{A}^1}(X)$ is its \emph{$\mathbb{A}^1$-homotopy sheaf}, i.e. the Nisnevich sheafification of the presheaf
	\[U\mapsto[S^n\wedge U_+, X]_{\mathbb{A}^1, *}.\]
	
	We say that $X\in\sps(\EuScript{S}\mathsf{m}_k)_*$ is \emph{$\mathbb{A}^1$-connected} if $\pi_{0}^{\mathbb{A}^1}(X)=0$, \emph{$\mathbb{A}^1$-simply connected} if $\pi_{0}^{\mathbb{A}^1}(X)=\pi_{1}^{\mathbb{A}^1}(X)=0$, and \emph{$\mathbb{A}^1$-$n$-connected} if $\pi_{i}^{\mathbb{A}^1}(X)=0$, for any $i$ with $0\leqslant i\leqslant n$.
\end{definition}

Results in the classical topological setting are always the guiding model for the study of abstract homotopy theory, and indicate directions where to go, give (counter)examples, etc. The following result about long exact sequence of sheaves of motivic homotopy groups associated to an $\mathbb{A}^1$-homotopy fibre sequence is of fundamental importance in motivic homotopy theory.

\begin{prop}\label{motlesfib}
	Let $k$ be a perfect field, let $F\rightarrow E\xrightarrow{q}B$ be an $\mathbb{A}^1$-homotopy fibre sequence in the motivic model category $\sps(\EuScript{S}\mathsf{m}_k)_*$. There are the \emph{boundary map} maps $\partial: \pi_n^{\mathbb{A}^1}B\rightarrow\pi_{n-1}^{\mathbb{A}^1}F$, which are homomorphism of sheaves of groups for each $n\geqslant 2$, fitting into a long exact sequence
	\[\cdots\rightarrow\pi_{n}^{\mathbb{A}^1}F\xrightarrow{i_*}\pi_{n}^{\mathbb{A}^1}E\xrightarrow{q_*}\pi_n^{\mathbb{A}^1}B\xrightarrow{\partial}\pi_{n-1}^{\mathbb{A}^1}F\rightarrow\cdots\]
	\[\cdots\rightarrow\pi_1^{\mathbb{A}^1}B\xrightarrow{\partial}\pi_{0}^{\mathbb{A}^1}F\xrightarrow{i_*}\pi_{0}^{\mathbb{A}^1}E\xrightarrow{q_*}\pi_0^{\mathbb{A}^1}B.\]
	This sequence is natural in the $\mathbb{A}^1$-fibre sequence $F\rightarrow E\xrightarrow{q}B$. There is a (left) action of $\pi_1^{\mathbb{A}^1}B$ on $\pi_{0}^{\mathbb{A}^1}F$.
\end{prop}
 
Now we are ready to state (without proof) the following result about Moore-Postnikov towers in the motivic model categories, which, as in the classical topological setting, is a fundamental tool in obstruction theory; where we use $(\EuScript{C}\downarrow B)$ to denote the category of objects over an object $B$ in a category $\EuScript{C}$ (some authors may prefer the notation $\EuScript{C}/B$) and the notion of fibre sequences is in the sense of \cite{Hov}. For an exposition of this kind of result, see \cite[Appendix B]{Mor} and \cite[Theorem 2.12]{Asok12}; see also \cite{AF14, AF15} for related discussions. We only state the result for $\sps(\EuScript{S}\mathsf{m}_k)_*$ but it also works for $\ssh(\EuScript{S}\mathsf{m}_k)_*$.
\begin{prop}[Moore-Postnikov systems in motivic model category]\label{MotMP}
	Let $k$ be a perfect field, let $F\rightarrow E\xrightarrow{q}B$ be an $\mathbb{A}^1$-fibre sequence in the motivic model category $\sps(\EuScript{S}\mathsf{m}_k)_*$, with all of $F, E, B$ being $\mathbb{A}^1$-connected and $\mathbb{A}^1$-fibrant. Denote $\pi_1^{\mathbb{A}^1}E=:\mathbf{G}\in\grp_{k}^{\mathbb{A}^1}$. Then there is the \emph{Moore-Postnikov tower}
	\[E\rightarrow\cdots\xrightarrow{q_{n+1}}E_{n}\xrightarrow{q_{n}} E_{n-1}\xrightarrow{q_{n-1}}\cdots\xrightarrow{q_2} E_{1}\xrightarrow{q_1} E_{0}\xrightarrow{p_0}B,\]
	in $\sps(\EuScript{S}\mathsf{m}_k)_*$, and maps $i_n: E\rightarrow E_{n},\ p_n: E_{n}\rightarrow B$ with the following properties:
	\begin{enumerate}[label=\emph{(\arabic*)}]
		\item All the spaces $E_n$ are $\mathbb{A}^1$-fibrant, the maps $q_n, n\geqslant 1$ are $\mathbb{A}^1$-fibrations, and $p_0: E_0\to B$ is an $\mathbb{A}^1$-weak equivalence.
		\item $i_{n-1}=q_ni_n, p_{n-1}q_n=p_n, \forall n\geqslant 1$.
		\item $p_ni_n=q, \forall n\in\mathbb{N}$.
		\item The $\mathbb{A}^1$-homotopy fibre $F(q_{n})$ of $q_{n}\ (n\geqslant 1)$ is $\mathbb{A}^1$-weakly equivalent to $\mathrm{K}(\pi_n^{\mathbb{A}^1}F, n)$, hence for each $n\geqslant 1$ we have an $\mathbb{A}^1$-homotopy fibre sequence
		\[\mathrm{K}(\pi_n^{\mathbb{A}^1}F, n)\rightarrow E_n\xrightarrow{q_{n}} E_{n-1}.\]
		\item There are homotopy pullback diagrams in $(\sps(\EuScript{S}\mathsf{m}_k)\downarrow\mathrm{B}\mathbf{G})$ (with model structure induced from the Morel-Voevodsky motivic model structure on $\sps(\EuScript{S}\mathsf{m}_k)$)
		\[\xymatrix{
			E_n  \ar[r]  \ar[d]_-{q_n}  & \mathrm{B}\mathbf{G} \ar[d]^-{s_n}\\
			E_{n-1}  \ar[r]^-{k_{n+1}} &  \mathrm{K}^{\mathbf{G}}(\pi_{n}^{\mathbb{A}^1}F, n+1),}\]
		for a unique $[k_{n+1}]\in[E_{n-1}, \mathrm{K}^{\mathbf{G}}(\pi_{n}^{\mathbb{A}^1}F, n+1)]_{\mathbb{A}^1}$, for all $n\geqslant 2$.
		\item The map $E\rightarrow{\mathrm{holim}}_{n\in\mathbb{N}^{\mathrm{op}}}E_{n}$ is a local weak equivalence hence an $\mathbb{A}^1$-weak equivalence.
	\end{enumerate}
\end{prop}

Specializing to the case $E=*$ or $B=*$, one gets respectively the \emph{Whitehead tower} and \emph{Postnikov tower}. Moreover, by the above existence results, we obtain the following.
\begin{corollary}\label{seqmotpt}
	Assumption and notations as in \Cref{MotMP}, we have:
	\begin{enumerate}[label=\emph{(\arabic*)}]
		\item For any $j\leqslant n$, $(i_n)_*: \pi_j^{\mathbb{A}^1}E\xrightarrow{\cong}\pi_j^{\mathbb{A}^1}E_{n}$.
		\item For any $j\geqslant n+2$, $(p_n)_*: \pi_j^{\mathbb{A}^1}E_{n}\xrightarrow{\cong}\pi_j^{\mathbb{A}^1}B$.
		\item There are exact sequences
		\[0\rightarrow\pi_{n+1}^{\mathbb{A}^1}E_{n}\xrightarrow{(p_n)_*}\pi_{n+1}^{\mathbb{A}^1}B\xrightarrow{\partial}\pi_{n}^{\mathbb{A}^1}F,\ n\geqslant 0,\]
		where $\partial$ is the connecting homomorphism in the homotopy long exact sequence of the homotopy fibre sequence $F\rightarrow E\xrightarrow{q}B$.
		\item There are $\mathbb{A}^1$-homotopy fibre sequences
		\[F[n]\rightarrow E_n\xrightarrow{p_{n}} B,\]
		\[F(q_{n})\rightarrow F[n]\xrightarrow{q_{n}'}F[n-1]\ (n\geqslant 1),\]
		where the map $q_{n}'$ is induced from $q_{n}: E_n\rightarrow E_{n-1}$ and $F[n]$ is the $n$-th stage of the Postnikov tower for $F$.
	\end{enumerate}
\end{corollary}

\begin{remark}
	In practice, we usually consider the liftings of the homotopy class of a map from a smooth scheme $X$, regarded as a simplicial presheaf, to a motivic space $B$, to elements in $[X, E]_{\mathbb{A}^1}\cong[X_+, E]_{\mathbb{A}^1, *}$. In this situation, the above results on Moore-Postnikov systems is very useful in finding when a given $\mathbb{A}^1$-homotopy class can be lifted, especially when the $\mathbb{A}^1$-homotopy fibre $F$ of $E\xrightarrow{q}B$ is highly connected.
\end{remark}

We also need the following results about (relative) cohomologies. We fix a \emph{strictly $\mathbb{A}^1$-invariant} sheaf of abelian groups $\mathbf{M}$ on the Nisnevich site $(\EuScript{S}\mathsf{m}_k, \mathrm{Nis})$ (see \cite{Mor}). Let $Y\in\sps(\EuScript{S}\mathsf{m}_k)_*$ be a pointed simplicial presheaf, we define its \emph{reduced cohomology with coefficients in $\mathbf{M}$} to be
\[
\widetilde{{\rm H}}^n(Y; \mathbf{M}):=[Y, \mathrm{K}(\mathbf{M}, n)]_{\mathbb{A}^1, *}.
\]
So $\widetilde{{\rm H}}^n(\Sigma Y; \mathbf{M})=\widetilde{{\rm H}}^{n-1}(Y; \mathbf{M})$ and for $X\in\sps(\EuScript{S}\mathsf{m}_k)$, we have ${\rm H}^n(X; \mathbf{M})=\widetilde{{\rm H}}^n(X_+; \mathbf{M})$.

If $X\subset Z$, we define ${\rm H}^n(Z, X; \mathbf{M})$, the \emph{cohomology of the pair $(Z, X)$ with coefficients in $\mathbf{M}$}, to be the reduced cohomology of $Z/X$ (note that it is naturally pointed), the (homotopy) cofiber of the inclusion $X\hookrightarrow Z$:
\[{\rm H}^n(Z, X; \mathbf{M}):=\widetilde{{\rm H}}^n(Z/X; \mathbf{M})=[Z/X, \mathrm{K}(\mathbf{M}, n)]_{\mathbb{A}^1, *}.\]
So ${\rm H}^n(X, \varnothing; \mathbf{M})=\widetilde{{\rm H}}^n(X_+; \mathbf{M})={\rm H}^n(X; \mathbf{M})$ and $\widetilde{{\rm H}}^n(Y; \mathbf{M})={\rm H}^n(Y, *; \mathbf{M})$. All these are abelian groups.
\begin{remark}
	Here we are taking homotopy classes in the $\mathbb{A}^1$-homotopy category. But since we are working with strictly $\mathbb{A}^1$-invariant sheaves of abelian groups (which is the right coefficient sheaves for $\mathbb{A}^1$-homotopy theory), the homotopy classes equal those in Jardine's local model structure (no $\mathbb{A}^1$-localization involved). The latter works for arbitrary sheaves of abelian groups.
\end{remark}
The following result, familiar from classical topology, follows easily from the general abstract formalism of \cite[Chapter 6]{Hov}. The long exact sequence for cohomology of a pair follows from that for the triple $\varnothing\subset Y\subset Z$.
\begin{prop}[Long exact sequence for cohomology of a triple]\label{lesrel}
	Given simplicial presheaves $X\subset Y\subset Z$, we have the long exact sequence of cohomology groups
	\[
	0\to{\rm H}^0(Z, Y; \mathbf{M})\to{\rm H}^0(Z, X; \mathbf{M})\to{\rm H}^0(Y, X; \mathbf{M})\to{\rm H}^1(Z, Y; \mathbf{M})\to\cdots
	\]
	\[
	\cdots\to{\rm H}^n(Z, Y; \mathbf{M})\to{\rm H}^n(Z, X; \mathbf{M})\to{\rm H}^n(Y, X; \mathbf{M})\to{\rm H}^{n+1}(Z, Y; \mathbf{M})\to\cdots,
	\]
	natural in the triple $X\subset Y\subset Z$.
	
	In particular, there is the \emph{long exact sequence for cohomology of a pair} $(Z, Y)$:
	\[
	0\to{\rm H}^0(Z, Y; \mathbf{M})\to{\rm H}^0(Z; \mathbf{M})\to{\rm H}^0(Y; \mathbf{M})\to{\rm H}^1(Z, Y; \mathbf{M})\to\cdots
	\]
	\[
	\cdots\to{\rm H}^n(Z, Y; \mathbf{M})\to{\rm H}^n(Z; \mathbf{M})\to{\rm H}^n(Y; \mathbf{M})\to{\rm H}^{n+1}(Z, Y; \mathbf{M})\to\cdots.
	\]
\end{prop}
For a pointed simplicial presheaf $Y\in\sps(\EuScript{S}\mathsf{m}_k)_*$ and $n\geqslant 1$, we have the \emph{suspension homomorphism} \[\sigma: \widetilde{{\rm H}}^n(Y; \mathbf{M})=[Y, \mathrm{K}(\mathbf{M}, n)]_{\mathbb{A}^1, *}\to[\Omega Y, \Omega\mathrm{K}(\mathbf{M}, n)]_{\mathbb{A}^1, *}=\widetilde{{\rm H}}^{n-1}(\Omega Y; \mathbf{M})\]
induced by the (derived) loop functor $\Omega$; it is induced by the counit map $\Sigma\Omega Y\to Y$ (as $\widetilde{{\rm H}}^n(\Sigma\Omega Y; \mathbf{M})=\widetilde{{\rm H}}^{n-1}(\Omega Y; \mathbf{M})$). So
\[\sigma([\theta])=[\Omega\theta].\]

Recall also that in the category $\sps(\EuScript{S}\mathsf{m}_k)_*$ of \emph{pointed} simplicial presheaves on $\EuScript{S}\mathsf{m}_k$, we have the \emph{smash product} $\wedge$. The smash product $(X, x)\wedge (Y, y)$ is defined as the sectionwise smash product of simplicial sets $U\mapsto (X, x)(U)\wedge (Y, y)(U)$---smash product $(X, x)\wedge (Y, y)$ corepresents maps from $X\times Y$ that are base-point-preserving separately in each variable. So for $X, Y\in\sps(\EuScript{S}\mathsf{m}_k)$, we have $(X\times Y)_+=X_+\wedge Y_+$, and for any $K\in\sps(\EuScript{S}\mathsf{m}_k)$ we have $(X, x)\wedge K_+\cong (X\times K)/(x\times K)$ as pointed simplicial presheaves.

On the other hand, there is the \emph{pointed internal function complex} $\underline{\mathrm{Hom}}_*((X, x), (Y, y))\in\sps(\EuScript{S}\mathsf{m}_k)_*$, given by the equalizer diagram
\begin{equation}
\begin{tikzcd}
\underline{\mathrm{Hom}}_*((X, x), (Y, y))\ar[r] & \underline{\mathrm{Hom}}(X, Y)\ar[r,shift left=.4ex,"x^*"]
\ar[r,shift right=.4ex,swap,"y_*"] &  \underline{\mathrm{Hom}}(*, Y)=Y,
\end{tikzcd}
\end{equation}
where $\underline{\mathrm{Hom}}$ is the internal hom of simplicial presheaves, with respect to the cartesian closed structure on $\sps(\EuScript{S}\mathsf{m}_k)$ and $y_*$ is the composite $\underline{\mathrm{Hom}}(X, Y)\to*\xrightarrow{y}\underline{\mathrm{Hom}}(*, Y)=Y$. We have
\[\underline{\mathrm{Hom}}_*(A_+, (Y, y))\cong\underline{\mathrm{Hom}}(A, Y)\]
as (pointed) simplicial presheaves, for $A\in\sps(\EuScript{S}\mathsf{m}_k), (Y, y)\in\sps(\EuScript{S}\mathsf{m}_k)_*$. There is the following adjunction
\[(X, x)\wedge(-): \sps(\EuScript{S}\mathsf{m}_k)_*\rightleftarrows\sps(\EuScript{S}\mathsf{m}_k)_*: \underline{\mathrm{Hom}}_*((X, x), -).\]
It passes to a Quillen adjunction for the motivic model structure.

The operations $\wedge, \underline{\mathrm{Hom}}_*$ make $\sps(\EuScript{S}\mathsf{m}_k)_*$ a \emph{closed symmetric monoidal category} with unit $S^0=\partial\Delta^1=\Delta^0_+$ (viewed as a constant pointed simplicial presheaf). The \emph{simplicial suspension} is given by $\Sigma X:=S^1\wedge X$ and the \emph{simplicial looping} is given by $\Omega X:=\underline{\mathrm{Hom}}_*(S^1, X)$, where $S^1:=\Delta^1/\partial\Delta^1$ (viewed as a constant pointed simplicial presheaf).

\vspace{5mm}

Now let $X, Y, Z\in\sps(\EuScript{S}\mathsf{m}_k)_*$ be pointed simplicial presheaves, let
\[u: (X\times Y, *\times Y)\to(Z, *)\]
be a map of pairs, which can be identified with a pointed map $u: X\wedge Y_+\to Z$ (whenever we write $Y_+$, we forget the base point of $Y$ and add a new base point). Then we have the maps $Y_+\to\underline{\mathrm{Hom}}_*(X, Z)\to\underline{\mathrm{Hom}}_*(\Omega X, \Omega Z)$, which induces a map $\Omega X\wedge Y_+\to \Omega Z$ by adjunction, hence a map of pairs
\[v: (\Omega X\times Y, *\times Y)\to(\Omega Z, *).\]

Since $*\times Y$ is a deformation retract of ${\rm P}X\times Y$, by the long exact sequence for cohomology of the pair $({\rm P}X\times Y, *\times Y)$, we see that ${\rm H}^*({\rm P}X\times Y, *\times Y; \mathbf{M})=0$. So the connecting homomorphism 
\[\delta: {\rm H}^q(\Omega X\times Y, *\times Y; \mathbf{M})\to{\rm H}^{q+1}({\rm P}X\times Y, \Omega X\times Y; \mathbf{M})\]
for the triple $({\rm P}X\times Y, \Omega X\times Y, *\times Y)$ is an isomorphism (by the long exact sequence for cohomology of a triple above).

The given map $u$ also induces a map of triples
\[U: ({\rm P}X\times Y, \Omega X\times Y, *\times Y)\to({\rm P}Z, \Omega Z, *)\]
which fits into a commutative diagram
\[\xymatrix{
	({\rm P}X\times Y, \Omega X\times Y, *\times Y)  \ar[r]^-{U}  \ar[d]_-{p_X\times{\rm id}_Y}  & ({\rm P}Z, \Omega Z, *) \ar[d]^-{p_Z}\\
	(X\times Y, *\times Y, *\times Y)  \ar[r]^-{u} &  (Z, *, *),}\]
where $p=p_X: {\rm P}X\to X$ denotes the ``path fibration''.
Naturality of the connecting homomorphism for $U$ gives $\delta v^*=U^*\delta$, or $v^*\delta^{-1}=\delta^{-1}U^*$. So $v^*\delta^{-1}p^*=\delta^{-1}U^*p^*=\delta^{-1}(p\times{\rm id}_Y)^*u^*$. Define
\[\rho: {\rm H}^{q+1}(X\times Y, *\times Y; \mathbf{M})\to{\rm H}^q(\Omega X\times Y, *\times Y; \mathbf{M})\]
to be the map
\[[X\wedge Y_+, \mathrm{K}(\mathbf{M}, q+1)]_{\mathbb{A}^1, *}=[Y_+, \underline{\mathrm{Hom}}_*(X, \mathrm{K}(\mathbf{M}, q+1))]_{\mathbb{A}^1, *}\]
\[\to[Y_+, \underline{\mathrm{Hom}}_*(\Omega X, \Omega\mathrm{K}(\mathbf{M}, q+1))]_{\mathbb{A}^1, *}
=[\Omega X\wedge Y_+, \mathrm{K}(\mathbf{M}, q)]_{\mathbb{A}^1, *}\]
induced by the simplicial looping. Note that $X\times Y/*\times Y=X\wedge Y_+$ and ${\rm H}^{q+1}((X\times Y)/(*\times Y); \mathbf{M})=[(X\times Y)/(*\times Y), \mathrm{K}(\mathbf{M}, q+1)]_{\mathbb{A}^1, *}=[X\wedge Y_+, \mathrm{K}(\mathbf{M}, q+1)]_{\mathbb{A}^1, *}$.

Now let $\iota_1: (X\times Y, \varnothing)\to(X\times Y, *\times Y)$ and $\iota_2: (\Omega X\times Y, \varnothing)\to(\Omega X\times Y, *\times Y)$ be the inclusions of pairs; smashing the cofiber sequence $S^0\to X_+\to X$ with $Y_+$ we get a cofiber sequence $Y_+=S^0\wedge Y_+\to X_+\wedge Y_+\xrightarrow{\iota_1} X\wedge Y_+$ and hence also a cofiber sequence $X_+\wedge Y_+\xrightarrow{\iota_1} X\wedge Y_+\to S^1\wedge Y_+=\Sigma (Y_+)$; we obtain the induced map
\[\iota_1^*: {\rm H}^{q+1}(X\times Y, *\times Y; \mathbf{M})\to{\rm H}^{q+1}(X\times Y; \mathbf{M}).\]
Clearly, the inclusion $*\times Y\hookrightarrow X\times Y$ splits (by the projection), and $\iota_1^*$ is an injection (by \Cref{lesrel}).

Similarly, we have the induced map
\[\iota_2^*: {\rm H}^{q}(\Omega X\times Y, *\times Y; \mathbf{M})\to{\rm H}^{q}(\Omega X\times Y; \mathbf{M}),\]
which is an injection as well.

We obtain the following commutative diagram:
\begin{equation}\label{relcoholoop}
\begin{tikzcd}
\widetilde{{\rm H}}^{q+1}(Z; \mathbf{M}) \arrow[r, "u^*"] \arrow[d, "\sigma"'] & {\rm H}^{q+1}(X\times Y, *\times Y; \mathbf{M})  \arrow[d, "\rho"] \arrow[r, "\iota_1^*"] & {\rm im}(\iota_1^*) \arrow[d, "\sigma'"] \ar[r, hook] & {\rm H}^{q+1}(X\times Y; \mathbf{M}) \\
\widetilde{{\rm H}}^q(\Omega Z; \mathbf{M}) \arrow[r, "v^*"]          & {\rm H}^q(\Omega X\times Y, *\times Y; \mathbf{M}) \arrow[r, "\iota_2^*"] &  {\rm im}(\iota_2^*)\ar[r, hook] & {\rm H}^{q}(\Omega X\times Y; \mathbf{M}),
\end{tikzcd}
\end{equation}
where the \emph{partial suspension homomorphism} $\sigma': {\rm im}(\iota_1^*)\to{\rm im}(\iota_2^*)$ is given by $\iota_1^*(x)\mapsto\iota_2^*\rho(x)$ (it's well-defined, since $\iota_1^*$ is an injective); commutativity of the square on the left follows from naturality of our construction.

Assume $\mathbf{K}=\displaystyle\bigoplus_{i\in\mathbb{Z}}\mathbf{K}_i$ is a graded sheaf of rings, each $\mathbf{K}_i$ being a strictly $\mathbb{A}^1$-invariant sheaf of abelian groups. Then we have the following commutative diagram:
\begin{equation}\label{sus-ext-prod}
\begin{tikzcd}
\widetilde{{\rm H}}^{r+1}(X; \mathbf{K}_i)\otimes{\rm H}^{s}(Y; \mathbf{K}_j) \arrow[r, "\mu"] \arrow[d, "\sigma\otimes 1"'] & {\rm H}^{r+1+s}(X\times Y, *\times Y; \mathbf{K}_{i+j})  \arrow[d, "\rho"] \\
\widetilde{{\rm H}}^r(\Omega X; \mathbf{K}_i)\otimes{\rm H}^{s}(Y; \mathbf{K}_j) \arrow[r, "\mu"']          & {\rm H}^{r+s}(\Omega X\times Y, *\times Y; \mathbf{K}_{i+j}).
\end{tikzcd}
\end{equation}
Here the top arrow $\mu$ is the \emph{external product}
\[[X, \mathrm{K}(\mathbf{K}_i, r+1)]_{\mathbb{A}^1, *}\otimes[Y, \mathrm{K}(\mathbf{K}_j, s)]_{\mathbb{A}^1}\to[X\times Y/*\times Y, \mathrm{K}(\mathbf{K}_{i+j}, r+1+s)]_{\mathbb{A}^1, *}\]
induced by the multiplication in the graded sheaves of rings $\mathbf{K}$ (for more details, see \cite[\S 8.4]{Jar15}, where it is called \emph{cup product pairing}; Jardine's assumption of commutativity of the coefficient sheaf of rings $\mathbf{K}$ is not essential). We denote $a\times b:=\mu(a\otimes b)$. Then by definition (when $X=Y$), the \emph{cup product} is given by $a\cdot b=a\smile b:=\Delta^*\mu(a\otimes b)=\Delta^*(a\times b)$, where $\Delta: X\to X\times X$ is the diagonal. With cup product, the cohomology ${\rm H}^{*}(X; \mathbf{K})$ becomes a ring (bigraded if $\mathbf{K}$ is a graded sheaf of rings), which may not be anti-commutative (depending on the commutativity property of $\mathbf{K}$). The commutativity of the diagram of diagram (\ref{sus-ext-prod}) means that $\rho(a\times b)=(\sigma a)\times b$ for $a\in\widetilde{{\rm H}}^{*}(X; \mathbf{K}), b\in{\rm H}^{*}(Y; \mathbf{K})$ ($\sigma$ is the suspension homomorphism for $X$), which follows from our alternative definition of $\rho$ (and in fact, we may totally discard the first definition of $\rho$; we give it only to respect the work of James-Thomas \cite{JTh65}).

Similarly, we have $\sigma'(a\times b)=(\sigma a)\times b$, for $a\in\widetilde{{\rm H}}^{*}(X; \mathbf{K}), b\in\widetilde{{\rm H}}^{*}(Y; \mathbf{K})$ in diagram (\ref{relcoholoop}) (here $\sigma$ is the suspension homomorphism for $X$).

\section{$\mathbb{A}^1$-homotopy study of vector bundles}\label{hogpfibs}

We first state the following affine representability results in $\mathbb{A}^1$-homotopy theory, which is (a special case of) \cite[Theorem 5.2.3]{AHW1}, where we use $\EuScript{V}_r(U)$ to denote the set of isomorphism classes of rank $r$ vector bundles over a scheme $U$ and $\EuScript{S}\mathsf{m}_k^{\operatorname{aff}}$ denotes the category of smooth affine $k$-schemes. By $\operatorname{Gr}_{r}$, we mean the ind-scheme of all finite Grassmannians $\operatorname{Gr}_{r, n}, n>r$; $\mathrm{BGL}_{r}$ is the simplicial classifying space for the linear $k$-group scheme $\mathrm{GL}_{r}$.

For $U\in\EuScript{S}\mathsf{m}_k^{\operatorname{aff}}$, there are canonical maps $\operatorname{Hom}(U, \operatorname{Gr}_{r})\to\EuScript{V}_r(U)$ given by pulling back the tautological rank-$r$ vector bundle and $\operatorname{Hom}(U, \operatorname{Gr}_{r})\to[U, \operatorname{Gr}_{r}]_{\mathbb{A}^1}$ (mapping a morphism to its $\mathbb{A}^1$-homotopy class), both are surjective.
\begin{prop}[Affine representability for vector bundles]
	Let $k$ be a perfect field, then there is a natural isomorphism of functors
	\[\EuScript{V}_r(-)\cong[-, \operatorname{Gr}_{r}]_{\mathbb{A}^1}: (\EuScript{S}\mathsf{m}_k^{\operatorname{aff}})^{\mathrm{op}}\to\set\]
	such that for every $U\in\EuScript{S}\mathsf{m}_k^{\operatorname{aff}}$, the following commutative diagram
	\[\xymatrix{
		\operatorname{Hom}(U, \operatorname{Gr}_{r}) \ar[rd]  \ar[d] \\
		\EuScript{V}_r(U)   \ar[r]^{\cong\ \ \ }  &  [U, \operatorname{Gr}_{r}]_{\mathbb{A}^1},}\]
	where the other two arrows are the canonical ones.
\end{prop}
As there is an $\mathbb{A}^1$-weak equivalence $\mathrm{BGL}_{r}\simeq\operatorname{Gr}_{r}$, we can use $\mathrm{BGL}_{r}$ in place of $\operatorname{Gr}_{r}$.

We have the following $\mathbb{A}^1$-homotopy fiber sequence (see \cite[\S 8.2]{Mor}):
\[\mathbb{A}^{n+1}\setminus{0}\to\mathrm{BGL}_{n}\to\mathrm{BGL}_{n+1}.\]
From this and the computation of the first few $\mathbb{A}^1$-homotopy sheaves of the motivic sphere $\mathbb{A}^{n+1}\setminus 0$ (see \cite[Corollary 6.43]{Mor})
\[\pi_i^{\mathbb{A}^1}(\mathbb{A}^{n+1}\setminus 0)=\begin{cases}
0,  & i<n, \\
\mathbf{K}^{\mathrm{MW}}_{n+1},  & i=n,
\end{cases}\]
we see that the induced map of $\mathbb{A}^1$-homotopy sheaves \[\pi_j^{\mathbb{A}^1}\mathrm{BGL}_{n}\rightarrow\pi_j^{\mathbb{A}^1}\mathrm{BGL}_{n+1}\]
is an isomorphism if $j<n$ and a surjection if $j=n$.

Thus also,  the induced map \[\pi_j^{\mathbb{A}^1}\mathrm{BGL}_{n}\rightarrow\pi_j^{\mathbb{A}^1}\mathrm{BGL}\]
is an isomorphism if $j<n$ and an epimorphism if $j=n$.

Let $F_n$ be the $\mathbb{A}^1$-homotopy fiber of the canonical map $\varphi_n: \mathrm{BGL}_{n}\rightarrow\mathrm{BGL}$, so we have an $\mathbb{A}^1$-homotopy fiber sequence
\[F_n\to\mathrm{BGL}_{n}\to\mathrm{BGL},\]
and from the above results, we see that 
\[\pi_j^{\mathbb{A}^1}F_n=0, j<n.\]

We now compute the next $\mathbb{A}^1$-homotopy sheaf of $F_n$. By \cite[\S 6.1-\S 6.5]{Hov}, the following diagram
\[\xymatrix{
	\mathbb{A}^{n+1}\setminus{0} \ar[r]^-{i} \ar@{=}[d] & F_{n}  \ar[r] \ar[d] & F_{n+1} \ar[d]\\
	\mathbb{A}^{n+1}\setminus{0} \ar[r] & \mathrm{BGL}_{n} \ar[r] \ar[d] & \mathrm{BGL}_{n+1} \ar[d]\\
	& \mathrm{BGL} \ar@{=}[r] & \mathrm{BGL}}\]
commutes, where the four 3-term rows and columns are $\mathbb{A}^1$-homotopy fiber sequences. From the $\mathbb{A}^1$-homotopy fiber sequence in the first row we get an exact sequence \[\pi_{n+1}^{\mathbb{A}^1}F_{n+1}\rightarrow\pi_n^{\mathbb{A}^1}(\mathbb{A}^{n+1}\setminus 0)\to\pi_n^{\mathbb{A}^1}F_n\to\pi_n^{\mathbb{A}^1}F_{n+1}=0,\]
so the map
\[\mathbf{K}^{\mathrm{MW}}_{n+1}=\pi_n^{\mathbb{A}^1}(\mathbb{A}^{n+1}\setminus 0)\to\pi_n^{\mathbb{A}^1}F_n\]
and hence also
\[\pi_{n+1}^{\mathbb{A}^1}(\mathbb{A}^{n+2}\setminus 0)\to\pi_{n+1}^{\mathbb{A}^1}F_{n+1}\]
are epimorphisms (see \cite[\S 3.2]{Mor} for a thorough treatment of the sheaves $\mathbf{K}^{\mathrm{MW}}_*$ and $\mathbf{K}^{\mathrm{M}}_*$). And so
\[\pi_n^{\mathbb{A}^1}F_n\cong{\rm coker}(\pi_{n+1}^{\mathbb{A}^1}F_{n+1}\rightarrow\pi_n^{\mathbb{A}^1}(\mathbb{A}^{n+1}\setminus 0))\cong{\rm coker}(\pi_{n+1}^{\mathbb{A}^1}(\mathbb{A}^{n+2}\setminus 0)\rightarrow\pi_n^{\mathbb{A}^1}(\mathbb{A}^{n+1}\setminus 0)).\]

We claim that the composite $\pi_{n+1}^{\mathbb{A}^1}(\mathbb{A}^{n+2}\setminus 0)\to\pi_{n+1}^{\mathbb{A}^1}F_{n+1}\to\pi_n^{\mathbb{A}^1}(\mathbb{A}^{n+1}\setminus 0)$ is the map
\[\mathbf{K}^{\mathrm{MW}}_{n+2}\to\mathbf{K}^{\mathrm{MW}}_{n+1}\]
discussed in \cite[Lemma 3.5]{AF14a}. To see this, using the naturality of the connecting homomorphism of the upper right square of the previous diagram and its upper left square with $n$ replaced by $n+1$, we obtain a commutative diagram
\[
\xymatrix{
	\pi_{n+1}^{\mathbb{A}^1}(\mathbb{A}^{n+2}\setminus 0) \ar[r]^-{i_*} \ar@{=}[d] & \pi_{n+1}^{\mathbb{A}^1}(F_{n+1})  \ar[r] \ar[d] & \pi_n^{\mathbb{A}^1}(\mathbb{A}^{n+1}\setminus 0) \ar@{=}[d]\\
	\pi_{n+1}^{\mathbb{A}^1}(\mathbb{A}^{n+2}\setminus 0) \ar[r]^-{\delta_{n+1}} & \pi_n^{\mathbb{A}^1}({\rm SL}_{n+1}) \ar[r]^-{q_n}  & \pi_n^{\mathbb{A}^1}(\mathbb{A}^{n+1}\setminus 0),}
\]
where $q_n$ and $\delta_{n+1}$ are the maps in \cite[Lemma 3.5]{AF14a}, proving the claim. Thus that composite map is $0$ if $n$ is even and is multiplication by $\eta$ if $n$ is odd; so
\begin{equation*}
{\rm im}(
\pi_{n+1}^{\mathbb{A}^1}F_{n+1}\rightarrow\pi_n^{\mathbb{A}^1}(\mathbb{A}^{n+1}\setminus 0)) \cong 
\begin{cases}
0,  & n\ \text{even}; \\
\eta\mathbf{K}^{\mathrm{MW}}_{n+2},  & n\ \text{odd}.
\end{cases}
\end{equation*}
Thus
\begin{equation*}
\pi_n^{\mathbb{A}^1}F_n\cong{\rm coker}(\mathbf{K}^{\mathrm{MW}}_{n+2}\to\mathbf{K}^{\mathrm{MW}}_{n+1}) \cong \begin{cases}
\mathbf{K}^{\mathrm{MW}}_{n+1},  & n\ \text{even}; \\
\mathbf{K}^{\mathrm{MW}}_{n+1}/\eta\mathbf{K}^{\mathrm{MW}}_{n+2}=\mathbf{K}^{\mathrm{M}}_{n+1},  & n\ \text{odd}.
\end{cases}
\end{equation*}
Moreover, by the $\mathbb{A}^1$-homotopy fiber sequence $\mathbb{A}^{n+1}\setminus{0}\to F_{n}\to F_{n+1}$ we get exact sequences \[\pi_{n+1}^{\mathbb{A}^1}(\mathbb{A}^{n+1}\setminus 0)\to \pi_{n+1}^{\mathbb{A}^1}F_n\to\mathbf{K}^{\mathrm{M}}_{n+2}\xrightarrow{0}\pi_n^{\mathbb{A}^1}(\mathbb{A}^{n+1}\setminus 0), n\ \text{even},\]
\[\pi_{n+1}^{\mathbb{A}^1}(\mathbb{A}^{n+1}\setminus 0)\to \pi_{n+1}^{\mathbb{A}^1}F_n\to\mathbf{K}^{\mathrm{MW}}_{n+2}\xrightarrow{\eta}\mathbf{K}^{\mathrm{MW}}_{n+1}, n\ \text{odd},\]
and since $\ker(\mathbf{K}^{\mathrm{MW}}_{n+2}\xrightarrow{\eta}\mathbf{K}^{\mathrm{MW}}_{n+1})\cong2\mathbf{K}^{\mathrm{M}}_{n+2}$ (see e.g. \cite[\S 2.4]{AF14a} especially \cite[Proposition 2.6]{AF14a} for a discussion; it relies on Voevodsky's confirmation of the Milnor conjecture \cite{OVV}), we get exact sequences
\begin{equation}
\begin{cases}
\pi_{n+1}^{\mathbb{A}^1}(\mathbb{A}^{n+1}\setminus 0)\to \pi_{n+1}^{\mathbb{A}^1}F_n\to\mathbf{K}^{\mathrm{M}}_{n+2}\to0,  & n\ \text{even}; \\
\pi_{n+1}^{\mathbb{A}^1}(\mathbb{A}^{n+1}\setminus 0)\to \pi_{n+1}^{\mathbb{A}^1}F_n\to2\mathbf{K}^{\mathrm{M}}_{n+2}\to0,  & n\ \text{odd}.
\end{cases}
\end{equation}

As $\pi_j^{\mathbb{A}^1}F_n=0$ for $j<n$, by the Moore-Postnikov decomposition in $\mathbb{A}^1$-homotopy theory stated in \Cref{MotMP}, we can factorize the map $\mathrm{BGL}_{n}\rightarrow\mathrm{BGL}$ as $\mathrm{BGL}_{n}\to E=E_n\to\mathrm{BGL}$ and since $\pi_1^{\mathbb{A}^1}\mathrm{BGL}=\mathbb{G}_m$, the map $E=E_n\to\mathrm{BGL}$ fits into a homotopy cartesian diagram
\[\xymatrix{
	E=E_n  \ar[r]  \ar[d]  & \mathrm{B}\mathbb{G}_m \ar[d]\\
	\mathrm{BGL}=E_{n-1}  \ar[r]^-{k_{n+1}} &  \mathrm{K}^{\mathbb{G}_m}(\pi_{n}^{\mathbb{A}^1}F_n, n+1),}\]
for a unique $[k_{n+1}]\in[E_{n-1}, \mathrm{K}^{\mathbb{G}_m}(\pi_{n}^{\mathbb{A}^1}F_n, n+1)]_{\mathbb{A}^1}$, if $n\geqslant 2$. For any $j\leqslant n$, we have $\pi_j^{\mathbb{A}^1}\mathrm{BGL}_{n}\cong\pi_j^{\mathbb{A}^1}E_{n}$.

If $n\geqslant 3$ is odd, then $\mathbb{G}_m$ acts trivially on $\pi_{n}^{\mathbb{A}^1}F_n=\mathbf{K}^{\mathrm{M}}_{n+1}$. Indeed, by the paragraph in \cite[pp. 1056-1057]{AF15}, this action is through the morphism 
\[\mathbb{G}_m\to(\mathbf{K}^{\mathrm{MW}}_0)^{\times}, u\mapsto\langle u\rangle: =1+\eta[u]\]
then the multiplication $\mathbf{K}^{\mathrm{MW}}_0\times\mathbf{K}^{\mathrm{MW}}_{n+1}\to\mathbf{K}^{\mathrm{MW}}_{n+1}$ and the quotient $\mathbf{K}^{\mathrm{MW}}_{n+1}\to\mathbf{K}^{\mathrm{MW}}_{n+1}/\eta\mathbf{K}^{\mathrm{MW}}_{n+2}=\mathbf{K}^{\mathrm{M}}_{n+1}$; but $\eta\mathbf{K}^{\mathrm{MW}}_{n+2}$ is mapped to $0$ in $\mathbf{K}^{\mathrm{M}}_{n+1}$.

So the above homotopy cartesian diagram reduces to an $\mathbb{A}^1$-homotopy fiber sequence
\begin{equation}
E=E_n\to\mathrm{BGL}=E_{n-1}\xrightarrow{\theta}\mathrm{K}(\mathbf{K}^{\mathrm{M}}_{n+1}, n+1).
\end{equation}
whence a principal $\mathbb{A}^1$-homotopy fiber sequence
\[\mathrm{K}(\mathbf{K}^{\mathrm{M}}_{n+1}, n)\to E=E_n\to\mathrm{BGL}=E_{n-1}.\]
The $\mathbb{A}^1$-homotopy class of the map $\theta$ is the universal $(n+1)$-st Chern class $c_{n+1}\in{\rm H}^{n+1}(\mathrm{BGL}; \mathbf{K}^{\mathrm{M}}_{n+1})={\rm CH}^{n+1}(\mathrm{BGL})$ (see \cite[Example 5.2 and Proposition 5.8]{AF16a}).

The following result is easily proven by the $\mathbb{A}^1$-homotopy long exact sequence and crawling up the Moore-Postnikov tower of the map $\mathrm{BGL}_{n}\rightarrow\mathrm{BGL}$ (with essentially the same proof as \cite[Proposition 6.2]{AF14}).
\begin{prop}\label{towiso}
	Let $k$ be a perfect field, $A$ a smooth affine $k$-algebra of Krull dimension $d\geqslant 3$, and $X={\rm Spec}(A)$. Then the map
	\[[X, \mathrm{BGL}_{n}]_{\mathbb{A}^1}\to[X, E]_{\mathbb{A}^1}\]
	is surjective if $n\geqslant d-1$, and is bijective if $n\geqslant d$.
\end{prop}

\section{Motivic approach to enumerating non-stable vector bundles}

In this section, we develop a motivic homotopy theoretic approach to the enumeration problem for non-stable vector bundles, following the ideas of I. M. James and E. Thomas \cite{JTh65} in the classical homotopy theoretic setting.

Again we consider the algebro-geometric situation: $k$ is a perfect field, $A$ a smooth affine $k$-algebra of Krull dimension $d\geqslant 3$, and $X={\rm Spec}(A)$. Recall that we have the $\mathbb{A}^1$-homotopy fiber sequence
\[F_n\to\mathrm{BGL}_{n}\xrightarrow{\varphi=\varphi_n}\mathrm{BGL}.\]
The problem is to study the induced map on $\mathbb{A}^1$-homotopy classes
\[\varphi_*: [X, \mathrm{BGL}_{n}]_{\mathbb{A}^1}\to[X, \mathrm{BGL}]_{\mathbb{A}^1}.\]

\begin{definition}
	Let $K\in\sps(\EuScript{S}\mathsf{m}_k)_*$ be a pointed simplicial presheaf on the Nisnevich site $(\EuScript{S}\mathsf{m}_k, \mathrm{Nis})$. We define its \emph{free path space} ${\rm P}^*K:=K^{\Delta^1}$ and define its \emph{free loop space} $\mathcal{L}K\subset{\rm P}^*K$ by the equalizer diagram
	\begin{equation*}
	\begin{tikzcd}
	\mathcal{L}K\ar[r] & {\rm P}^*K=K^{\Delta^1}\ar[r,shift left=.4ex,"(d^1)^*"]
	\ar[r,shift right=.4ex,swap,"(d^0)^*"] &  K^{\Delta^0}=K,
	\end{tikzcd}
	\end{equation*}
	where the two parallel arrows are induced by the coface maps $d^1, d^0: \Delta^0\rightrightarrows\Delta^1$.
\end{definition}
Denote by
\[r: \mathcal{L}K\to K\]
the composite map in the above equalizer diagram; it is easy to see that we have a cartesian diagram
\[\begin{tikzcd}
	\mathcal{L}K \arrow[r] \arrow[d, "r"']\arrow[dr, phantom, "\pb", near start]  & {\rm P}^*K=K^{\Delta^1} \arrow[d, "{((d^1)^*, (d^0)^*)}"] \\
	K \arrow[r, "\delta_K"]          & {K\times K=K^{\partial\Delta^1}},
\end{tikzcd}\]
where $\delta_K: K\to K\times K$ is the diagonal map (we can use this cartesian diagram to define $\mathcal{L}K$ and $r: \mathcal{L}K\to K$). If $K$ is a sectionwise Kan complex, then the right vertical map is a sectionwise Kan fibration (since it is induced by the inclusion $\partial\Delta^1\hookrightarrow\Delta^1$; cf. \cite[Proposition 9.3.9]{Hir}), so the map $r: \mathcal{L}K\to K$ is also a sectionwise Kan fibration.

The map $(s^0)^*: K^{\Delta^0}=K\to{\rm P}^*K=K^{\Delta^1}$ defines a map $c: K\to\mathcal{L}K$ which is a section of the map $r$ (as $s^0d^0={\rm id}=s^0d^1$, by the cosimplicial identities).

The usual (based) \emph{loop space} $\Omega K\subset\mathcal{L}K$ fits into the cartesian square
\begin{equation*}
\begin{tikzcd}
\Omega K \arrow[r, "i"] \arrow[d]\arrow[dr, phantom, "\pb", near start]  & \mathcal{L}K \arrow[d, "r"] \\
* \arrow[r]          & K,
\end{tikzcd}
\end{equation*}
where the bottom arrow is the base point of $K$. Thus there is a sectionwise homotopy fiber sequence
\[\Omega K\xrightarrow{i}\mathcal{L} K\xrightarrow{r}K.\]

Now consider the following situation: we are given a \emph{strictly $\mathbb{A}^1$-invariant} sheaf of abelian groups $\mathbf{M}$ on $(\EuScript{S}\mathsf{m}_k, \mathrm{Nis})$ and a principal $\mathbb{A}^1$-homotopy fiber sequence
\[\mathrm{K}(\mathbf{M}, n)\to E\xrightarrow{q}\mathrm{BGL}\]
classified by a map $\theta: \mathrm{BGL}\to\mathrm{K}(\mathbf{M}, n+1), n\geqslant 2$. So the above $\mathbb{A}^1$-homotopy fiber sequence extends one step to the right (so that each 3-term forms an $\mathbb{A}^1$-homotopy fiber sequence):
\begin{equation}\label{1stfibseq}
\mathrm{K}(\mathbf{M}, n)\to E\xrightarrow{q}\mathrm{BGL}\xrightarrow{\theta}\mathrm{K}(\mathbf{M}, n+1).
\end{equation}

Moreover, by \cite[Theorem 2.2.5]{AHW2} or \cite[Lemma 3.1.3]{AFHo19}, we have an $\mathbb{A}^1$-homotopy fiber sequence
\[\mathrm{K}(\mathbf{M}, n)\xrightarrow{i} \mathcal{L}\mathrm{K}(\mathbf{M}, n+1)\xrightarrow{r}\mathrm{K}(\mathbf{M}, n+1).\]

Later we will specialize to the case of $\mathbf{M}$ being the typical and naturally-arising strictly $\mathbb{A}^1$-invariant sheaves like $\mathbf{K}^{\mathrm{M}}_{n+1}, \mathbf{K}^{\mathrm{MW}}_{n+1}$ to give enumeration results for non-stable vector bundles.

\begin{prop}\label{loopinjfreeloop}
	For $n\geqslant 1$, the induced map
	\[i_*: {\rm H}^n(X; \mathbf{M})=[X, \Omega\mathrm{K}(\mathbf{M}, n+1)]_{\mathbb{A}^1}\to[X, \mathcal{L}\mathrm{K}(\mathbf{M}, n+1)]_{\mathbb{A}^1}\]
	is injective. If moreover $\dim X=d\leqslant n$, then ${\rm H}^{n+1}(X; {\bf M})=[X, \mathrm{K}(\mathbf{M}, n+1)]_{\mathbb{A}^1}=0$ and hence $i_*: {\rm H}^n(X; \mathbf{M})=[X, \Omega\mathrm{K}(\mathbf{M}, n+1)]_{\mathbb{A}^1}\to[X, \mathcal{L}\mathrm{K}(\mathbf{M}, n+1)]_{\mathbb{A}^1}$ is an isomorphism.
\end{prop}
\begin{proof}
	The abelian group ${\rm H}^n(X; \mathbf{M})=[X, \Omega\mathrm{K}(\mathbf{M}, n+1)]_{\mathbb{A}^1}\cong[X, \Omega\mathrm{L}_{\mathbb{A}^1}\mathrm{K}(\mathbf{M}, n+1)]_{\mathbb{A}^1}$ is the same as the set of \emph{simplicial} homotopy classes of maps $X\to\Omega\mathrm{L}_{\mathbb{A}^1}\mathrm{K}(\mathbf{M}, n+1)$ by the strict $\mathbb{A}^1$-invariance assumption of $\mathbf{M}$ (which ensures that $\mathrm{K}(\mathbf{M}, n)$ is $\mathbb{A}^1$-local for all $n\in\mathbb{N}$), which is in turn $\pi_{0}(\Omega\mathrm{L}_{\mathbb{A}^1}\mathrm{K}(\mathbf{M}, n+1)(X))=[\Delta^0_+, \Omega\mathrm{L}_{\mathbb{A}^1}\mathrm{K}(\mathbf{M}, n+1)(X)]_{\sset_*}$. Note that all the constructions interplay well with the $\mathbb{A}^1$-fibrant replacement functor $\mathrm{L}_{\mathbb{A}^1}$, so the result follows from the classical result \cite[Theorem 2.6]{JTh65} (applying it to $A=S^0$ with $X$ replaced by the geometric realization of our $\mathrm{L}_{\mathbb{A}^1}\mathrm{K}(\mathbf{M}, n+1)(X)$).
	
	The last statement follows easily from the $\mathbb{A}^1$-homotopy fiber sequence
	\[\mathrm{K}(\mathbf{M}, n)\xrightarrow{i} \mathcal{L}\mathrm{K}(\mathbf{M}, n+1)\xrightarrow{r}\mathrm{K}(\mathbf{M}, n+1).\]
\end{proof}
\begin{remark}
	In \cite{JTh65}, the authors work with \emph{pointed connected CW-complexes} from the outset, but one easily sees that the relevant results are still true for $A=\Delta^0_+=S^0$.
	
	Moreover, at this point we do not know yet if the set $[X, \mathcal{L}\mathrm{K}(\mathbf{M}, n+1)]_{\mathbb{A}^1}$ is a group (though it indeed is, by \Cref{loopfreeloopbij}), we could not use the fact that the homomorphism $r_*: [X, \Omega\mathcal{L}\mathrm{K}(\mathbf{M}, n+1)]_{\mathbb{A}^1}\to[X, \Omega\mathrm{K}(\mathbf{M}, n+1)]_{\mathbb{A}^1}$ is split by $c_*$ (hence $r_*$ is surjective) yielding that the connecting homomorphism $[X, \Omega\mathrm{K}(\mathbf{M}, n+1)]_{\mathbb{A}^1}\to[X, \mathrm{K}(\mathbf{M}, n)]_{\mathbb{A}^1}$ vanishes to deduce the injectivity of $i_*$.
\end{remark}

Giving a stable vector bundle over $X$ is the same as giving a class $\xi\in[X, \mathrm{BGL}]_{\mathbb{A}^1}$. We will first study the lifting set $q_*^{-1}(\xi)$ following the treatment of \cite{JTh65}.

For $\theta\in{\rm H}^{n+1}(\mathrm{BGL}; \mathbf{M})$, let \[\theta'=\mathcal{L}\theta: \mathcal{L}\mathrm{BGL}\to\mathcal{L}\mathrm{K}(\mathbf{M}, n+1).\]
Given $\xi\in[X, \mathrm{BGL}]_{\mathbb{A}^1}$ as above, we say that $\gamma\in{\rm H}^n(X; \mathbf{M})$ is \emph{$\theta$-correlated to $\xi$} if there exists an element $\psi\in[X, \mathcal{L}\mathrm{BGL}]_{\mathbb{A}^1}$ such that $r_*\psi=\xi, \theta'_*\psi=i_*\gamma$. The situation is as the following diagram:
\begin{equation}\label{correl}
\begin{tikzcd}
& {[X, \mathcal{L}\mathrm{BGL}]_{\mathbb{A}^1}}  \arrow[d, "\theta'_*"']\arrow[r, "r_*"] & {[X, \mathrm{BGL}]_{\mathbb{A}^1}}\arrow[d, "\theta_*"]\\
{\gamma\in[X, \Omega\mathrm{K}(\mathbf{M}, n+1)]_{\mathbb{A}^1}} \arrow[r, "i_*"] & {[X, \mathcal{L}\mathrm{K}(\mathbf{M}, n+1)]_{\mathbb{A}^1}}  \arrow[r, "r_*"] & {[X, \mathrm{K}(\mathbf{M}, n+1)]_{\mathbb{A}^1}}.
\end{tikzcd}
\end{equation}
We denote the set of such elements $\gamma\in[X, \Omega\mathrm{K}(\mathbf{M}, n+1)]_{\mathbb{A}^1}$ by $C_{\theta}(\xi)\subset[X, \Omega\mathrm{K}(\mathbf{M}, n+1)]_{\mathbb{A}^1}={\rm H}^n(X, \mathbf{M})$. It is easy to check that
\begin{equation}
C_{\theta}(\xi)\neq\varnothing\Longleftrightarrow\theta_*(\xi)=0\in{\rm H}^{n+1}(X; \mathbf{M})=[X, \mathrm{K}(\mathbf{M}, n+1)]_{\mathbb{A}^1}\Longleftrightarrow q_*^{-1}(\xi)\neq\varnothing,
\end{equation}
in which case, one can see that for $\psi=c_*\xi$, we have $\theta'_*\psi\in i_*[X, \Omega\mathrm{K}(\mathbf{M}, n+1)]_{\mathbb{A}^1}(=\ker r_*)$.

\vspace{3mm}

Consider the action of ${\rm H}^n(X; \mathbf{M})=[X, \mathrm{K}(\mathbf{M}, n)]_{\mathbb{A}^1}$ on the set $[X, E]_{\mathbb{A}^1}$ associated to the fiber sequence (\ref{1stfibseq}) (given by ``concatenation of paths'' in the simplicial direction by suitably applying $\mathrm{L}_{\mathbb{A}^1}$).

For each $\eta\in[X, E]_{\mathbb{A}^1}$, this action restricts to a transitive action of ${\rm H}^n(X; \mathbf{M})=[X, \mathrm{K}(\mathbf{M}, n)]_{\mathbb{A}^1}$ on the set $q_*^{-1}(q_*\eta)\subset[X, E]_{\mathbb{A}^1}$, where \[q_*: [X, E]_{\mathbb{A}^1}\to[X, \mathrm{BGL}]_{\mathbb{A}^1}\]
is induced by the maps $q: E\to\mathrm{BGL}$ in (\ref{1stfibseq}). The following corresponds to \cite[Theorem 3.4]{JTh65}.
\begin{prop}\label{fundact}
	Fix $\eta\in[X, E]_{\mathbb{A}^1}$ Then the stabilizer of any $\eta'\in q_*^{-1}(q_*\eta)$ under this action is $C_{\theta}(q_*\eta)$; each orbit set forms a fiber of the map $q_*: [X, E]_{\mathbb{A}^1}\to[X, \mathrm{BGL}]_{\mathbb{A}^1}$.
	
	If $q_*^{-1}(\xi)\neq\varnothing$, then $q_*^{-1}(\xi)\cong{\rm H}^n(X; \mathbf{M})/C_{\theta}(\xi)$ and it has an abelian group structure.
\end{prop}
\begin{proof}
	As in the proof of \Cref{loopinjfreeloop}, by using the functorial $\mathbb{A}^1$-fibrant replacement functor $\mathrm{L}_{\mathbb{A}^1}$, we are reduced to the classical topological situation via \emph{simplicial} homotopy. Then the result follows from \cite[Chapter I, Lemma 7.3]{GJ} and \cite[Lemma 3.1, Theorems 3.2 and 3.3]{JTh65} (applying them to $A=S^0$ with $C$ replaced by our $\mathrm{L}_{\mathbb{A}^1}\mathrm{K}(\mathbf{M}, n+1)(X)$); notice that diagram (\ref{correl}) translates well to a diagram of simplicial homotopy classes in $\sset_*$; the stabilizer under the action is then identified with $C_{\theta}(q_*\eta)$, depending only on $q_*\eta'=q_*\eta\in[X, \mathrm{BGL}]_{\mathbb{A}^1}$ (note that $\theta_*(q_*\eta)=0$, compatible with the fact that $C_{\theta}(q_*\eta)\neq\varnothing$ required in (5.3), being a group).
\end{proof}

Note that $\mathrm{K}(\mathbf{M}, n+1)$ is a sheaf of simplicial abelian groups, and $\mathrm{BGL}$ has a binary operation $m$ induced by the maps $\mathrm{GL}_r\times\mathrm{GL}_s\to\mathrm{GL}_{r+s}, (A, B)\mapsto\left( \begin{array} { c c } { A } & \\  & { B }  \end{array} \right)$. Strictly speaking, these maps do not give a map $\mathrm{GL}\times\mathrm{GL}\to\mathrm{GL}$ on colimits, but will give a map $\coprod_{n\geqslant 0}\mathrm{BGL}_n\times\coprod_{n\geqslant 0}\mathrm{BGL}_n\to\coprod_{n\geqslant 0}\mathrm{BGL}_n$. By \cite[Proposition 4.3.10]{MV} (and see \cite[Remark 2 (p. 1162)]{SchT} for a correction), we have an $\mathbb{A}^1$-weak equivalence $\mathrm{BGL}\times\mathbb{Z}\xrightarrow{\simeq}\mathbf{R}\Omega\mathrm{B}(\coprod_{n\geqslant 0}\mathrm{BGL}_n)$, we get a map
\[m: (\mathrm{BGL}\times\mathbb{Z})\times(\mathrm{BGL}\times\mathbb{Z})\to\mathrm{BGL}\times\mathbb{Z}\]
in the $\mathbb{A}^1$-homotopy category, which restricts to the desired operation
\begin{equation}\label{hspop}
m: \mathrm{BGL}\times\mathrm{BGL}\to\mathrm{BGL}.
\end{equation}
Similar considerations in passing to the colimits appear in Quillen's construction of algebraic K-theory space using a choice of a bijection $\mathbb{N}\times\mathbb{N}\to\mathbb{N}$ and shows the choice doesn't matter up to homotopy on the $+$-construction (see also \cite[\S 15.2, Remark 2.3]{HJJS} in the topological situation); in the sequel, we just write $\mathrm{BGL}$ instead of the correct (but awkward) form $\mathrm{L}_{\mathbb{A}^1}\mathrm{BGL}$. Below, we prove that $\mathrm{BGL}$ is an abelian group object in $\EuScript{H}^{\mathbb{A}^1}_*(S)$.

Below we will need some results on motivic $T$-spectra, where $T=\mathbb{P}^1\simeq S^1\wedge\mathbb{G}_m$.
\begin{prop}
	Let $H=(H_0, H_1, H_2, \cdots)$ be a motivic $T$-spectrum for a base scheme $S$. Define $H'=(H_0', H_1', H_2', \cdots)$ by $H_n':=\mathbf{R}\underline{\mathrm{Hom}}_*(\mathbb{G}_m^{\wedge n}, H_n)$. Then
	\begin{enumerate}[label=\emph{(\arabic*)}]
		\item $H'=(H_0', H_1', H_2', \cdots)$ is a motivic $S^1$-spectrum (with suitable bonding maps), and $H_0'\simeq\mathrm{L}_{\mathbb{A}^1}H_0$.
		\item $H_0$ is an abelian group object in the pointed $\mathbb{A}^1$-homotopy category $\EuScript{H}^{\mathbb{A}^1}_*(S)$.
		\item For any $X\in\sps(\EuScript{S}\mathsf{m}_S)_*$, the derived mapping space ${\rm RMap}_*(X, H_0)\in\sset_*$ is an $\infty$-loop space and hence all its components are weakly equivalent.
	\end{enumerate}
\end{prop}
\begin{proof}
	We have $H_0'=\mathbf{R}\underline{\mathrm{Hom}}_*(\mathbb{G}_m^{\wedge 0}, H_0)=\mathbf{R}\underline{\mathrm{Hom}}_*(S^0, H_0)\simeq\underline{\mathrm{Hom}}_*(S^0, \mathrm{L}_{\mathbb{A}^1}H_0)=\mathrm{L}_{\mathbb{A}^1}H_0$.
	
	By performing a fibrant replacement, we may assume that $H$ is a fibrant motivic $T$-spectrum, so that each $H_n\in\sps(\EuScript{S}\mathsf{m}_S)_*$ is $\mathbb{A}^1$-fibrant and the adjoint bonding maps $H_n\to \underline{\mathrm{Hom}}_*(\mathbb{P}^1, H_{n+1})$ are $\mathbb{A}^1$-weak equivalences (even local weak equivalences). Thus
	\begin{equation*}
	\begin{split}
	H_n'&=\underline{\mathrm{Hom}}_*(\mathbb{G}_m^{\wedge n}, H_n)\simeq\underline{\mathrm{Hom}}_*(\mathbb{G}_m^{\wedge n}, \underline{\mathrm{Hom}}_*(S^1\wedge\mathbb{G}_m, H_{n+1}))\\
	&\cong\underline{\mathrm{Hom}}_*(S^1\wedge\mathbb{G}_m^{\wedge n+1}, H_{n+1})\cong\underline{\mathrm{Hom}}_*(S^1, \underline{\mathrm{Hom}}_*(\mathbb{G}_m^{\wedge n+1}, H_{n+1}))=\Omega H_{n+1}',
	\end{split}
	\end{equation*}
	showing that $H'=(H_0', H_1', H_2', \cdots)$ is a motivic $S^1$-spectrum, with bonding maps adjoint of the above weak equivalences, (1) is proved.
	
	For (2), just note that for any $X\in\sps(\EuScript{S}\mathsf{m}_S)_*$, the set $[X, H_0]_{\mathbb{A}^1, *}=[X, H_0']_{\mathbb{A}^1, *}=[X, \Omega^2H_2']_{\mathbb{A}^1, *}$ is an abelian group (any term in a motivic $S^1$-spectrum is an abelian group object in $\EuScript{H}^{\mathbb{A}^1}_*(S)$).
	
	For (3), note that for all $K\in\sset_*, X\in\sps(\EuScript{S}\mathsf{m}_S)_*$, there are canonical isomorphisms
	\[[K, {\rm RMap}_*(X, H_0)]_{\sset_*}\cong[K\wedge X, H_0]_{\mathbb{A}^1, *},\]
	the right hand side are abelian groups by (2).
\end{proof}
\begin{corollary}\label{bgl-comp}
	The space $\mathrm{BGL}$ is an abelian group object in the pointed $\mathbb{A}^1$-homotopy category $\EuScript{H}^{\mathbb{A}^1}_*(S)$ and for any $X\in\EuScript{S}\mathsf{m}_S$, all the components of the derived mapping space ${\rm RMap}(X, \mathrm{BGL})\in\sset_*$ are weakly equivalent.
	
	If $S={\rm Spec}(k)$ for a perfect field $k$, then $\mathrm{BSL}$ is an abelian group object in the pointed $\mathbb{A}^1$-homotopy category $\EuScript{H}^{\mathbb{A}^1}_*(k)$ and all the components of the derived mapping space ${\rm RMap}(X, \mathrm{BSL})\in\sset_*$ are weakly equivalent.
\end{corollary}
\begin{proof}
	Since the motivic $T$-spectrum representing algebraic K-theory has term $\mathrm{BGL}\times\mathbb{Z}$ in each level, by the previous result we see that $\mathrm{BGL}\times\mathbb{Z}$ is an abelian group object in $\EuScript{H}^{\mathbb{A}^1}(S)$. We conclude by noting that the projection $\mathrm{BGL}\times\mathbb{Z}\to\mathbb{Z}$ is a homomorphism of abelian group objects with kernel $\mathrm{BGL}$ (by kernel here we really mean that the kernel of the map $[-,\mathrm{BGL}\times\mathbb{Z}]_{\mathbb{A}^1, *}\to[-,\mathbb{Z}]_{\mathbb{A}^1, *}$ on $\EuScript{H}^{\mathbb{A}^1}_*(S)$).
	
	If $S={\rm Spec}(k)$, note that there is an $\mathbb{A}^1$-homotopy fiber sequence $\mathrm{BSL}\to\mathrm{BGL}\to\mathrm{B}\mathbb{G}_m=\mathrm{K}(\mathbf{K}^{\mathrm{M}}_{1}, 1)$ (thanks to the fact that the Picard group of a normal scheme is $\mathbb{A}^1$-invariant, yielding that $\mathbb{G}_m\in\grp_{k}^{\mathbb{A}^1}$), realizing $\mathrm{BSL}$ as the kernel of $\mathrm{BGL}\to\mathrm{K}(\mathbf{K}^{\mathrm{M}}_{1}, 1)$ (note that the second arrow splits), thus $\mathrm{BSL}$ is an abelian group object as well.
\end{proof}
This is to be compared with the fact in classical topology that for ``non-stable'' groups, the path-components of ${\rm map}(X, {\rm B}G)$---whose homotopy types are closely related with gauge groups---may represent (infinitely) many distinct homotopy types (see e.g. \cite{Smithhofun}). Recall that an \emph{abelian group object} in the pointed homotopy category ${\rm Ho}(\sset_*)$, also called an \emph{abelian $H$-group}, is a pointed space with a binary operation which makes the sets of homotopy classes to it abelian groups in a functorial way.
\begin{prop}\label{abhspcomp}
	Let $(B, 0)\in\sset_*$ be an abelian $H$-group with a binary operation $m: (B, 0)\times(B, 0)\to(B, 0), (b, b')\mapsto bb'$. Denote its path components by $B_{\xi}, \xi\in\pi_0B$. Then
	\[(B, 0)\cong (B_0, 0)\times\pi_0(B, 0)\]
	as abelian $H$-groups.
\end{prop}
\begin{proof}
	Fix a base point $b_{\xi}$ in each component $B_{\xi}$ with $b_0=0$. Then we easily find that the maps
	\[B\to B_0\times\pi_0(B, 0), (b\in B_{\xi})\mapsto(bb_{-\xi}, \xi)\]
	and
	\[B_0\times\pi_0(B, 0)\to B, (b, \xi)\mapsto bb_{\xi}\]
	are homomorphisms of $H$-groups and are homotopy inverse to each other.
\end{proof}

The map $m: \mathrm{BGL}\times\mathrm{BGL}\to\mathrm{BGL}$ as in (\ref{hspop}) induces for any $X\in\EuScript{S}\mathsf{m}_k$ the addition on $\widetilde{{\rm K}}_0(X)=[X, \mathrm{BGL}]_{\mathbb{A}^1}$: $m_*(\xi, \xi')=\xi+\xi'$ (one sees this by noting the effect of $m$ on the cocycles of vector bundles gives the direct sums).

Using this, for a given pointed map $\theta: (\mathrm{BGL}, e)\to(\mathrm{K}(\mathbf{M}, n+1), 0)$, we can define a map
\[\theta_1: \mathrm{BGL}\times\mathrm{BGL}\to\mathrm{K}(\mathbf{M}, n+1)\]
which on sections is given by $(x, y)\mapsto\theta(m(x, y))-\theta(y)$ and further a map
\[\theta_2: \Omega\mathrm{BGL}\times\mathrm{BGL}\to\Omega\mathrm{K}(\mathbf{M}, n+1)\]
which on a loop is given by $\theta_1$ (see \cite[p. 486]{JTh65}) since $\theta_1(e, y)=0$. Formally, let $\theta_1^{\flat}: \mathrm{BGL}\to\underline{\mathrm{Hom}}(\mathrm{BGL}, \mathrm{K}(\mathbf{M}, n+1))$ be the adjoint of $\theta_1: \mathrm{BGL}\times\mathrm{BGL}\to\mathrm{K}(\mathbf{M}, n+1)$, then $\theta_2$ is the restriction of the adjoint of the composite
\[\mathrm{BGL}\xrightarrow{\theta_1^{\flat}}\underline{\mathrm{Hom}}(\mathrm{BGL}, \mathrm{K}(\mathbf{M}, n+1))\to\underline{\mathrm{Hom}}(\mathcal{L}\mathrm{BGL}, \mathcal{L}\mathrm{K}(\mathbf{M}, n+1)).\]
It can be obtained in the same way as we get $v$ from $u$ in \S 3.

We thus get a map
\[(\theta_2)_*: [X, \Omega\mathrm{BGL}]_{\mathbb{A}^1}\times[X, \mathrm{BGL}]_{\mathbb{A}^1}\to[X, \Omega\mathrm{K}(\mathbf{M}, n+1)]_{\mathbb{A}^1}=[X, \mathrm{K}(\mathbf{M}, n)]_{\mathbb{A}^1}.\]
For a class $\xi\in[X, \mathrm{BGL}]_{\mathbb{A}^1}$, we obtain the map
\begin{equation}\label{delt}
\Delta(\theta, \xi): {\rm K}_1(X)=[X, \Omega\mathrm{BGL}]_{\mathbb{A}^1}\to[X, \mathrm{K}(\mathbf{M}, n)]_{\mathbb{A}^1}={\rm H}^{n}(X; \mathbf{M})
\end{equation}
given by $\beta\mapsto(\theta_2)_*(\beta, \xi)$. Then $\Delta(\theta, \xi)=(\theta_2)_*(-, \xi)$ is a homomorphism of abelian groups, whose effect is given by \Cref{cptbmult} below.

Note further that if $\tau: \mathbf{M}\to\mathbf{M}'$ is a homomorphism of sheaves of abelian groups, then it is easy to see that
\[\tau_*\Delta(\theta, \xi)=\Delta(\tau_*\theta, \xi).\]

The binary operation $m: \mathrm{BGL}\times\mathrm{BGL}\to\mathrm{BGL}$ also determines a map $m': \Omega\mathrm{BGL}\times\mathrm{BGL}\to\mathcal{L}\mathrm{BGL}$ which on a loop is given by $m$ (see \cite[p. 493]{JTh65}). Formally, let $m^{\flat}: \mathrm{BGL}\to\underline{\mathrm{Hom}}(\mathrm{BGL}, \mathrm{BGL})$ be the adjoint of $m: \mathrm{BGL}\times\mathrm{BGL}\to\mathrm{BGL}$, then $m'$ is the restriction of the adjoint of the composite
\[\mathrm{BGL}\xrightarrow{m^{\flat}}\underline{\mathrm{Hom}}(\mathrm{BGL}, \mathrm{BGL})\xrightarrow{(-)^{\Delta^1}}\underline{\mathrm{Hom}}(\mathrm{BGL}^{\Delta^1}, \mathrm{BGL}^{\Delta^1}).\]
On local sections $(\lambda, y)$, $m'$ ``translates'' a loop $\lambda$ at the base point of $\mathrm{BGL}$ to a loop based at $y$.

\begin{prop}\label{loopfreeloopbij}
	The map $m': \Omega\mathrm{BGL}\times\mathrm{BGL}\to\mathcal{L}\mathrm{BGL}$ is an isomorphism in $\EuScript{H}^{\mathbb{A}^1}_*(S)$.
	
	So the induced map
	\[m'_*: [X, \Omega\mathrm{BGL}]_{\mathbb{A}^1}\times[X, \mathrm{BGL}]_{\mathbb{A}^1}\to[X, \mathcal{L}\mathrm{BGL}]_{\mathbb{A}^1}\]
	is an isomorphism of abelian groups and satisfies
	\[m'_*(\beta, \alpha)=i_*\beta+c_*\alpha,\]
	where $c: \mathrm{BGL}\to\mathcal{L}\mathrm{BGL}, i: \Omega\mathrm{BGL}\to\mathcal{L}\mathrm{BGL}$ are the maps constructed before. So $r_*m'_*(\beta, \alpha)=\alpha$.
\end{prop}
\begin{proof}
	As in the proof of \Cref{loopinjfreeloop}, for $X\in\EuScript{S}\mathsf{m}_k$, using
	\[[X, \Omega\mathrm{BGL}]_{\mathbb{A}^1}\cong[X, \Omega\mathrm{L}_{\mathbb{A}^1}\mathrm{BGL}]_{\mathbb{A}^1}=[\Delta^0_+, \Omega\mathrm{L}_{\mathbb{A}^1}\mathrm{BGL}(X)]_{\sset_*}\]
	etc., and in general, for $X\in\sps(\EuScript{S}\mathsf{m}_S)_*$, we have
	\[[X, \Omega\mathrm{BGL}]_{\mathbb{A}^1, *}\cong\pi_0{\rm RMap}_*(X, \Omega\mathrm{L}_{\mathbb{A}^1}\mathrm{BGL})\]
	etc., we are reduced to the classical topological situation, which is given in \cite[Theorem 2.7]{JTh65}.
\end{proof}

Denoting the addition on $\mathrm{K}(\mathbf{M}, n+1)$ by $a$, then we have a similar construction of an isomorphism $a': \Omega\mathrm{K}(\mathbf{M}, n+1)\times\mathrm{K}(\mathbf{M}, n+1)\to\mathcal{L}\mathrm{K}(\mathbf{M}, n+1)$, which induces an isomorphism
\[a'_*: [X, \Omega\mathrm{K}(\mathbf{M}, n+1)]_{\mathbb{A}^1}\times[X, \mathrm{K}(\mathbf{M}, n+1)]_{\mathbb{A}^1}\to[X, \mathcal{L}\mathrm{K}(\mathbf{M}, n+1)]_{\mathbb{A}^1}.\]
In fact, the map $a'$ itself is an isomorphism by the degreewise split fiber sequence
\[
\begin{tikzcd}
\Omega\mathrm{K}(\mathbf{M}, n+1) \arrow[r,"i"] & \mathcal{L}\mathrm{K}(\mathbf{M}, n+1)  \arrow[r,shift left,"r"]
\arrow[r,<-,shift right,swap,"c"]& \mathrm{K}(\mathbf{M}, n+1).
\end{tikzcd}
\]

\begin{prop}\label{cptbmult}
	We have
	\begin{equation}
	\theta'_*m'_*(\beta, \xi)=a'_*(\Delta(\theta, \xi)(\beta), \theta_*\xi)=a'_*((\theta_2)_*(\beta, \xi), \theta_*\xi),
	\end{equation}
	as shown in the diagram below:
	\[
	\begin{tikzcd}[column sep = 6em]
	{[X, \Omega\mathrm{BGL}]_{\mathbb{A}^1}\times[X, \mathrm{BGL}]_{\mathbb{A}^1}} \arrow{r}{((\theta_2)_*, \theta_*{\rm pr}_2)} \arrow[d, "\rotatebox{-90}{$\cong$}", "m'_*"']  & {[X, \Omega\mathrm{K}(\mathbf{M}, n+1)]_{\mathbb{A}^1}\times[X, \mathrm{K}(\mathbf{M}, n+1)]_{\mathbb{A}^1}} \arrow[d, "a'_*", "\rotatebox{90}{$\cong$}"']\\
	{[X, \mathcal{L}\mathrm{BGL}]_{\mathbb{A}^1}} \arrow[r, "\theta'_*"']          & {[X, \mathcal{L}\mathrm{K}(\mathbf{M}, n+1)]_{\mathbb{A}^1}}
	\end{tikzcd}\]
	
	If $\theta_*\xi=0$, then 
	\begin{equation}\label{fundcordelta}
	\theta'_*m'_*(\beta, \xi)=i_*(\Delta(\theta, \xi)(\beta)).
	\end{equation}
\end{prop}
\begin{proof}
	As in the proof of \Cref{loopinjfreeloop}, we are reduced to the classical topological situation, which is given in \cite[eq. (2.8), p. 495]{JTh65}.
\end{proof}

The following corresponds to \cite[Theorems 2.9 and 1.2]{JTh65}, which identifies the set $C_{\theta}(\xi)$ of those $\gamma\in{\rm H}^n(X; \mathbf{M})$ which are $\theta$-correlated to $\xi$ with the image of our map $\Delta(\theta, \xi)$; see (\ref{1stfibseq}) for the notations and the relevant maps).
\begin{prop}\label{lift=ck}
	Let $\theta\in{\rm H}^{n+1}(\mathrm{BGL}; \mathbf{M})=[\mathrm{BGL}, \mathrm{K}(\mathbf{M}, n+1)]_{\mathbb{A}^1}$ and $\xi\in[X, \mathrm{BGL}]_{\mathbb{A}^1}$ with $\theta_*(\xi)=0\in{\rm H}^{n+1}(X; \mathbf{M})$.
	Then for $\gamma\in[X, \mathrm{K}(\mathbf{M}, n)]_{\mathbb{A}^1}$, we have
	\[\gamma\in C_{\theta}(\xi)\Longleftrightarrow\gamma\in{\rm im}\big({\rm K}_1(X)\xrightarrow{\Delta(\theta, \xi)}{\rm H}^{n}(X; \mathbf{M})\big).\]
	If $\theta_*(\xi)=0\in{\rm H}^{n+1}(X; \mathbf{M})$, we have
	\begin{equation}
	{\rm im}\ \Delta(\theta, \xi)=C_{\theta}(\xi), q_*^{-1}(\xi)\cong{\rm coker}\ \Delta(\theta, \xi).
	\end{equation}
\end{prop}
\begin{proof}
Assume $\theta_*(\xi)=0\in{\rm H}^{n+1}(X; \mathbf{M})$. If $\gamma=\Delta(\theta, \xi)\beta$ for some $\beta\in{\rm K}_1(X)=[X, \Omega\mathrm{BGL}]_{\mathbb{A}^1}$, we take $\psi=m'_*(\beta, \xi)$, then $r_*\psi=\xi$ and by \cref{fundcordelta}, $\theta'_*\psi=i_*\gamma$, i.e. $\gamma\in C_{\theta}(\xi)$.

Conversely, let $\gamma\in C_{\theta}(\xi)$ and assume $r_*\psi=\xi, \theta'_*\psi=i_*\gamma$ for some $\psi\in[X, \mathcal{L}\mathrm{BGL}]_{\mathbb{A}^1}$. As $m'_*$ is a bijection, we can take $\beta\in[X, \Omega\mathrm{BGL}]_{\mathbb{A}^1}$ and $\alpha\in[X, \mathrm{BGL}]_{\mathbb{A}^1}$ such that $m'_*(\beta, \alpha)=\psi$, then $\alpha=r_*m'_*(\beta, \alpha)=r_*\psi=\xi$ and again by \cref{fundcordelta},
\[i_*\gamma=\theta'_*\psi=\theta'_*m'_*(\beta, \xi)=i_*\Delta(\theta, \xi)\beta.\]
Since $i_*$ is injective by \Cref{loopinjfreeloop}, we see that $\gamma=\Delta(\theta, \xi)\beta\in{\rm im}\ \Delta(\theta, \xi)$.
\end{proof}

\section{Application to vector bundles of critical rank}

Recall that we have the $\mathbb{A}^1$-homotopy fiber sequence
\[F_n\to\mathrm{BGL}_{n}\xrightarrow{\varphi_n}\mathrm{BGL}.\]
We have the induced map on $\mathbb{A}^1$-homotopy classes
\[(\varphi_n)_*: [X, \mathrm{BGL}_{n}]_{\mathbb{A}^1}\to[X, \mathrm{BGL}]_{\mathbb{A}^1}.\]
We want to describe the inverse-image under $\varphi_*$ of a class in the right hand side.

The following corresponds to \cite[Theorem 1.6]{JTh65}.
\begin{theorem}\label{enu-oddrk}
	Let $n\geqslant 3$ be odd, let $A$ be a smooth affine $k$-algebra of dimension $n$ and let $\xi$ be a stable vector bundle over ${\rm Spec}(A)$, whose classifying map is still denoted by $\xi: {\rm Spec}(A)\to\mathrm{BGL}$. Then for $\Delta(c_{n+1}, \xi): {\rm K}_1(A)\to{\rm H}^{n}({\rm Spec}(A); \mathbf{K}^{\mathrm{M}}_{n+1})$ given as in \emph{(\ref{delt})}, we have a bijection
	\[(\varphi_n)_*^{-1}(\xi)\longleftrightarrow{\rm coker}(\Delta(c_{n+1}, \xi)).\]
\end{theorem}
\begin{proof}
This follows from \Cref{lift=ck}, using \Cref{towiso}: the condition $c_{n+1}(\xi)=0$ is satisfied as Chern classes vanish above the rank of the vector bundle.
\end{proof}

As usual, we have the universal $j$-th Chern classes $c_{j}\in{\rm H}^{j}(\mathrm{BGL}; \mathbf{K}^{\mathrm{M}}_{j})={\rm CH}^{j}(\mathrm{BGL})$ for $j\in\mathbb{N}$ (see e.g. \cite[Chapter 2]{Tot14}).

For the result below, we have the \emph{exterior product} $c_r\times c_{n-r}:=\mu(c_r\otimes c_{n-r})$, where $\mu$ is the map
\[[\mathrm{BGL}, \mathrm{K}(\mathbf{K}^{\mathrm{M}}_{r}, r)]_{\mathbb{A}^1}\otimes[\mathrm{BGL}, \mathrm{K}(\mathbf{K}^{\mathrm{M}}_{n-r}, n-r)]_{\mathbb{A}^1}\to[\mathrm{BGL}\times\mathrm{BGL}, \mathrm{K}(\mathbf{K}^{\mathrm{M}}_{n}, n)]_{\mathbb{A}^1}\]
induced by the multiplication in the graded sheaves of Milnor K-groups $\mathbf{K}^{\mathrm{M}}_*$ (see the end of \S 3).
\begin{theorem}\label{comulp}
	The universal Chern classes satisfy
	\begin{equation}\label{ch-cartan}
	m^*c_n=\sum_{r=0}^{n}(c_r\times c_{n-r})\in[\mathrm{BGL}\times\mathrm{BGL}, \mathrm{K}(\mathbf{K}^{\mathrm{M}}_{n}, n)]_{\mathbb{A}^1}.
	\end{equation}
\end{theorem}
\begin{proof}
	Since the universal $r$-th Chern class $c_{r}: \mathrm{BGL}\to \mathrm{K}(\mathbf{K}^{\mathrm{M}}_{r}, r)$ factors through the inclusion $\mathrm{BGL}_n\to\mathrm{BGL}$ for each $n\geqslant r$ and stabilize, both sides in the formula above are determined by their restriction to $\mathrm{BGL}_n\times\mathrm{BGL}_n$. And we have ${\rm CH}^{*}(\mathrm{BGL}_n)=\mathbb{Z}[c_1, \cdots, c_n]$ (\cite[Theorem 2.13]{Tot14}).
	
	By the Chow K\"{u}nneth formula in \cite[\S 6]{TotCw97} (since $\mathrm{BGL}_n$ can be approximated by ``linear varieties'', the Grassmannians; see also \cite[Chapter 17]{Tot14} and \cite{Tot16} for more discussions on when the Chow K\"unneth formula holds), we may write
	\[{\rm CH}^{*}(\mathrm{BGL}_n\times\mathrm{BGL}_n)=\mathbb{Z}[c_i\times c_j: 1\leqslant i, j\leqslant n]\]
	with $|c_i\times c_j|=i+j$.
	
	On the other hand, for any $U\in\EuScript{S}\mathsf{m}_k$ and $(\xi, \xi')\in[U, \mathrm{BGL}]_{\mathbb{A}^1}\times[U, \mathrm{BGL}]_{\mathbb{A}^1}$, as $m_*(\xi, \xi')=\xi+\xi'$, by Whitney sum formula for Chern classes, we have
	\[(m^*c_n)(\xi, \xi')=c_nm_*(\xi, \xi')=c_n(\xi+\xi')=\sum_{r=0}^{n}c_r(\xi)\cdot c_{n-r}(\xi')=\sum_{r=0}^{n}(c_r\times c_{n-r})(\xi, \xi').\]
	
	Now we take $U=X/G$, where $X$ is an open subscheme of a representation $\mathbb{A}^N_k$ of $G=\mathrm{GL}_n\times\mathrm{GL}_n$ for $N$ big enough such that $\mathbb{A}^N_k\setminus X$ is $G$-invariant and has codimension bigger than $n$, and such that $G$ acts freely on $X$ (see \cite[\S 2.2]{Tot14}), so $X\to U$ is a principal $G$-bundle in the category of varieties; we then have ${\rm CH}^{*}(\mathrm{BGL}_n\times\mathrm{BGL}_n)={\rm CH}^{*}(U)$ (\cite[Theorem 2.5]{Tot14}) which is induced by a map $u=(\xi, \xi'): U\to\mathrm{BGL}_n\times\mathrm{BGL}_n$ classifying the principal $G$-bundle $X\to U$ (so that we can consider the induced map $u^*$ as the identity on Chow rings). Then the above formula becomes 
	\[u^*m^*c_n=u^*\sum_{r=0}^{n}(c_r\times c_{n-r}),\]
	proving (\ref{ch-cartan}).
\end{proof}

We have the \emph{suspension homomorphism}
\[\sigma: {\rm CH}^j(\mathrm{BGL})=[\mathrm{BGL}, \mathrm{K}(\mathbf{K}^{\mathrm{M}}_{j}, j)]_{\mathbb{A}^1, *}\to[\Omega\mathrm{BGL}, \Omega\mathrm{K}(\mathbf{K}^{\mathrm{M}}_{j}, j)]_{\mathbb{A}^1, *}=\widetilde{{\rm H}}^{j-1}(\mathrm{GL}; \mathbf{K}^{\mathrm{M}}_{j})\]
for $j\geqslant 1$, induced by the loop functor $\Omega$ which is given by
\[\sigma([\theta])=[\Omega\theta].\]
Note that in the above, it doesn't matter whether we use pointed or unpointed $\mathbb{A}^1$-homotopy classes provided $j>2$, see \cite[Lemma 2.1]{AF14a}. In fact this is also the case for $j=2$ since in positive degree, the reduced and unreduced cohomologies agree.

\begin{prop}\label{compute-delta}
	For $\beta\in{\rm K}_1(X)=[X, \mathrm{GL}]_{\mathbb{A}^1}$, we have
	\begin{equation}
	\Delta(c_{n+1}, \xi)\beta=\beta^*\sigma c_{n+1}+\sum_{r=1}^{n}(\beta^*\sigma c_r)\cdot c_{n+1-r}(\xi)=(\Omega c_{n+1})(\beta)+\sum_{r=1}^{n}((\Omega c_{r})(\beta))\cdot c_{n+1-r}(\xi).
	\end{equation}
\end{prop}

To prove this, we first prove the following more general result, which corresponds to \cite[Theorem 1.3]{JTh65} whose proof is given in \cite[\S 5]{JTh65}.
\begin{prop}\label{computedeltaform}
	Let $\mathbf{K}$ be a sheaf of rings, strictly $\mathbb{A}^1$-invariant as a sheaf of abelian groups. Let $\beta\in{\rm K}_1(X)=[X, \Omega\mathrm{BGL}]_{\mathbb{A}^1}=[X, \mathrm{GL}]_{\mathbb{A}^1}$ and $\alpha\in\widetilde{{\rm H}}^*(\mathrm{K}(\mathbf{M}, n+1); \mathbf{K})$. Assume $a^*\alpha=\alpha\times1+1\times\alpha, m^*\theta^*\alpha\in\mu(\widetilde{{\rm H}}^*(\mathrm{BGL}; \mathbf{K})^{\otimes 2})$; more precisely, let $m^*\theta^*\alpha=\theta^*\alpha\times1+1\times\theta^*\alpha+\sum u_i\times v_i=\mu(\theta^*\alpha\otimes1+1\otimes\theta^*\alpha+\sum u_i\otimes v_i)$, where $u_i, v_i\in\widetilde{{\rm H}}^*(\mathrm{BGL}; \mathbf{K})$.
	\begin{enumerate}[label=\emph{(\arabic*)}]
		\item We have $\theta_1^*\alpha=m^*\theta^*\alpha-1\times\theta^*\alpha$.
		\item Denote $\gamma=\Delta(\theta, \xi)\beta$, then
		\[\gamma^*\sigma\alpha=(\beta^*\sigma)\cdot(\theta^*\alpha)+\sum (\beta^*\sigma u_i)\cdot (\xi^*v_i).\]
	\end{enumerate}
\end{prop}
We draw the following diagram to trace the various maps in this result:
\begin{equation*}
\begin{tikzcd}
\alpha\in\widetilde{{\rm H}}^*(\mathrm{K}(\mathbf{M}, n+1); \mathbf{K}) \arrow[r,"\sigma"] \arrow[d, equal] &\widetilde{{\rm H}}^{*-1}(\mathrm{K}(\mathbf{M}, n); \mathbf{K})
\arrow[r,"\gamma^*"] &{\rm H}^{*-1}(X; \mathbf{K})
\\
\alpha\in\widetilde{{\rm H}}^*(\mathrm{K}(\mathbf{M}, n+1); \mathbf{K}) \arrow[r,"a^*"]\arrow[d, "\theta^*"']\arrow[rd, "\theta_1^*"] & \widetilde{{\rm H}}^*(\mathrm{K}(\mathbf{M}, n+1)\times\mathrm{K}(\mathbf{M}, n+1); \mathbf{K})  & \widetilde{{\rm H}}^*(\mathrm{K}(\mathbf{M}, n+1); \mathbf{K})^{\otimes 2} \arrow[l,"\mu"'] \\
\alpha'=\theta^*\alpha\in\widetilde{{\rm H}}^*(\mathrm{BGL}; \mathbf{K}) \arrow[r,"m^*"]\ & \widetilde{{\rm H}}^*(\mathrm{BGL}\times\mathrm{BGL}; \mathbf{K})  & \widetilde{{\rm H}}^*(\mathrm{BGL}; \mathbf{K})^{\otimes 2}. \arrow[l,"\mu"'] 
\end{tikzcd}
\end{equation*}
\begin{proof}
	Write $\alpha'=\theta^*\alpha$.
	\begin{enumerate}[label=(\arabic*)]
		\item Let $g$ be the composite $\mathrm{BGL}\times\mathrm{BGL}\xrightarrow{1\times\Delta}\mathrm{BGL}\times\mathrm{BGL}\times\mathrm{BGL}\xrightarrow{m\times 1}\mathrm{BGL}\times\mathrm{BGL}$ and let $h$ be the composite $\mathrm{K}(\mathbf{M}, n+1)\times\mathrm{K}(\mathbf{M}, n+1)\xrightarrow{1\times w}\mathrm{K}(\mathbf{M}, n+1)\times\mathrm{K}(\mathbf{M}, n+1)\xrightarrow{a}\mathrm{K}(\mathbf{M}, n+1)$, where $\Delta: \mathrm{BGL}\to\mathrm{BGL}\times\mathrm{BGL}$ is the diagonal and $w: \mathrm{K}(\mathbf{M}, n+1)\to\mathrm{K}(\mathbf{M}, n+1)$ is the inverse map. We have
		\[\theta_1=h(\theta\times\theta)g.\]
		As $w^*\alpha=-\alpha$, we have $h^*\alpha=(1\times w)^*a^*\alpha=(1\times w)^*(\alpha\times1+1\times\alpha)=\alpha\times1-1\times\alpha$, so $(\theta\times\theta)^*h^*\alpha=\alpha'\times1-1\times\alpha'$. We conclude by noting that
		\[g^*(\alpha'\times1)=(1\times\Delta)^*(m\times 1)^*(\alpha'\times1)=(1\times\Delta)^*(m^*\alpha'\times1)=m^*\alpha'\]
		(here we need the fact that $m^*\alpha'=\sum a_i\times b_i\in\mu(\widetilde{{\rm H}}^*(\mathrm{BGL}; \mathbf{K})^{\otimes 2})$ and the observation that $(1\times\Delta)^*((a\times b)\times 1)=a\times b$) and
		\[g^*(1\times\alpha')=(1\times\Delta)^*(m\times 1)^*(1\times\alpha')=(1\times\Delta)^*(1\times(1\times\alpha'))=1\times\alpha'.\]
		\item Consider the maps of pairs
		\[u: (\mathrm{BGL}\times\mathrm{BGL}, *\times\mathrm{BGL})\to(\mathrm{K}(\mathbf{M}, n+1), 0)\]
		and
		\[v: (\Omega\mathrm{BGL}\times\mathrm{BGL}, *\times\mathrm{BGL})\to(\Omega\mathrm{K}(\mathbf{M}, n+1), 0)\]
		determined by the maps $\theta_1, \theta_2$ respectively.
		
		By (1), we can write $\theta_1^*\alpha=\mu(\alpha_0')$ with $\alpha_0'=\alpha'\otimes1+\sum u_i\otimes v_i\in\widetilde{{\rm H}}^*(\mathrm{BGL}; \mathbf{K})^{\otimes 2}$. Let $\iota_0: (\mathrm{BGL}, \varnothing)\to(\mathrm{BGL}, *), \iota_1: (\mathrm{BGL}\times\mathrm{BGL}, \varnothing)\to(\mathrm{BGL}\times\mathrm{BGL}, *\times \mathrm{BGL})$ and $\iota_2: (\Omega\mathrm{BGL}\times\mathrm{BGL}, \varnothing)\to(\Omega\mathrm{BGL}\times\mathrm{BGL}, *\times \mathrm{BGL})$ be the inclusions of pairs, so $\theta_1=u\iota_1, \theta_2=v\iota_2$. Let $\alpha_0:=(1\otimes\iota_0^*)(\alpha_0')=\alpha'\otimes1+\sum u_i\otimes \iota_0^*v_i\in\widetilde{{\rm H}}^*(\mathrm{BGL}; \mathbf{K})\otimes{\rm H}^*(\mathrm{BGL}; \mathbf{K})$.
		\begin{equation*}
		 \begin{tikzcd}
		 \alpha\in\widetilde{{\rm H}}^*(\mathrm{K}(\mathbf{M}, n+1); \mathbf{K}) \arrow[r,"\theta_1^*"]\arrow[d, equal] & {\rm H}^*(\mathrm{BGL}\times\mathrm{BGL}; \mathbf{K}) & \widetilde{{\rm H}}^*(\mathrm{BGL}; \mathbf{K})^{\otimes 2}\ni\alpha_0' \arrow[l,"\mu"'] \arrow[d, "1\otimes\iota_0^*"]  \\
		  \alpha\in\widetilde{{\rm H}}^*(\mathrm{K}(\mathbf{M}, n+1); \mathbf{K}) \arrow[r,"u^*"]\ & {\rm H}^*(\mathrm{BGL}\times\mathrm{BGL}, *\times\mathrm{BGL}; \mathbf{K})  \arrow[u, "\iota_1^*"]& \widetilde{{\rm H}}^*(\mathrm{BGL}; \mathbf{K})\otimes{\rm H}^*(\mathrm{BGL}; \mathbf{K})\ni\alpha_0 \arrow[l,"\mu"'] 
		 \end{tikzcd}
		 \end{equation*}
		 Then $\iota_1^*u^*\alpha=\theta_1^*\alpha=\mu(\alpha_0')=\iota_1^*\mu(\alpha_0), u^*\alpha=\mu(\alpha_0)+\varepsilon, \varepsilon\in\ker\iota_1^*$. So by (\ref{relcoholoop}) and (\ref{sus-ext-prod}),
		 \[\theta_2^*\sigma\alpha=\iota_2^*v^*\sigma\alpha=\iota_2^*\rho u^*\alpha=\iota_2^*\rho\mu(\alpha_0)+\iota_2^*\rho\varepsilon,\]
		 \[\iota_2^*\rho\mu(\alpha_0)=\iota_2^*\mu(\sigma\otimes1)\alpha_0=\iota_2^*\mu(\sigma\alpha'\otimes1+\sum \sigma u_i\otimes \iota_0^*v_i)=\iota_2^*(\sigma\alpha'\times1+\sum \sigma u_i\times \iota_0^*v_i),\]
		 \[\iota_2^*\rho\varepsilon=\sigma'\iota_1^*\varepsilon=0,\]
		 hence
		 \[\theta_2^*\sigma\alpha=\sigma'\theta_1^*(\alpha).\]
		 Of course, as ${\rm im}\theta_1^*\subset{\rm im}\iota_1^*$, one can get this relation more directly from the ``outer contour'' of (\ref{relcoholoop}), without using (\ref{sus-ext-prod}).
		 
		  By (1) and the fact that $\sigma'(a\times b)=(\sigma a)\times b$ in (\ref{relcoholoop}), we get
		\[\theta_2^*\sigma\alpha=\sigma'(m^*\alpha'-1\times\alpha')=\sigma\alpha'\times1+\sum (\sigma u_i)\times v_i.\]
		
		Since $\gamma=\Delta(\theta, \xi)\beta=(\theta_2)_*(\beta, \xi)$ is represented by the composite
		\[X\xrightarrow{\Delta}X\times X\xrightarrow{\beta\times\xi}\Omega\mathrm{BGL}\times\mathrm{BGL}\xrightarrow{\theta_2}\Omega\mathrm{K}(\mathbf{M}, n+1),\]
		we have
		\begin{equation*}
		\begin{split}
		\gamma^*\sigma\alpha&=\Delta^*(\beta\times\xi)^*\theta_2^*\sigma\alpha=\Delta^*(\beta\times\xi)^*(\sigma\alpha'\times1+\sum (\sigma u_i)\times v_i)\\
		&=\Delta^*(\beta^*\sigma\alpha'\times1+\sum (\beta^*\sigma u_i)\times(\xi^*v_i))=(\beta^*\sigma)\cdot(\theta^*\alpha)+\sum (\beta^*\sigma u_i)\cdot (\xi^*v_i).
		\end{split}
		\end{equation*}
	\end{enumerate}
\end{proof}
\begin{proof}[Proof of {\bf \Cref{compute-delta}}]
	We apply \Cref{computedeltaform} to the situation when $\mathbf{M}=\mathbf{K}^{\mathrm{M}}_{n+1}, \mathbf{K}=\mathbf{K}^{\mathrm{M}}_{*}, \theta=c_{n+1}: \mathrm{BGL}\to\mathrm{K}(\mathbf{K}^{\mathrm{M}}_{n+1}, n+1)$ and $\alpha\in\widetilde{{\rm H}}^{n+1}(\mathrm{K}(\mathbf{K}^{\mathrm{M}}_{n+1}, n+1); \mathbf{K}^{\mathrm{M}}_{n+1})=[\mathrm{K}(\mathbf{K}^{\mathrm{M}}_{n+1}, n+1), \mathrm{K}(\mathbf{K}^{\mathrm{M}}_{n+1}, n+1)]_{\mathbb{A}^1, *}$ being represented by the identity map of $\mathrm{K}(\mathbf{K}^{\mathrm{M}}_{n+1}, n+1)$, in which case, $\theta^*\alpha=\theta, \gamma^*\sigma\alpha=\gamma$.
	
	We need to prove that $[a]=a^*\alpha=\alpha\times1+1\times\alpha$ holds. But $\alpha\times1+1\times\alpha$ is also represented by the addition map $a$ (Eckmann-Hilton property).
\end{proof}
We give the following two obvious corollaries of \Cref{enu-oddrk} and \Cref{compute-delta}.
\begin{corollary}
	Assume that $\dim X=d\geqslant3$ is odd. Let $\xi, \xi'\in[X, \mathrm{BGL}]_{\mathbb{A}^1}$. If $c_j(\xi)=c_j(\xi'), 1\leqslant j\leqslant d$, then $\xi$ and $\xi'$ have the same number of representatives in $\EuScript{V}_d(X)$.
\end{corollary}
In particular, we have the following.
\begin{corollary}
Assume that $\dim X=d\geqslant3$ is odd. Let $\xi, \xi'\in\EuScript{V}_d(X)$, sharing the same total Chern class. If $\xi$ is cancellative, then so is $\xi'$.
\end{corollary}
 For $\xi\in[X, \mathrm{BGL}]_{\mathbb{A}^1}$, we define the translation isomorphism
 \[T_{\xi}:=m(-, \xi)_*: \pi_1({\rm RMap}(X, \mathrm{BGL}), 0)\to\pi_1({\rm RMap}(X, \mathrm{BGL}), \xi),\]
which is the homomorphism induced by $m(-, \xi): {\rm RMap}(X, \mathrm{BGL})_0\to{\rm RMap}(X, \mathrm{BGL})_{\xi}$ introduced in \Cref{abhspcomp}.

The following result says that the map
\[\Delta(c_{n+1}, \xi): [X, \Omega\mathrm{BGL}]_{\mathbb{A}^1}\to[X, \mathrm{K}(\mathbf{K}^{\mathrm{M}}_{n+1}, n)]_{\mathbb{A}^1}\]
is essentially the map induced by $c_{n+1}: \mathrm{BGL}\to\mathrm{K}(\mathbf{K}^{\mathrm{M}}_{n+1}, n+1)$; so the complicated operations we made just transfer things on a general component ${\rm RMap}(X, \mathrm{BGL})_{\xi}$ of the derived mapping space ${\rm RMap}(X, \mathrm{BGL})\in\sset_*$ to the weakly equivalent component ${\rm RMap}(X, \mathrm{BGL})_0$ (component of the canonical base point). In light of \Cref{bgl-comp}, this is not very surprising.
\begin{prop}\label{chern-delta}
	Given $\xi\in[X, \mathrm{BGL}]_{\mathbb{A}^1}$ with $c_{n+1}(\xi)=0, n\geqslant 1$, we have
	\[c_{n+1}\circ T_{\xi}=\Delta(c_{n+1}, \xi): {\rm K}_1(X)\to{\rm H}^{n}(X; \mathbf{K}^{\mathrm{M}}_{n+1}).\]
	
	Similar statement holds with $\mathrm{BSL}$ in place of $\mathrm{BGL}$.
\end{prop}
\begin{proof}
	By \Cref{abhspcomp}, the map $T_{\xi}$ is an isomorphism.
	
	Since $(m^*c_{n+1})(-,\xi)=c_{n+1}m(-,\xi): {\rm RMap}(X, \mathrm{BGL})\to{\rm RMap}(X, \mathrm{K}(\mathbf{K}^{\mathrm{M}}_{n+1}, n+1))$, we see
	\[(m^*c_{n+1}(-,\xi))_*=(c_{n+1})_*T_{\xi}\]
	as maps $\pi_1({\rm RMap}(X, \mathrm{BGL}), 0)\to\pi_1({\rm RMap}(X, \mathrm{K}(\mathbf{K}^{\mathrm{M}}_{n+1}, n+1)), 0)={\rm H}^{n}(X; \mathbf{K}^{\mathrm{M}}_{n+1})$.
	
	On the other hand, by \cref{ch-cartan},
	\[
	m^*c_{n+1}=\sum_{r=0}^{n+1}c_r\times c_{n+1-r}\in[\mathrm{BGL}\times\mathrm{BGL}, \mathrm{K}(\mathbf{K}^{\mathrm{M}}_{n+1}, n+1)]_{\mathbb{A}^1},\]
	so applying ${\rm RMap}(X, -)$ to the two sides we see that the following two maps in $\sset_*$ are homotopic:
	\begin{equation}\label{shiftchern}
	m^*c_{n+1}(-,\xi)\simeq\sum_{r=0}^{n+1}c_r(-)\cdot c_{n+1-r}(\xi).
	\end{equation}
	Here the map $c_r(-)\cdot c_{n+1-r}(\xi)$ is the composite
	\[{\rm RMap}(X, \mathrm{BGL})\xrightarrow{(c_r(-),  c_{n+1-r}(\xi))}{\rm RMap}(X, \mathrm{K}(\mathbf{K}^{\mathrm{M}}_r, r))\times{\rm RMap}(X, \mathrm{K}(\mathbf{K}^{\mathrm{M}}_{n+1-r}, n+1-r))\]
	\[\to{\rm RMap}(X\times X, \mathrm{K}(\mathbf{K}^{\mathrm{M}}_{n+1}, n+1))\xrightarrow{\Delta^*}{\rm RMap}(X, \mathrm{K}(\mathbf{K}^{\mathrm{M}}_{n+1}, n+1)),\]
	where the second coordinate of the first arrow is given by
	\[{\rm RMap}(X, \mathrm{BGL})\to\Delta^0\xrightarrow{c_{n+1-r}(\xi)}{\rm RMap}(X, \mathrm{K}(\mathbf{K}^{\mathrm{M}}_{n+1-r}, n+1-r)),\]
	the second arrow is the obvious map and the third is induced by the diagonal of $X$. Taking $\Omega$ of the above composite we also get a similar description of $(\Omega c_r(-))\cdot c_{n+1-r}(\xi)$, so
	\[\Omega(c_r(-)\cdot c_{n+1-r}(\xi))\simeq(\Omega c_r(-))\cdot c_{n+1-r}(\xi): \Omega{\rm RMap}(X, \mathrm{BGL})\to{\rm RMap}(X, \mathrm{K}(\mathbf{K}^{\mathrm{M}}_{n+1}, n)).\]
	Additions are also preserved: for any $[a], [b]\in[{\rm RMap}(X, \mathrm{BGL}), {\rm RMap}(X, \mathrm{K}(\mathbf{K}^{\mathrm{M}}_{n+1}, n+1))]_{\sset_*}$, we have $[a]_*+[b]_*\simeq([a]+[b])_*: \Omega{\rm RMap}(X, \mathrm{BGL})\to{\rm RMap}(X, \mathrm{K}(\mathbf{K}^{\mathrm{M}}_{n+1}, n))$ (the subscript $*$ refers the effects on loop spaces). Indeed, by facts about $H$-spaces, the sum $[a]+[b]$ is represented the composite
	\[{\rm RMap}(X, \mathrm{BGL})\xrightarrow{\Delta}{\rm RMap}(X, \mathrm{BGL})\times{\rm RMap}(X, \mathrm{BGL})\xrightarrow{a\times b}{\rm RMap}(X, \mathrm{K}(\mathbf{K}^{\mathrm{M}}_{n+1}, n+1))\]
	\[\times{\rm RMap}(X, \mathrm{K}(\mathbf{K}^{\mathrm{M}}_{n+1}, n+1))\xrightarrow{+}{\rm RMap}(X, \mathrm{K}(\mathbf{K}^{\mathrm{M}}_{n+1}, n+1)),\]
	where $\Delta$ refers the diagonal map and ``$+$'' is the $H$-group operation of ${\rm RMap}(X, \mathrm{K}(\mathbf{K}^{\mathrm{M}}_{n+1}, n+1))$; taking $\Omega$ everywhere gives a similar composite for $[a]_*+[b]_*\in[\Omega{\rm RMap}(X, \mathrm{BGL}), {\rm RMap}(X, \mathrm{K}(\mathbf{K}^{\mathrm{M}}_{n+1}, n))]_{\sset_*}$, which says exactly that $[a]_*+[b]_*=([a]+[b])_*\in[\Omega{\rm RMap}(X, \mathrm{BGL}), {\rm RMap}(X, \mathrm{K}(\mathbf{K}^{\mathrm{M}}_{n+1}, n))]_{\sset_*}$. So applying $\pi_0\Omega=\pi_1$ to (\ref{shiftchern}) we find
	\[(m^*c_{n+1}(-,\xi))_*=\sum_{r=0}^{n+1}(\Omega c_r)(-)\cdot c_{n+1-r}(\xi)=\Delta(c_{n+1}, \xi)\]
	by \Cref{compute-delta}. So $c_{n+1}\circ T_{\xi}=(c_{n+1})_*T_{\xi}=\Delta(c_{n+1}, \xi)$.
	
	The $\mathrm{BSL}$ case is also valid by \Cref{bgl-comp}.
\end{proof}
\begin{remark}
	\Cref{chern-delta} and \cite[Chapter I, Lemma 7.3]{GJ} together give an independent proof of \Cref{enu-oddrk} and \Cref{compute-delta}, so we could totally avoid using the method of \cite{JTh65}. We still present both, as the method of \cite{JTh65} is the motivation of our work. What is important to us is the existence of a homomorphism like $\Delta(c_{n+1}, \xi)$ whose cokernel enumerates the lifting set we are interested in, as in \Cref{enu-oddrk} (such enumeration result is much easier to obtain for $c_{n+1}\circ T_{\xi}$, without the complicated constructions in \cite{JTh65}, as we did in the last part of \S 3). 
\end{remark}

We next turn to more refined computation. Following \cite{ADF17}, we denote
\[A_{2n+1}:=k[x_0,\cdots, x_n,y_0,\cdots, y_n]/(\sum_{i=0}^{n}x_iy_i-1)\]
and $Q_{2n+1}:={\rm Spec} (A_{2n+1})$. The projection to the first $n+1$ coordinates gives a morphism $Q_{2n+1}\to\mathbb{A}^{n+1}\setminus0=\mathbb{A}^{n+1}_k\setminus0$, which is a Zariski-locally trivial fibre bundle with fibre $\mathbb{A}^n$ and is thus an $\mathbb{A}^1$-weak equivalence.

Let $A$ be a $k$-algebra. To any $A$-point $(a, b)$ of $Q_{2n+1}$, Suslin associated (systematically and inductively) a matrix $\alpha_{n+1}(a, b)\in{\rm SL}_{2^n}(A)\subset{\rm GL}_{2^n}(A)$ (these matrices have determinant $1$, which is not important here but will be important in later part when we restrict to \emph{oriented} bundles), yielding a $k$-morphism
\[\alpha_{n+1}: Q_{2n+1}\to{\rm GL}_{2^n}\hookrightarrow{\rm GL}.\]
Moreover, Suslin gave a systematic way to reduce the matrix $\alpha_{n+1}(a, b)\in{\rm SL}_{2^n}(A)\subset{\rm GL}_{2^n}(A)$ via elementary matrix operations (though non-explicitly) to an element $\beta_{n+1}(a, b)\in{\rm SL}_{n+1}(A)\subset{\rm GL}_{n+1}(A)$, whose first row can be designated to be $(a_0^{n!},a_1, \cdots, a_n)$ if $a=(a_0,a_1, \cdots, a_n)$---(a special case of) Suslin's famous \emph{$n!$ Theorem} (see \cite{Sus77, Lam06}). This gives a $k$-morphism
\[\beta_{n+1}: Q_{2n+1}\to{\rm SL}_{n+1}\subset{\rm GL}_{n+1}.\]
Note that we have $[\alpha_{n+1}(a, b)]=[\beta_{n+1}(a, b)]\in{\rm GL}(A)/{\rm E}(A)={\rm K}_1(A)$ by construction. The maps here are depicted as follows:
\[\begin{tikzcd}[row sep = 3em, column sep = 3em]
	 &  & {\rm GL}_{n+1} \arrow[d, hookrightarrow] \\
	X\arrow[r, "{(a, b)}"]  &  Q_{2n+1} \arrow[r, "\alpha_{n+1}"] \arrow[ru, "\beta_{n+1}"]          & {\rm GL}_{2^n}.
\end{tikzcd}\]

With the $\mathbb{A}^1$-weak equivalences $\mathbb{A}^{n+1}\setminus0\simeq Q_{2n+1}$, we obtain a morphism
\[\alpha_{n+1}: \mathbb{A}^{n+1}\setminus0\to{\rm GL}\simeq\Omega{\rm BGL}\]
in $\EuScript{H}_{*}^{\mathbb{A}^1}(k)$ (base points suitably chosen). Taking adjunction we get a morphism
\[\alpha_{n+1}^{\sharp}: (\mathbb{P}^1)^{\wedge(n+1)}\simeq\Sigma(\mathbb{A}^{n+1}\setminus0)\to{\rm BGL}\]
in $\EuScript{H}_{*}^{\mathbb{A}^1}(k)$. Under the canonical isomorphism
\[[\mathbb{A}^{n+1}\setminus0, \mathrm{K}(\mathbf{K}^{\mathrm{M}}_{r}, r-1)]_{\mathbb{A}^1, *}\xrightarrow{\cong}[(\mathbb{P}^1)^{\wedge(n+1)}, \mathrm{K}(\mathbf{K}^{\mathrm{M}}_{r}, r)]_{\mathbb{A}^1, *},\]
the class $(\Omega c_{r})(\alpha_{n+1})$ corresponds to $c_{r}(\alpha_{n+1}^{\sharp})$.
\[\mathbb{A}^{n+1}\setminus0\xrightarrow{\alpha_{n+1}}\Omega{\rm BGL}\xrightarrow{\Omega c_{r}}\mathrm{K}(\mathbf{K}^{\mathrm{M}}_{r}, r-1)\leftrightsquigarrow(\mathbb{P}^1)^{\wedge(n+1)}\xrightarrow{\alpha_{n+1}^{\sharp}}{\rm BGL}\xrightarrow{c_r}\mathrm{K}(\mathbf{K}^{\mathrm{M}}_{r}, r)\]
As $c_{r}(\alpha_{n+1}^{\sharp})\in\widetilde{{\rm H}}^r((\mathbb{P}^1)^{\wedge(n+1)}; \mathbf{K}^{\mathrm{M}}_{r})$ and by \cite[Lemma 4.5]{AF14a},
\begin{equation*}
\widetilde{{\rm H}}^{r-1}(\mathbb{A}^{n+1}\setminus0; \mathbf{K}^{\mathrm{M}}_{r})\cong \begin{cases}
0,  & r\neq n+1; \\
(\mathbf{K}^{\mathrm{M}}_{r})_{-(n+1)}({\rm Spec}\ k)=\mathbf{K}^{\mathrm{M}}_{0}({\rm Spec}\ k)=\mathbb{Z}, & r= n+1
\end{cases}
\end{equation*}
(where $\mathbf{M}_{-i}$ denotes the $i$-th contraction of a sheaf $\mathbf{M}$ and we have used \cite[Lemma 2.7]{AF14a}), we obtain the following.
\begin{prop}\label{nicesim}
	$(\Omega c_{r})(\alpha_{n+1})=(\Omega c_{r})(\beta_{n+1})=0$ for $r\neq n+1$.
\end{prop}
\begin{remark}
By \cref{chern-sus} below, we see that $(\Omega c_{n+1})(\alpha_{n+1})$ is $\pm n!$ in $(\mathbf{K}^{\mathrm{M}}_{0})({\rm Spec}\ k)=\mathbb{Z}$.
\end{remark}
It is an easy exercise of scheme theory that for any $k$-algebra $A$, we have ${\rm Hom}_k({\rm Spec}\ A, \mathbb{A}^{n+1}\setminus0)\cong{\rm Um}_{n+1}(A)$ (the set of unimodular rows in $A$ of length $n+1$). Suppose now that $X:={\rm Spec}\ A$ is smooth over $k$. With some effort, one can show the following (see \cite[Theorem 2.1]{Fas11} for a proof).
\begin{prop}\label{factoresu}
	Assume that $2\leqslant\dim A\leqslant n$, then ${\rm H}^n({\rm Spec}\ A; \mathbf{K}^{\mathrm{MW}}_{n+1})\cong[{\rm Spec}\ A, \mathbb{A}^{n+1}\setminus0]_{\mathbb{A}^1}\cong{\rm Um}_{n+1}(A)/{\rm E}_{n+1}(A)$. Here ${\rm E}_{n+1}(A)$ denotes the group of elementary matrices of size $n+1$, with its natural action on rows of length $n+1$.
	
	Moreover, let ${\rm pr}_1: {\rm GL}_{n+1}\to\mathbb{A}^{n+1}\setminus0$ be the projection to the first row, then we obtain a morphism
	\[\psi_{n+1}: \mathbb{A}^{n+1}\setminus0\to\mathbb{A}^{n+1}\setminus0,\ (x_0,x_1, \cdots, x_n)\mapsto(x_0^{n!},x_1, \cdots, x_n)\]
	in $\EuScript{H}_{*}^{\mathbb{A}^1}(k)$, being the composite $\mathbb{A}^{n+1}\setminus0\simeq Q_{2n+1}\xrightarrow{\beta_{n+1}}{\rm GL}_{n+1}\xrightarrow{{\rm pr}_1}\mathbb{A}^{n+1}\setminus0$. Then the induced map
	\[(\psi_{n+1})_*: [{\rm Spec}\ A, \mathbb{A}^{n+1}\setminus0]_{\mathbb{A}^1}\to[{\rm Spec}\ A, \mathbb{A}^{n+1}\setminus0]_{\mathbb{A}^1}\]
	 on the cohomotopy set is given by
	\[(\psi_{n+1})_*([a_0,a_1, \cdots, a_n])=[a_0^{n!},a_1, \cdots, a_n]=\frac{n!}{2}h\cdot[a_0,a_1, \cdots, a_n]\]
	under the above isomorphism, for $(a_0,\cdots, a_n)\in{\rm Um}_{n+1}(A)$ (where $[a_0,\cdots, a_n]$ denotes its class in the orbit set ${\rm Um}_{n+1}(A)/{\rm E}_{n+1}(A)$ and $h=\langle1, -1\rangle$).
\end{prop}
\begin{remark}
Let $\omega: \mathbb{A}^{n+1}\setminus0\to\mathrm{K}(\mathbf{K}^{\mathrm{MW}}_{n+1}, n)$ be the morphism determined by the projection to the first non-trivial ($n$-th) Postnikov level, and let $[\omega]\in{\rm H}^n(\mathbb{A}^{n+1}\setminus0; \mathbf{K}^{\mathrm{MW}}_{n+1})$ be the cohomology class it represents. Then ${\rm H}^n(\mathbb{A}^{n+1}\setminus0; \mathbf{K}^{\mathrm{MW}}_{n+1})$ is a rank-$1$ free  $\mathbf{K}^{\mathrm{MW}}_{0}(k)$-module (again by \cite[Lemma 4.5]{AF14a}) generated by $[\omega]$ and \[\psi_{n+1}^*[\omega]=\frac{n!}{2}h\cdot[\omega].\]
Indeed, let $a=(a_0,\cdots, a_n)\in{\rm Um}_{n+1}(A)$, then $\omega_*(\psi_{n+1})_*([a])=[\omega\circ\psi_{n+1}\circ a]=(\psi_{n+1}^*[\omega])[a]=(\frac{n!}{2}h\cdot[\omega])[a]=\omega_*(\frac{n!}{2}h[a])$, and so $(\psi_{n+1})_*([a])=\frac{n!}{2}h[a]$ (using Postnikov tower argument to see that $\omega_*: [X, \mathbb{A}^{n+1}\setminus0]_{\mathbb{A}^1}\to{\rm H}^n(X; \mathbf{K}^{\mathrm{MW}}_{n+1})$ is a bijection, since $d\leqslant n$).
\[X={\rm Spec}\ A\xrightarrow{a}\mathbb{A}^{n+1}\setminus0\xrightarrow{\psi_{n+1}}\mathbb{A}^{n+1}\setminus0\xrightarrow{\omega}\mathrm{K}(\mathbf{K}^{\mathrm{MW}}_{n+1}, n)\]

In fact, let $V\subset\mathbb{A}^{n+1}\setminus0$ be the closed subvariety defined by $x_1=\cdots=x_n=0$, so $V\cong\mathbb{A}^{1}\setminus0, k(V)\cong k(x_0)$. The homology of the portion
\[\bigoplus_{y\in(\mathbb{A}^{n+1}\setminus0)^{(n-1)}}\mathbf{K}^{\mathrm{MW}}_2(\kappa_y)\to\bigoplus_{y\in(\mathbb{A}^{n+1}\setminus0)^{(n)}}\mathbf{K}^{\mathrm{MW}}_1(\kappa_y)\to\bigoplus_{y\in(\mathbb{A}^{n+1}\setminus0)^{(n+1)}}\mathbf{K}^{\mathrm{MW}}_0(\kappa_y)\]
of the Rost-Schmid complex computes ${\rm H}^n(\mathbb{A}^{n+1}\setminus0; \mathbf{K}^{\mathrm{MW}}_{n+1})$, and one finds that $[x_0]\in\mathbf{K}^{\mathrm{MW}}_1(k(V))$ is indeed a cycle. Moreover, in the localization exact sequence
\[0={\rm H}^n(\mathbb{A}^{n+1}; \mathbf{K}^{\mathrm{MW}}_{n+1})\to{\rm H}^n(\mathbb{A}^{n+1}\setminus0; \mathbf{K}^{\mathrm{MW}}_{n+1})\xrightarrow{\partial}{\rm H}^{n+1}_{\{0\}}(\mathbb{A}^{n+1}; \mathbf{K}^{\mathrm{MW}}_{n+1})=\mathbf{K}^{\mathrm{MW}}_0(k)\to0\]
we have $\partial[x_0]=\langle1\rangle\in\mathbf{K}^{\mathrm{MW}}_0(k)$, so ${\rm H}^n(\mathbb{A}^{n+1}\setminus0; \mathbf{K}^{\mathrm{MW}}_{n+1})=\mathbf{K}^{\mathrm{MW}}_0(k)\cdot[x_0]$. By the relation $[a^n]=n_{\epsilon}[a]$ in $\mathbf{K}^{\mathrm{MW}}_1$ we see $[x_0^{n!}]=\frac{n!}{2}h[x_0]$ so $\partial[x_0^{n!}]=\frac{n!}{2}h$. Thus
 \[\psi_{n+1}^*[x_0]=\frac{n!}{2}h\cdot[x_0].\]
\end{remark}

For an abelian group $H$ and an integer $m$, we write $H/m:=H/mH$ for the quotient abelian group. We say that $H$ is \emph{$m$-divisible} if $H/m=0$, i.e. if $mH=H$.
\begin{theorem}\label{divisability-cancel}
	Let $k$ be a perfect field with ${\rm char}(k)\neq 2$. Assume that $A$ is a smooth $k$-algebra of odd Krull dimension $d\geqslant3$. Let $X={\rm Spec}\ A$, $\xi\in[X, \mathrm{BGL}]_{\mathbb{A}^1}$.
	\begin{enumerate}[label=\emph{(\arabic*)}]
		\item We have $d!\cdot{\rm H}^d(X; \mathbf{K}^{\mathrm{M}}_{d+1})\subset{\rm im}\Delta(c_{d+1}, \xi)$.
		\item If ${\rm H}^d(X; \mathbf{K}^{\mathrm{M}}_{d+1})$ is $d!$-divisible, then ${\rm coker}\Delta(c_{d+1}, \xi)=0$. In this case, any rank $d$ vector bundle is cancellative. Moreover,  the map
		\[(c_{d+1})_*: \pi_1({\rm RMap}(X, \mathrm{BGL}), \xi)\to\pi_1({\rm RMap}(X, \mathrm{K}(\mathbf{K}^{\mathrm{M}}_{d+1}, d+1)), 0)={\rm H}^{d}(X; \mathbf{K}^{\mathrm{M}}_{d+1})\]
		is surjective for every $\xi\in[X, \mathrm{BGL}]_{\mathbb{A}^1}$.
	\end{enumerate}
\end{theorem}
\begin{proof}
	\begin{enumerate}[label=(\arabic*)]
		\item By \cite[Proposition 2.6]{AF14a}, there is an exact sequence
		\[0\to\mathbf{I}^{d+2}\to\mathbf{K}^{\mathrm{MW}}_{d+1}\to\mathbf{K}^{\mathrm{M}}_{d+1}\to0,\]
		where $\mathbf{I}=\ker(\mathbf{GW}=\mathbf{K}^{\mathrm{MW}}_0\to\mathbf{K}^{\mathrm{M}}_0)\cong\eta\mathbf{K}^{\mathrm{MW}}_1$ is the fundamental ideal and $\mathbf{I}^{d+2}$ its power. Since $X$ has dimension $d$, the cohomology ${\rm H}^j(X; \mathbf{A})$ vanishes for $j>d$ for any ${\bf A}\in\ab_{k}^{\mathbb{A}^1}$ (by the Rost-Schmid complex computation); so ${\rm H}^{d+1}(X; \mathbf{I}^{d+2})=0$, and hence the map
		\[\tau: {\rm H}^d(X; \mathbf{K}^{\mathrm{MW}}_{d+1})\to{\rm H}^d(X; \mathbf{K}^{\mathrm{M}}_{d+1})\]
		given by reducing coefficients is a surjection. So we can write any element of ${\rm H}^d(X; \mathbf{K}^{\mathrm{M}}_{d+1})$ as $\tau([a, b])$ with $(a, b)\in Q_{2d+1}(A), a=(a_0,a_1, \cdots, a_d)\in{\rm Um}_{d+1}(A)$; here we are identifying a class $[a, b]\in[X, Q_{2d+1}]_{\mathbb{A}^1}$ with $[a]\in[X, \mathbb{A}^{d+1}\setminus 0]_{\mathbb{A}^1}$.
		
		We will show that the following equality holds:
		\begin{equation}\label{chern-sus-d!}
		d!\cdot\tau([a, b])=\pm\Delta(c_{d+1}, \xi)([\beta_{d+1}(a, b)]).
		\end{equation}
		Indeed, since $h=\eta\cdot[-1]+2$ becomes $2$ in $\mathbf{K}^{\mathrm{M}}_*$, yielding that $\tau(\frac{d!}{2}h)=d!$, so by \Cref{factoresu},
		\[d!\cdot\tau([a, b])=\tau(d!\cdot[a, b])=\tau((\psi_{d+1})_*[a, b])=\tau(({\rm pr}_1)_*[\beta_{d+1}(a, b)])=\tau([a_0^{d!},a_1, \cdots, a_d]).\]
		While \Cref{nicesim} tells
		\[\Delta(c_{d+1}, \xi)([\beta_{d+1}(a, b)])=(\Omega c_{d+1})([\beta_{d+1}(a, b)]).\]
		We are thus reduced to showing that
		\begin{equation}\label{chern-sus}
		\tau([a_0^{d!},a_1, \cdots, a_d])=\tau(({\rm pr}_1)_*[\beta_{d+1}(a, b)])=\pm(\Omega c_{d+1})([\beta_{d+1}(a, b)]).
		\end{equation}
		This follows by taking $\beta=\beta_{d+1}(a, b)$ in the identity
		\begin{equation}\label{chern-row1}
		\tau([(1,0,\cdots,0)\cdot\beta])=\pm(\Omega c_{d+1})([\beta]), \text{for}\ \beta\in{\rm SL}_{d+1}(A)\subset{\rm SL}(A).
		\end{equation}
		To see the latter, note that $(\mathbb{A}^{d+1}\setminus0)[d]=\mathrm{K}(\mathbf{K}^{\mathrm{MW}}_{d+1}, d)$, by \Cref{seqmotpt} (4), we have an $\mathbb{A}^1$-homotopy fibre sequence
		\[\mathrm{K}(\mathbf{K}^{\mathrm{MW}}_{d+1}, d)\rightarrow (\mathrm{BGL}_{d})^{[d]}\xrightarrow{p_{d}} \mathrm{BGL}_{d+1},\]
		where $(\mathrm{BGL}_{d})^{[d]}$ is the $d$-th Moore-Postnikov level for the map $\mathrm{BGL}_{d}\to\mathrm{BGL}_{d+1}$. By the naturality statement of (dual of) \cite[Proposition 6.5.3]{Hov}, we get a map of $\mathbb{A}^1$-homotopy fibre sequences
		\[
		\begin{tikzcd}
		\mathrm{GL}_{d+1} \arrow[r,"{\rm pr}_1"]\arrow[d, equal] & \mathbb{A}^{d+1}\setminus0 \ar[d, "\tau'"] \ar[r]& \mathrm{BGL}_{d}  \arrow[d]  \ar[r] & \mathrm{BGL}_{d+1}\arrow[d, equal] \\
		\mathrm{GL}_{d+1} \arrow[r]\ & \mathrm{K}(\mathbf{K}^{\mathrm{MW}}_{d+1}, d) \ar[r] & (\mathrm{BGL}_{d})^{[d]}  \ar[r, "p_{d}"] & \mathrm{BGL}_{d+1}.
		\end{tikzcd}\]
		In this diagram, the map $\mathrm{GL}_{d+1}\xrightarrow{{\rm pr}_1}\mathbb{A}^{d+1}\setminus0$ is indeed ``projection to the first row'', since after applying $[U, -]_{\mathbb{A}^1}$, the induced map on homotopy is given by the natural action of $\mathrm{GL}_{d+1}(U)/{\rm E}_{d+1}(U)$ on the class of the base point $(1,0,\cdots,0)\in(\mathbb{A}^{d+1}\setminus0)(k)$.
		
		Thus $\Omega e_{d+1}=\tau'\circ{\rm pr}_1: \mathrm{GL}_{d+1}\to\mathrm{K}(\mathbf{K}^{\mathrm{MW}}_{d+1}, d)$ (in $\EuScript{H}_{*}^{\mathbb{A}^1}(k)$), composing with the maps $X\xrightarrow{\beta}\mathrm{GL}_{d+1}$ and $\mathrm{K}(\mathbf{K}^{\mathrm{MW}}_{d+1}, d)\to\mathrm{K}(\mathbf{K}^{\mathrm{M}}_{d+1}, d)$ we find that, after reducing coefficients, $\Omega e_{d+1}([\beta])$ equals to $\tau([(1,0,\cdots,0)\cdot\beta])$. While by \cite[Theorem 1 and Proposition 5.8]{AF16a}, our $k$-invariant $e_{d+1}$ is the (universal) Euler class (up to a unit in ${\rm GW}(k)$), whose reduction coincides with the universal Chern class $c_{d+1}$ (up to sign, as a unit in ${\rm GW}(k)=\mathbf{K}^{\mathrm{MW}}_0(k)$ is mapped to a unit in $\mathbf{K}^{\mathrm{M}}_0(k)\cong\mathbb{Z}$), establishing our \cref{chern-row1}.
	\end{enumerate}

Statement (2) then follows easily, except for the last statement, for which we apply \cite[Chapter I, Lemma 7.3]{GJ} to the homotopy fiber sequence
\[{\rm RMap}(X, E)\xrightarrow{q}{\rm RMap}(X, \mathrm{BGL})\xrightarrow{c_{d+1}}{\rm RMap}(X, \mathrm{K}(\mathbf{K}^{\mathrm{M}}_{d+1}, d+1))\]
induced by the $\mathbb{A}^1$-homotopy fiber sequence (\ref{1stfibseq}).
\end{proof}
\begin{remark}\label{d!-div}
If $k$ is algebraically closed, then using the Rost-Schmid complex (see \cite{CF17, Fas18} for some nice expositions on related notions and results), one easily finds that ${\rm H}^d(X; \mathbf{K}^{\mathrm{M}}_{d+1})$ is $d!$-divisible, as it is a quotient of the direct sum of groups of the form $\mathbf{K}^{\mathrm{M}}_{1}(\kappa_x)=\kappa_x^{\times}=k^{\times}$ which is $d!$-divisible ($x$ ranges over closed points of $X$). Of course, for the same reason, we have the more refined result that ${\rm H}^d(X; \mathbf{K}^{\mathrm{M}}_{d+1})$ is $d!$-divisible if $\kappa_x^{\times}$ is $d!$-divisible (i.e. $\kappa_x=(\kappa_x)^{d!}$) for all closed points $x\in X$.

More generally, ${\rm H}^d(X; \mathbf{K}^{\mathrm{M}}_{d+1})$ is $d!$-divisible if $k$ has cohomological dimension at most $1$ and $d!\in k^{\times}$ (see for instance \cite[Theorems 1.4 and 2.2]{Fas15}), as a consequence of Voevodsky's confirmation of the motivic Bloch-Kato conjecture (or \emph{norm residue isomorphism theorem}), a highly non-trivial result.

Here is an example not covered by this general criterion. We take $k=\mathbb{R}$, if a smooth affine $\mathbb{R}$-variety $X={\rm Spec}(A)$ of odd dimension $d$ has no $\mathbb{R}$-point (e.g. if $A=\mathbb{R}[x,y,z,w]/(x^2+y^2+z^2+w^2+1)$), then $\kappa_x\cong\mathbb{C}$ for all closed points $x\in X$. Hence on such an $X$, all rank $d$ vector bundles are cancellative. Note that for this class of examples, Bhatwadekar's result \cite{Bha03} can't tell the cancellation property of such modules, since ${\rm c.d.}(\mathbb{R})=\infty$ and $\mathbb{R}$ is far from being a $C_1$-field.
\end{remark}

\vspace{5mm}

Now we treat the case when $n=d$ is even. To still get a principal $\mathbb{A}^1$-homotopy fiber sequence, we need to restrict ourselves to the case of oriented vector bundles---namely those classified by homotopy classes of maps to $\mathrm{BSL}$ or $\mathrm{BSL}_d$ (since $\pi_{1}^{\mathbb{A}^1}\mathrm{BSL}_d=0$, as opposed to the fact that $\pi_{1}^{\mathbb{A}^1}\mathrm{BGL}_d=\mathbb{G}_m$).

Note first that our discussion from \S 4 up to \S 6 here are still valid if we replace $\mathrm{GL}$ with $\mathrm{SL}$ everywhere, essentially because we also have the $\mathbb{A}^1$-homotopy fiber sequence
\[\mathbb{A}^{n+1}\setminus{0}\to\mathrm{BSL}_{n}\to\mathrm{BSL}_{n+1}.\]
The only difference is that the obstruction class (namely the $k$-invariant $\theta$) is different from the case when $n=d$ is odd.

We now identify the $k$-invariant $\theta$. By the functoriality of the Moore-Postnikov tower, applied to the square
\[
\begin{tikzcd}
\mathrm{BSL}_{d} \ar[r]\ar[d, equal] &\mathrm{BSL}_{d+1}\ar[d]\\
\mathrm{BSL}_{d} \ar[r] &\mathrm{BSL}
\end{tikzcd}\]
(factoring the rows up to the first non-trivial stage $E', E$) we have the following map of $\mathbb{A}^1$-homotopy fibre sequences (when deleting the last column)
\[
\begin{tikzcd}
\mathrm{K}(\mathbf{K}^{\mathrm{MW}}_{d+1}, d) \ar[r]\arrow[d, equal] & E' \ar[r]\ar[d]&\mathrm{BSL}_{d+1}\arrow[r,"e_{d+1}"]\arrow[d, "s_{d+1}"]  & \mathrm{K}(\mathbf{K}^{\mathrm{MW}}_{d+1}, d+1) \arrow[d, equal]  \ar[r, "\tau"] & \mathrm{K}(\mathbf{K}^{\mathrm{M}}_{d+1}, d+1)\arrow[d, equal] \\
\mathrm{K}(\mathbf{K}^{\mathrm{MW}}_{d+1}, d) \ar[r]  & E \ar[r] &\mathrm{BSL}\arrow[r,"\theta"] & \mathrm{K}(\mathbf{K}^{\mathrm{MW}}_{d+1}, d+1)  \ar[r, "\tau"] & \mathrm{K}(\mathbf{K}^{\mathrm{M}}_{d+1}, d+1).
\end{tikzcd}\]
So $\tau e_{d+1}=c_{d+1}=\tau\theta s_{d+1}$ (as Chern classes stabilize, we can write $c_{d+1}=\tau\theta$).

Note that as $\dim X=d$, we have ${\rm SK}_1(X)=[X, \Omega\mathrm{BSL}]_{\mathbb{A}^1}=[X, \mathrm{SL}]_{\mathbb{A}^1}\stackrel{(s_{d+1})_*}{=\joinrel=\joinrel=\joinrel=}[X, \mathrm{SL}_{d+1}]_{\mathbb{A}^1}$. We thus have a commutative diagram (since $\tau$ is induced by the homomorphism $\mathbf{K}^{\mathrm{MW}}_{d+1}\to\mathbf{K}^{\mathrm{M}}_{d+1}$, we can move $\tau$ out; cf. \cite[(1.1)]{JTh65})
\[
\begin{tikzcd}[column sep=4em]
{[X, \mathrm{SL}]_{\mathbb{A}^1}} \ar[r, "{\Delta(\theta, \xi)}"]\ar[d, equal] &{{\rm H}^d(X; \mathbf{K}^{\mathrm{MW}}_{d+1})}\ar[d, "\tau"]\\
{[X, \mathrm{SL}]_{\mathbb{A}^1}} \ar[r, "{\Delta(c_{d+1}, \xi)}"] & {{\rm H}^d(X; \mathbf{K}^{\mathrm{M}}_{d+1}).}
\end{tikzcd}\]
Again by the exact sequence
\[0\to\mathbf{I}^{d+2}\to\mathbf{K}^{\mathrm{MW}}_{d+1}\to\mathbf{K}^{\mathrm{M}}_{d+1}\to0,\]
we see that the right vertical map $\tau: {\rm H}^d(X; \mathbf{K}^{\mathrm{MW}}_{d+1})\to{\rm H}^d(X; \mathbf{K}^{\mathrm{M}}_{d+1})$ is surjective (since $X$ has $\mathbb{A}^1$-cohomological dimension at most $d$).

If we further assume that the $2$-cohomological dimension of our base field $k$ (perfect and ${\rm char} (k)\neq 2$) is at most $2$: ${\rm c.d.}_2(k)\leqslant 2$, then by (the proof of) \cite[Theorem 2.1]{Fas15} (using Gersten-Witt complex of $X={\rm Spec}\ A$ and assuming $d\geqslant 3$), we have ${\rm H}^d(X; \mathbf{I}^{d+2})=0$. So in this case, the right vertical map $\tau: {\rm H}^d(X; \mathbf{K}^{\mathrm{MW}}_{d+1})\to{\rm H}^d(X; \mathbf{K}^{\mathrm{M}}_{d+1})$ is an isomorphism.

Finally, we are able to give results similar to those in the odd dimensional case (the proof of which is obtained by suitably changing notation).

Similarly as before, we have the induced map on $\mathbb{A}^1$-homotopy classes
\[\varphi_*: [X, \mathrm{BSL}_{n}]_{\mathbb{A}^1}\to[X, \mathrm{BSL}]_{\mathbb{A}^1}\]
and want to enumerate its fibers. We also use $\EuScript{V}_n^{\circ}(X)$ to denote the set of isomorphism classes of rank $n$ oriented vector bundle over $X$.
\begin{prop}\label{enu-evenrk}
	Assume that the base field $k$ is perfect and ${\rm char} (k)\neq 2$ with ${\rm c.d.}_2(k)\leqslant 2$. Let $d\geqslant 2$ be even, let $\xi$ be a stable oriented vector bundle over $X$, whose classifying map is still denoted $\xi: X\to\mathrm{BSL}$. Then there is a bijection
	\[\varphi_*^{-1}(\xi)\cong{\rm coker}\big({\rm SK}_1(X)=[X, \Omega\mathrm{BSL}]_{\mathbb{A}^1}\xrightarrow{\Delta(c_{d+1}, \xi)}[X, \mathrm{K}(\mathbf{K}^{\mathrm{M}}_{d+1}, d)]_{\mathbb{A}^1}={\rm H}^{d}(X; \mathbf{K}^{\mathrm{M}}_{d+1})\big).\]
\end{prop}

\begin{prop}
	Assume that the condition of \Cref{enu-evenrk} is satisfied.
	\begin{enumerate}[label=\emph{(\arabic*)}]
		\item Let $\xi, \xi'\in[X, \mathrm{BSL}]_{\mathbb{A}^1}$. If $c_j(\xi)=c_j(\xi'), 1\leqslant j\leqslant d$, then $\xi$ and $\xi'$ have the same number of representatives in $\EuScript{V}_d^{\circ}(X)$.
		\item Let $\xi, \xi'\in\EuScript{V}_d^{\circ}(X)$, sharing the same total Chern class. If $\xi$ is cancellative, then so is $\xi'$.
	\end{enumerate}
\end{prop}
\begin{theorem}\label{divisability-cancel-even}
Assume that the base field $k$ is perfect and ${\rm char} (k)\neq 2$ with ${\rm c.d.}_2(k)\leqslant 2$. Assume that $A$ is a smooth $k$-algebra of even Krull dimension $d\geqslant 2$. Let $X={\rm Spec}\ A$, $\xi\in[X, \mathrm{BSL}]_{\mathbb{A}^1}$.
\begin{enumerate}[label=\emph{(\arabic*)}]
	\item We have $d!\cdot{\rm H}^d(X; \mathbf{K}^{\mathrm{M}}_{d+1})\subset{\rm im}\Delta(c_{d+1}, \xi)$. So there is a surjective homomorphism
	\[{\rm H}^d(X; \mathbf{K}^{\mathrm{M}}_{d+1})/d!\twoheadrightarrow{\rm coker}\Delta(c_{d+1}, \xi).\]
	\item If ${\rm H}^d(X; \mathbf{K}^{\mathrm{M}}_{d+1})$ is $d!$-divisible, then ${\rm coker}\Delta(c_{d+1}, \xi)=0$. In this case, any rank $d$ oriented vector bundle is cancellative. Moreover,  the map
	\[(c_{d+1})_*: \pi_1({\rm RMap}(X, \mathrm{BSL}), \xi)\xrightarrow{\theta_*}\pi_1({\rm RMap}(X, \mathrm{K}(\mathbf{K}^{\mathrm{MW}}_{d+1}, d+1)), 0)={\rm H}^{d}(X; \mathbf{K}^{\mathrm{MW}}_{d+1})\xrightarrow{\tau}{\rm H}^{d}(X; \mathbf{K}^{\mathrm{M}}_{d+1})\]
	is surjective for every $\xi\in[X, \mathrm{BSL}]_{\mathbb{A}^1}$.
\end{enumerate}
\end{theorem}
\begin{remark}
For the even rank case, our assumption on the $2$-cohomological dimension of the base field $k$ cannot be omitted in order to get ${\rm H}^d(X; \mathbf{I}^{d+2})=0$ so that $\tau: {\rm H}^d(X; \mathbf{K}^{\mathrm{MW}}_{d+1})\to{\rm H}^d(X; \mathbf{K}^{\mathrm{M}}_{d+1})$ is an isomorphism: if we take $A=\mathbb{R}[x, y, z]/(x^2+y^2+z^2-1)$, then ${\rm H}^d({\rm Spec}\ A; \mathbf{I}^{d+2})\neq 0$ (note that ${\rm c.d.}_2(\mathbb{R})=\infty$); and the rank $2$ stably free $A$-module $A^3/A(x, y, z)$ is indeed not free.

On the other hand, quite a lot of fields satisfy our assumption, e.g. any finite field (of odd characteristic), any algebraically closed field, or any field of the form $L(t)$ or $L(t_1, t_2)$ for an algebraically closed field $L$ with ${\rm char}(L)=0$.
\end{remark}

\section{Application to vector bundles below critical rank}

Assume that the base field $k$ is perfect and ${\rm char} (k)\neq 2$, $A$ is a smooth affine $k$-algebra of Krull dimension $d\geqslant 3$, and $X={\rm Spec}(A)$. Let $\xi$ be a stable oriented vector bundle over $X$, whose classifying map is still denoted $\xi: X\to\mathrm{BSL}$. We will now investigate the isomorphism classes of rank $d-1$ oriented vector bundles that are stably equivalent to the given $\xi$.

If $d\geqslant 4$ is even, we have the $\mathbb{A}^1$-homotopy fiber sequence
\[F_{d-1}\to\mathrm{BSL}_{d-1}\xrightarrow{\varphi=\varphi_{d-1}}\mathrm{BSL}.\]
We consider the two-stage Moore-Postnikov factorization (\Cref{MotMP}) of the map $\varphi: \mathrm{BSL}_{d-1}\to\mathrm{BSL}$,
\begin{equation}\label{MPfact}
\begin{tikzcd}
F=F_{d-1} \ar[r] & \mathrm{BSL}_{d-1} \arrow[d, "q'"'] \\
\mathrm{K}(\pi_{d}^{\mathbb{A}^1}F_{d-1}, d) \ar[r] & E' \arrow[d, "q"'] \\
\mathrm{K}(\mathbf{K}^{\mathrm{M}}_{d}, d-1) \ar[r] & E \arrow[r, "\theta'"] \arrow[d, "p"'] & \mathrm{K}(\pi_{d}^{\mathbb{A}^1}F_{d-1}, d+1) \\
X  \ar[r, "\xi"] \ar[ru, "\xi_E"] &  \mathrm{BSL} \arrow[r, "\theta"]          & \mathrm{K}(\mathbf{K}^{\mathrm{M}}_{d}, d).
\end{tikzcd}
\end{equation}
By the properties listed in \Cref{MotMP}, it is easy to see that the map $q': \mathrm{BSL}_{d-1}\to E'$ induces a bijection
\[q'_*: [X, \mathrm{BSL}_{d-1}]_{\mathbb{A}^1}\to[X, E']_{\mathbb{A}^1}.\]
We thus need only to find under what conditions, the maps
\[p_*: [X, E]_{\mathbb{A}^1}\to[X, \mathrm{BSL}]_{\mathbb{A}^1}\]
and
\[q_*: [X, E']_{\mathbb{A}^1}\to[X, E]_{\mathbb{A}^1}\]
are injections.

As before, we have the following commutative diagram
\[
\begin{tikzcd}
\mathrm{BSL}_{d}\arrow[r,"e_d"]\arrow[d, "s_d"']  & \mathrm{K}(\mathbf{K}^{\mathrm{MW}}_d, d) \arrow[d, "\tau"] \\
\mathrm{BSL}\arrow[r,"\theta"] & \mathrm{K}(\mathbf{K}^{\mathrm{M}}_d, d).
\end{tikzcd}\]
Thus $\theta s_d=\tau e_d=c_d$ by \cite[Example 5.2 and Proposition 5.8]{AF16a} as before.

On the other hand, by Serre's splitting theorem, we can write $\xi=(s_d)_*([\xi_d])$ for some $[\xi_d]\in[X, \mathrm{BSL}_d]_{\mathbb{A}^1}$. So $\theta_*([\xi])=(\theta s_d)_*[\xi_d]=(\tau e_d)_*[\xi_d]=c_d(\xi_d)=c_d(\xi)$ (as Chern classes stabilizes), we can write $\theta=c_d$.

Note that whenever $\xi$ is represented by a rank $d-1$ vector bundle, we will have $c_d(\xi)=0$. So by \Cref{lift=ck} with \Cref{compute-delta} and \Cref{computedeltaform} applied to the case $n=d-1$, we obtain the following description of $p_*^{-1}([\xi])$.
\begin{prop}\label{Delt-evend-1}
	Let $k$ be a perfect field with ${\rm char}(k)\ne 2$, let $A$ be a smooth affine $k$-algebra of even Krull dimension $d\geqslant 4$, and $X={\rm Spec}(A)$, let $\xi$ be a stable oriented vector bundle over $X$, whose classifying map is still denoted $\xi: X\to\mathrm{BSL}$. If $\xi$ is represented by a rank $d-1$ vector bundle, there is a bijection
	\[p_*^{-1}(\xi)\longleftrightarrow{\rm coker}(\Delta(c_{d}, \xi)).\]
	The homomorphism $\Delta(c_d, \xi): {\rm SK}_1(X)\to{\rm H}^{d-1}(X; \mathbf{K}^{\mathrm{M}}_{d})$ is given as follows: for $\beta\in{\rm SK}_1(X)$,
	\[\Delta(c_d, \xi)\beta=(\Omega c_d)(\beta)+\sum_{r=1}^{d-1}((\Omega c_{r})(\beta))\cdot c_{d-r}(\xi).\]
\end{prop}

Below, for a sheaf of abelian groups $\mathbf{K}$ and an integer $m$, we denote by $\mathbf{K}/m:=\mathbf{K}/m\mathbf{K}$ for the mod-$m$ quotient sheaf and $\leftidx{_m}{\mathbf{K}}:=\ker(\mathbf{K}\xrightarrow{m}m\mathbf{K})$ for the subsheaf killed by $m$.  Since the contraction functor $(-)_{-1}: \ab_{k}^{\mathbb{A}^1}\to\ab_{k}^{\mathbb{A}^1}$ is exact, the constructions of the mod-$m$ quotient and the subsheaves killed by $m$ are preserved by (iterated) contractions. We also write $\mu_m$ for the \'etale sheaf of $m$-th roots of $1$, and $\mu_m^{\otimes n}$ for its $n$-th tensor power. Denote $\bar{\mathbf{I}}^j:=\mathbf{I}^j/\mathbf{I}^{j+1}\cong\mathbf{K}^{\mathrm{M}}_{j}/2$. The proof of the following result is adapted from that of \cite[Proposition 6.1]{FRS}.
\begin{prop}\label{divofcohmil}
	Assume that the base field $k$ is algebraically closed, $X={\rm Spec}(A)$ is a connected smooth affine $k$-scheme of dimension $d\geqslant 3$. Then
	\[{\rm H}^{d-1}(X; \mathbf{K}^{\mathrm{M}}_{d})/m=0\]
	
	In particular, if ${\rm char}(k)=0$ or ${\rm char}(k)\geqslant d$, then ${\rm H}^{d-1}(X; \mathbf{K}^{\mathrm{M}}_{d})$ is $(d-1)!$-divisible.
\end{prop}
\begin{proof}
	Writing $m=\ell_1^{r_1}\cdots\ell_s^{r_s}$, with each $\ell_i$ prime, $\ell_i\ne{\rm char}(k)$ and $r_i\in\mathbb{N}$, we see that it suffices to consider the case $m=\ell^r$. Let $\ell$ be a prime number and $\ell\ne{\rm char}(k)$, let $r\in\mathbb{N}$. Consider the short exact sequences
	\[0\to\leftidx{_{\ell^r}}{\mathbf{K}^{\mathrm{M}}_{d}}\to\mathbf{K}^{\mathrm{M}}_{d}\xrightarrow{\ell^r}\ell^r\mathbf{K}^{\mathrm{M}}_{d}\to0\]
	and
	\[0\to\ell^r\mathbf{K}^{\mathrm{M}}_{d}\to\mathbf{K}^{\mathrm{M}}_{d}\to\mathbf{K}^{\mathrm{M}}_{d}/\ell^r\to0.\]
	Using the Rost-Schmid complex we see that ${\rm H}^d(X; \leftidx{_{\ell^r}}{\mathbf{K}^{\mathrm{M}}_{d}})$ is a quotient of a direct sum of groups of the form $\leftidx{_{\ell^r}}{\mathbf{K}^{\mathrm{M}}_0(\kappa_x)}\cong\leftidx{_{\ell^r}}{\mathbb{Z}}=0$ (over all closed points $x\in X^{(d)}$) as $\mathbb{Z}$ is torsion-free, thus ${\rm H}^d(X; \leftidx{_{\ell^r}}{\mathbf{K}^{\mathrm{M}}_{d}})=0$. We then have exact sequences
	\[{\rm H}^{d-1}(X; \mathbf{K}^{\mathrm{M}}_{d})\to{\rm H}^{d-1}(X; \ell^r\mathbf{K}^{\mathrm{M}}_{d})\to{\rm H}^d(X; \leftidx{_{\ell^r}}{\mathbf{K}^{\mathrm{M}}_{d}})=0\]
	and
	\[{\rm H}^{d-1}(X; \ell^r\mathbf{K}^{\mathrm{M}}_{d})\to{\rm H}^{d-1}(X; \mathbf{K}^{\mathrm{M}}_{d})\to{\rm H}^{d-1}(X; \mathbf{K}^{\mathrm{M}}_{d}/\ell^r).\]
	Splicing together we get an exact sequence
	\[{\rm H}^{d-1}(X; \mathbf{K}^{\mathrm{M}}_{d})\xrightarrow{\ell^r}{\rm H}^{d-1}(X; \mathbf{K}^{\mathrm{M}}_{d})\to{\rm H}^{d-1}(X; \mathbf{K}^{\mathrm{M}}_{d}/\ell^r)\]
	and hence
	\[0\to{\rm H}^{d-1}(X; \mathbf{K}^{\mathrm{M}}_{d})/\ell^r\to{\rm H}^{d-1}(X; \mathbf{K}^{\mathrm{M}}_{d}/\ell^r).\]
	Therefore to prove ${\rm H}^{d-1}(X; \mathbf{K}^{\mathrm{M}}_{d})/\ell^r=0$, it suffices to prove that ${\rm H}^{d-1}(X; \mathbf{K}^{\mathrm{M}}_{d}/\ell^r)=0$.
	
	For $j, n\in\mathbb{N}$, let $\mathscr{H}^j(n)=(\mathbf{R}^j\iota_*)\mu_{\ell^r}^{\otimes n}$ be the Zariski sheaf associated to the presheaf $U\mapsto{\rm H}^j_{\rm \acute{e}t}(U; \mu_{\ell^r}^{\otimes n})$ (where $\iota$ is the inclusion of the Zariski site into the \'etale site). We have the biregular \emph{Bloch-Ogus spectral sequence} (\cite{BOSS74}; it is one incarnation of the \emph{Leray spectral sequence})
	\[{\rm E}_2^{ij}={\rm H}^i_{\rm Zar}(X; \mathscr{H}^j(n))\Longrightarrow{\rm H}^{i+j}_{\rm \acute{e}t}(X; \mu_{\ell^r}^{\otimes n}).\]
	The term ${\rm E}_2^{ij}={\rm H}^i_{\rm Zar}(X; \mathscr{H}^j(n))$ can be computed as the $i$-th cohomology of the \emph{Gersten complex}
	\[{\rm H}_{\rm \acute{e}t}^j(\kappa_{\eta}; \mu_{\ell^r}^{\otimes n})\to\cdots\to\bigoplus_{x\in X^{(i)}}{\rm H}_{\rm \acute{e}t}^{j-i}(\kappa_{x}; \mu_{\ell^r}^{\otimes n-i})\to\cdots,\]
	where $\eta\in X^{(0)}$ is the generic point of $X$.
	
	Since $k=\bar{k}$, the cohomological dimension ${\rm c.d.}(\kappa_x)\leqslant d-\dim(\mathscr{O}_{X, x})$ (\cite[\S4.2, Proposition 11]{Ser94}), we see ${\rm H}_{\rm \acute{e}t}^{j-i}(\kappa_x; \mu_{\ell^r}^{\otimes n-i})=0$ for $x\in X^{(i)}$ if $j-i>d-i$. Thus ${\rm E}_2^{ij}={\rm H}^i_{\rm Zar}(X; \mathscr{H}^j(n))=0$ if $i>d=\dim\ X$ or $j>d$ or $i>j$. Hence in the filtration of the converging term ${\rm H}^{2d-1}_{\rm \acute{e}t}(X; \mu_{\ell^r}^{\otimes n})$, the only (possibly) non-trivial term is ${\rm E}_2^{d-1, d}={\rm H}^{d-1}_{\rm Zar}(X; \mathscr{H}^d(n))$. While by \cite[Chapter VI, Theorem 7.2]{Milne80}, ${\rm H}^{2d-1}_{\rm \acute{e}t}(X; \mu_{\ell^r}^{\otimes n})=0$ since $X$ is affine over $k=\bar{k}$. Thus ${\rm E}_2^{d-1, d}={\rm H}^{d-1}_{\rm Zar}(X; \mathscr{H}^d(n))=0$ as well.
	
	There is a commutative diagram (\cite[Theorem 2.3]{Bloch80})
	\[
	\begin{tikzcd}
	\bigoplus_{x\in X^{(d-2)}}{\rm K}_2^{\rm M}(\kappa_{x})/\ell^r \ar[r]\ar[d] & \bigoplus_{x\in X^{(d-1)}}{\rm K}_1^{\rm M}(\kappa_{x})/\ell^r  \ar[r]\ar[d] & \bigoplus_{x\in X^{(d)}}{\rm K}_0^{\rm M}(\kappa_{x})/\ell^r\ar[d]  \\
	\bigoplus_{x\in X^{(d-2)}}{\rm H}_{\rm \acute{e}t}^{2}(\kappa_{x}; \mu_{\ell^r}^{\otimes 2}) \ar[r] & \bigoplus_{x\in X^{(d-1)}}{\rm H}_{\rm \acute{e}t}^{1}(\kappa_{x}; \mu_{\ell^r})  \ar[r] & \bigoplus_{x\in X^{(d)}}{\rm H}_{\rm \acute{e}t}^{0}(\kappa_{x}; \mathbb{Z}/\ell^r),
	\end{tikzcd}\]
	where the vertical maps are isomorphisms by \cite{MerSus} (or Voevodsky's confirmation of the motivic Bloch-Kato conjecture). The homology of the middle terms in the two rows compute ${\rm H}^{d-1}(X; \mathbf{K}^{\mathrm{M}}_{d}/\ell^r)$ and respectively ${\rm H}^{d-1}_{\rm Zar}(X; \mathscr{H}^d(d))(=0)$. Thus
	\[{\rm H}^{d-1}(X; \mathbf{K}^{\mathrm{M}}_{d}/\ell^r)={\rm H}^{d-1}_{\rm Zar}(X; \mathscr{H}^d(d))=0.\]
	We are done.
\end{proof}
Let's now treat the case when the dimension of $X$ is odd. Still assume $k=\bar{k}$. By the exact sequence
\[0\to\mathbf{I}^{d+1}\to\mathbf{K}^{\mathrm{MW}}_{d}\xrightarrow{\tau}\mathbf{K}^{\mathrm{M}}_{d}\to0\]
we get an exact sequence
\[{\rm H}^{d-1}(X; \mathbf{K}^{\mathrm{MW}}_{d})\xrightarrow{\tau}{\rm H}^{d-1}(X; \mathbf{K}^{\mathrm{M}}_{d})\to{\rm H}^{d}(X; \mathbf{I}^{d+1}).\]
The Rost-Schmid complex for $\mathbf{I}^{d+1}$ says that ${\rm H}^{d}(X; \mathbf{I}^{d+1})$ is a subquotient of $\displaystyle\bigoplus_{x\in X^{(d)}}\mathbf{I}(\kappa_x)=\bigoplus_{x\in X^{(d)}}\mathbf{I}(\bar{k})=0$, thus ${\rm H}^{d}(X; \mathbf{I}^{d+1})=0$ and $\tau$ is surjective as well. In fact, more is true: by Voevodsky's confirmation of the Milnor conjecture, we have an isomorphism of sheaves of abelian groups $\bar{\mathbf{I}}^{d+j}\cong\mathscr{H}^{d+j}(d+j)\ (j\geqslant 1)$, where $\mathscr{H}^{d+j}(d+j)$ is the Zariski sheaf associated to the presheaf $U\mapsto{\rm H}^{d+j}_{\rm \acute{e}t}(U; \mu_2^{\otimes d+j})$; by reason of cohomological dimension, $\mathscr{H}^{d+j}(d+j)|_X=0\ (j\geqslant 1)$ (restricting to the Zariski site of $X$). Thus we have $\bar{\mathbf{I}}^{j}|_X=0, j>d,  \mathbf{I}^{d+1}|_X=\mathbf{I}^{d+2}|_X=\cdots$.

The Arason-Pfister Hauptsatz gives $\bigcap_{j\geqslant 1}\mathbf{I}^{d+j}=0$, thus $\mathbf{I}^{d+1}|_X=0$ and so $\tau: \mathbf{K}^{\mathrm{MW}}_{j}|_X\to\mathbf{K}^{\mathrm{M}}_{j}|_X$ is in fact an isomorphism for every $j\geqslant d$. This suffices to conclude that the induced maps on cohomologies $\tau: {\rm H}^{i}(X; \mathbf{K}^{\mathrm{MW}}_{j})\to{\rm H}^{i}(X; \mathbf{K}^{\mathrm{M}}_{j})$ for $j\geqslant d=\dim\ X$ are isomorphisms, since these sheaves are strictly $\mathbb{A}^1$-invariant, Nisnevich and Zariski cohomologies of $X$ coincide (and are computed by Rost-Schmid complexes).

The exact sequence \[0\to2\mathbf{K}^{\mathrm{M}}_{j}\to\mathbf{K}^{\mathrm{M}}_{j}\to\bar{\mathbf{I}}^{j}\to0\]
gives isomorphisms
\[{\rm H}^{i}(X; 2\mathbf{K}^{\mathrm{M}}_{j})\xrightarrow{\cong}{\rm H}^{i}(X; \mathbf{K}^{\mathrm{M}}_{j}), j>d.\]
We summarize the results as follows:
\begin{equation}
\begin{cases}
{\rm H}^{i}(X; \mathbf{I}^{j})=0,\ \ {\rm H}^{i}(X; 2\mathbf{K}^{\mathrm{M}}_{j})\cong{\rm H}^{i}(X; \mathbf{K}^{\mathrm{M}}_{j}),\ j>d; \\
\tau: {\rm H}^{i}(X; \mathbf{K}^{\mathrm{MW}}_{j})\xrightarrow{\cong}{\rm H}^{i}(X; \mathbf{K}^{\mathrm{M}}_{j}),\ j\geqslant d.
\end{cases}
\end{equation}

\begin{prop}\label{Delt-d-1}
	Let $k$ be an algebraically closed field with ${\rm char}(k)\ne 2$, $A$ a smooth affine $k$-algebra of Krull dimension $d\geqslant 3$, and $X={\rm Spec}(A)$, let $\xi$ be a stable oriented vector bundle over $X$, whose classifying map is still denoted $\xi: X\to\mathrm{BSL}$. If $\xi$ is represented by a rank $d-1$ vector bundle, then there is a bijection
	\[p_*^{-1}(\xi)\cong{\rm coker}\big({\rm SK}_1(X)\xrightarrow{\Delta(c_{d}, \xi)}{\rm H}^{d-1}(X; \mathbf{K}^{\mathrm{M}}_{d})\big).\]
	The homomorphism $\Delta(c_d, \xi)$ is given as follows: for $\beta\in{\rm SK}_1(X)=[X, \mathrm{SL}]_{\mathbb{A}^1}$,
	\[\Delta(c_d, \xi)\beta=(\Omega c_d)(\beta)+\sum_{r=1}^{d-1}((\Omega c_{r})(\beta))\cdot c_{d-r}(\xi).\]
	
	So $\Delta(c_d, \xi)([\beta_{d+1}(a, b)])=0$ for all $A$-point $(a, b)$ of $Q_{2d+1}$.
\end{prop}
\begin{proof}
	We already treated in \Cref{Delt-evend-1} the case when $d$ is even. For $d$ odd, since  $\pi_{d-1}^{\mathbb{A}^1}F_{d-1}\cong\mathbf{K}^{\mathrm{MW}}_d$, we have a similar  two-stage Moore-Postnikov factorization as in diagram (\ref{MPfact}), with $\mathbf{K}^{\mathrm{M}}_d$ replaced by $\mathbf{K}^{\mathrm{MW}}_d$ there.
	
	There are the following commutative diagrams:
	\[
	\begin{tikzcd}
	\mathrm{BSL}_{d}\arrow[r,"e_d"]\arrow[d, "s_d"']  & \mathrm{K}(\mathbf{K}^{\mathrm{MW}}_d, d) \arrow[d, "\tau"] \\
	\mathrm{BSL}\arrow[r,"c_d"]\ar[ru, "\theta"] & \mathrm{K}(\mathbf{K}^{\mathrm{M}}_d, d),
	\end{tikzcd}
	\qquad
	\text{and hence}
	\qquad
	\begin{tikzcd}[column sep = 6em]
	{\rm SK}_1(X) \arrow{r}{\Delta(\theta, \xi)} \arrow[rd, "{\Delta(c_d, \xi)}"']  & {\rm H}^{d-1}(X; \mathbf{K}^{\mathrm{MW}}_d) \arrow[d, "\tau"]\\
	& {\rm H}^{d-1}(X; \mathbf{K}^{\mathrm{M}}_d).
	\end{tikzcd}\]
	Since $\tau: \mathbf{K}^{\mathrm{MW}}_{d}|_X\to\mathbf{K}^{\mathrm{M}}_{d}|_X$ is an isomorphism, so is the right vertical map. Thus $\Delta(\theta, \xi)$ and $\Delta(c_d, \xi)$ are essentially the same. Therefore the result for the $d$ odd case holds as with the case when $d$ is even in \Cref{Delt-evend-1}.
	
	The last statement follows from \Cref{nicesim}.
\end{proof}
\begin{remark}
Since $c_d(\xi)=\tau\theta(\xi)$, we see that $\xi$ lifts to a class in $[X, E]_{\mathbb{A}^1}$ iff $c_d(\xi)=0$. While $\theta'_*$ maps $[X, E]_{\mathbb{A}^1}$ to $0$, hence no further obstruction. We thus get Murthy's splitting result \cite{Mur94} for oriented rank $d$ vector bundles: Let $X$ be a smooth affine variety of dimension $d$ over an algebraically closed field $k$, then an oriented rank $d$ vector bundle $\xi$ over $X$ splits off a trivial line bundle iff $c_d(\xi)=0$. See also \cite[Remark 1.33]{Mor}.
\end{remark}
\begin{prop}\label{divisability-cancel-d-1}
	Assume that $k=\bar{k}$ and ${\rm char} (k)\neq 2$. Let $A$ be a smooth $k$-algebra of Krull dimension $d\geqslant 3$. Let $X={\rm Spec}\ A$, $\xi\in[X, \mathrm{BSL}]_{\mathbb{A}^1}$ which is represented by a rank $d-1$ vector bundle.
	\begin{enumerate}[label=\emph{(\arabic*)}]
		\item We have $(d-1)!\cdot{\rm H}^{d-1}(X; \mathbf{K}^{\mathrm{M}}_{d})\subset{\rm im}\Delta(c_d, \xi)$.
		\item If ${\rm char}(k)=0$ or ${\rm char}(k)\geqslant d$, then the lifting set $p_*^{-1}(\xi)\subset[X, E]_{\mathbb{A}^1}$ is a singleton. So the map
		\[\theta_*: \pi_1({\rm RMap}(X, \mathrm{BSL}), \xi)\to\pi_1({\rm RMap}(X, \mathrm{K}(\mathbf{K}^{\mathrm{M}}_d, d)), 0)={\rm H}^{d-1}(X; \mathbf{K}^{\mathrm{M}}_{d})\]
		is surjective.
	\end{enumerate}
\end{prop}
\begin{proof}
	This is along the same lines as the proof of \Cref{divisability-cancel}. We only briefly write down some points. Using the Postnikov tower of $\mathbb{A}^d\setminus0$, it is easy to see that we have a surjective map $[X, \mathbb{A}^d\setminus0]_{\mathbb{A}^1}\twoheadrightarrow{\rm H}^{d-1}(X; \mathbf{K}^{\mathrm{MW}}_{d})$. Thus the composite
	\[{\rm Um}_d(A)/{\rm E}_d(A)=[X, Q_{2d-1}]_{\mathbb{A}^1}=[X, \mathbb{A}^d\setminus0]_{\mathbb{A}^1}\twoheadrightarrow{\rm H}^{d-1}(X; \mathbf{K}^{\mathrm{MW}}_{d})\xrightarrow{\tau}{\rm H}^{d-1}(X; \mathbf{K}^{\mathrm{M}}_{d})\]
	is surjective. So every element in ${\rm H}^{d-1}(X; \mathbf{K}^{\mathrm{M}}_{d})$ is the image of some $[a, b]$ with $(a, b)\in Q_{2d-1}(A), a=(a_1, \cdots, a_d)\in{\rm Um}_{d}(A)$ which we write as $\tau([a, b])$, by abusing notaion as in the proof of \Cref{divisability-cancel}.
	
	We will show
	\[(d-1)!\cdot\tau([a, b])=\pm\Delta(c_d, \xi)([\beta_d(a, b)]).\]
	
	As in the proof of \Cref{divisability-cancel} (from \cref{chern-sus} to the end, where $X$ essentially plays no role), we have
	\[\tau(({\rm pr}_1)_*[\beta])=\tau([(1,0,\cdots,0)\cdot\beta])=\pm(\Omega c_d)([\beta]), \text{for}\ \beta\in{\rm SL}_d(A)\subset{\rm SL}(A)\]
	and
	\begin{equation*}
	(\Omega c_{r})(\alpha_d)= \begin{cases}
	0,  & r\neq d; \\
	\pm(d-1)!\in\mathbf{K}^{\mathrm{M}}_{0}(k)=\mathbb{Z}, & r= d.
	\end{cases}
	\end{equation*}
	We get
	\[\Delta(c_d, \xi)([\beta_d(a, b)])=(\Omega c_d)([\beta_d(a, b)]).\]
	Thus
	\[
	\begin{split}
	(d-1)!\cdot\tau([a, b])&=\tau((d-1)!\cdot[a, b])=\tau((\psi_d)_*[a, b])=\tau(({\rm pr}_1)_*[\beta_d(a, b)])\\
	&=\pm(\Omega c_d)([\beta_d(a, b)])=\pm\Delta(c_d, \xi)([\beta_d(a, b)]).
	\end{split}
	\]
	This finishes proving (1). Statement (2) then follows easily from (1), the divisibility result in \Cref{divofcohmil} and \cite[Chapter I, Lemma 7.3]{GJ}.
\end{proof}

Finally we study the map $q_*: [X, E']_{\mathbb{A}^1}\to[X, E]_{\mathbb{A}^1}$. By the discussion in \Cref{hogpfibs}, we get exact sequences
\begin{equation}
\begin{cases}
\pi_{d}^{\mathbb{A}^1}(\mathbb{A}^{d}\setminus 0)\to \pi_{d}^{\mathbb{A}^1}F_{d-1}\to\mathbf{K}^{\mathrm{M}}_{d+1}=\pi_d^{\mathbb{A}^1}F_d\to0,  & d\ \text{odd}; \\
\pi_{d}^{\mathbb{A}^1}(\mathbb{A}^{d}\setminus 0)\to \pi_{d}^{\mathbb{A}^1}F_{d-1}\to2\mathbf{K}^{\mathrm{M}}_{d+1}\to0,  & d\ \text{even},
\end{cases}
\end{equation}
where in the case $d$ even, the term $2\mathbf{K}^{\mathrm{M}}_{d+1}$ sits in an exact sequence $0\to2\mathbf{K}^{\mathrm{M}}_{d+1}\to\mathbf{K}^{\mathrm{MW}}_{d+1}=\pi_d^{\mathbb{A}^1}F_d\to\mathbf{I}^{d+1}\to0$, and $\mathbf{I}^{d+1}|_X=0$ if $k=\bar{k}$, telling that the canonical homomorphism $2\mathbf{K}^{\mathrm{M}}_{d+1}\to\mathbf{K}^{\mathrm{MW}}_{d+1}=\pi_d^{\mathbb{A}^1}F_d$ induces an isomorphism ${\rm H}^d(X; 2\mathbf{K}^{\mathrm{M}}_{d+1})\cong{\rm H}^d(X; \pi_d^{\mathbb{A}^1}F_d)$. Using the fact that the Nisnevich cohomological dimension of $X$ is bounded above by $\dim(X)=d$, we get exact sequences for highest degree cohomology:
\begin{equation*}
\begin{cases}
{\rm H}^d(X; \pi_{d}^{\mathbb{A}^1}(\mathbb{A}^{d}\setminus 0))\to{\rm H}^d(X; \pi_{d}^{\mathbb{A}^1}F_{d-1})\to{\rm H}^d(X; \mathbf{K}^{\mathrm{M}}_{d+1})\to0,  & d\ \text{odd}; \\
{\rm H}^d(X; \pi_{d}^{\mathbb{A}^1}(\mathbb{A}^{d}\setminus 0))\to{\rm H}^d(X; \pi_{d}^{\mathbb{A}^1}F_{d-1})\to{\rm H}^d(X; 2\mathbf{K}^{\mathrm{M}}_{d+1})\to0,  & d\ \text{even}.
\end{cases}
\end{equation*}
If $k=\bar{k}$, then these two exact sequences become one:
\begin{equation}\label{mainvan}
{\rm H}^d(X; \pi_{d}^{\mathbb{A}^1}(\mathbb{A}^{d}\setminus 0))\to{\rm H}^d(X; \pi_{d}^{\mathbb{A}^1}F_{d-1})\to{\rm H}^d(X; \mathbf{K}^{\mathrm{M}}_{d+1})\cong{\rm H}^d(X; \pi_d^{\mathbb{A}^1}F_d)\to0, d\geqslant 3.
\end{equation}

We now invoke the following conjecture of Asok-Fasel describing $\pi_{d}^{\mathbb{A}^1}(\mathbb{A}^{d}\setminus 0)$ (see \cite[Conjecture 7]{AF13}).
\begin{cj}[Asok-Fasel]\label{afcj}
	Let $k$ be a perfect field with ${\rm char}(k)\ne 2$, then there is a sequence
	\begin{equation}
	\mathbf{K}^{\mathrm{M}}_{d+2}/24\to\pi_{d}^{\mathbb{A}^1}(\mathbb{A}^{d}\setminus 0)\to\mathbf{GW}^d_{d+1}\to0
	\end{equation}
	of homomorphisms in $\ab_{k}^{\mathbb{A}^1}$ which is exact at $\pi_{d}^{\mathbb{A}^1}(\mathbb{A}^{d}\setminus 0)$ and also becomes exact at $\mathbf{GW}^d_{d+1}$ after $d-3$-fold contraction.
\end{cj}

This gives an exact sequence
\begin{equation}
{\rm H}^d(X; \mathbf{K}^{\mathrm{M}}_{d+2}/24)\to{\rm H}^d(X; \pi_{d}^{\mathbb{A}^1}(\mathbb{A}^{d}\setminus 0))\to{\rm H}^d(X; \mathbf{GW}^d_{d+1})\to0
\end{equation}
if $X$ is a smooth $k$-scheme of dimension $d$.

The exact sequence on cohomologies above follows easily from the statement on contracted sheaves by a diagram chase using Rost-Schmid complexes. This conjecture is stably true after $d$-fold contractions, as confirmed by the recently published work \cite{RSO} of R\"ondigs-Spitzweck-{\O}stv{\ae}r.

If $k=\bar{k}$, then any generator of the group $\mathbf{K}^{\mathrm{M}}_2(k)$ can be written in the form $\{a^{24}, b\}=24\{a, b\}$ by the group law of Milnor K-theory. Thus $\mathbf{K}^{\mathrm{M}}_2(k)/24=0$. Since ${\rm H}^d(X; \mathbf{K}^{\mathrm{M}}_{d+2}/24)$ is a subquotient of $\displaystyle\bigoplus_{x\in X^{(d)}}\mathbf{K}^{\mathrm{M}}_2(\kappa_x)/24\cong\bigoplus_{x\in X^{(d)}}\mathbf{K}^{\mathrm{M}}_2(k)/24=0$, we see ${\rm H}^d(X; \mathbf{K}^{\mathrm{M}}_{d+2}/24)=0$ and so if \Cref{afcj} holds, then
\begin{equation}
{\rm H}^d(X; \pi_{d}^{\mathbb{A}^1}(\mathbb{A}^{d}\setminus 0))\xrightarrow{\cong}{\rm H}^d(X; \mathbf{GW}^d_{d+1}).
\end{equation}

We have the sheafified Karoubi periodicity sequences
\[\mathbf{K}^{\mathrm{Q}}_{d+1}\xrightarrow{H}\mathbf{GW}^d_{d+1}\xrightarrow{\eta}\mathbf{GW}^{d-1}_d\xrightarrow{f}\mathbf{K}^{\mathrm{Q}}_d\]
in $\ab_{k}^{\mathbb{A}^1}$, which is exact. Let $\mathbf{A}:={\rm im}(H), \mathbf{B}:={\rm im}(\eta)$, then we have exact sequences
\[\mathbf{K}^{\mathrm{Q}}_{d+1}\xrightarrow{H}\mathbf{A}\to0,\]
\[0\to\mathbf{A}\to\mathbf{GW}^d_{d+1}\xrightarrow{\eta}\mathbf{B}\to0,\]
yielding exact sequences on cohomologies:
\[
{\rm H}^d(X; \mathbf{K}^{\mathrm{M}}_{d+1})\cong{\rm H}^d(X; \mathbf{K}^{\mathrm{Q}}_{d+1})\to{\rm H}^d(X; \mathbf{A})\to0,\]
\[
{\rm H}^d(X; \mathbf{A})\to{\rm H}^d(X; \mathbf{GW}^d_{d+1})\to{\rm H}^d(X; \mathbf{B})\to0\]
and hence
\[
{\rm H}^d(X; \mathbf{K}^{\mathrm{M}}_{d+1})\to{\rm H}^d(X; \mathbf{GW}^d_{d+1})\to{\rm H}^d(X; \mathbf{B})\to0.\]

Contracting the sheafified Karoubi periodicity sequence $d$-times we get an exact sequence $\mathbf{K}^{\mathrm{M}}_1\xrightarrow{H}\mathbf{GW}^0_1\xrightarrow{\eta}\mathbf{B}_{-d}\to0$. While the composite $\mathbf{K}^{\mathrm{M}}_1\xrightarrow{H}\mathbf{GW}^3_1\xrightarrow{f}\mathbf{K}^{\mathrm{M}}_1$ is multiplication by $2$, we see the composite $2\mathbf{K}^{\mathrm{M}}_1\hookrightarrow\mathbf{K}^{\mathrm{M}}_1\xrightarrow{H}\mathbf{GW}^0_1$ is $0$, so we have an exact sequence $\mathbf{K}^{\mathrm{M}}_1/2\xrightarrow{H}\mathbf{GW}^0_1\xrightarrow{\eta}\mathbf{B}_{-d}\to0$, which splits into two: $\mathbf{K}^{\mathrm{M}}_1/2\xrightarrow{H}\mathbf{A}_{-d}\to0, 0\to\mathbf{A}_{-d}\to\mathbf{GW}^0_1\xrightarrow{\eta}\mathbf{B}_{-d}\to0$. Using again Rost-Schmid complexes we find an exact sequence ${\rm H}^d(X; \mathbf{K}^{\mathrm{M}}_{d+1}/2)\to{\rm H}^d(X; \mathbf{A})\to0$, with ${\rm H}^d(X; \mathbf{A})\to{\rm H}^d(X; \mathbf{GW}^d_{d+1})\to{\rm H}^d(X; \mathbf{B})\to0$ we obtain an exact sequence
\begin{equation}
0={\rm H}^d(X; \mathbf{K}^{\mathrm{M}}_{d+1}/2)\to{\rm H}^d(X; \mathbf{GW}^d_{d+1})\to{\rm H}^d(X; \mathbf{B})\to0.
\end{equation}

By \cite[Lemma 3.6.3]{AF15}, we have ${\rm H}^d(X; \mathbf{B})\cong{\rm Ch}^d(X)$, where ${\rm Ch}^d(X)={\rm H}^d(X; \mathbf{K}^{\mathrm{M}}_d/2)\cong{\rm CH}^d(X)/2$ is the group of mod-$2$ codimension-$d$ cycle classes on $X$. Since $k=\bar{k}$, we have ${\rm Ch}^d(X)=0$ (in fact, ${\rm CH}^d(X)$ is uniquely divisible, see e.g. \cite{Sr0cy} for a discussion). Thus ${\rm H}^d(X; \mathbf{GW}^d_{d+1})=0$ and so ${\rm H}^d(X; \pi_{d}^{\mathbb{A}^1}(\mathbb{A}^{d}\setminus 0))=0$ (assuming \Cref{afcj}). And by (\ref{mainvan}),
\begin{equation}\label{cohodiffib}
{\rm H}^d(X; \pi_{d}^{\mathbb{A}^1}F_{d-1})\xrightarrow{\cong}{\rm H}^d(X; \pi_d^{\mathbb{A}^1}F_d)\cong{\rm H}^d(X; \mathbf{K}^{\mathrm{M}}_{d+1})\cong[X, \mathbb{A}^{d+1}\setminus0]_{\mathbb{A}^1}\cong[X, Q_{2d+1}]_{\mathbb{A}^1}, d\geqslant 3.
\end{equation}

\begin{prop}\label{stage2cancel-d-1}
	Assume that $k=\bar{k}$ and ${\rm char} (k)\neq 2$. Let $A$ be a smooth $k$-algebra of Krull dimension $d\geqslant 3$, let $X={\rm Spec}\ A$. \emph{Assume \Cref{afcj} holds.} Then the map $q_*: [X, E']_{\mathbb{A}^1}\to[X, E]_{\mathbb{A}^1}$ is a bijection.
\end{prop}
\begin{proof}
	Since there is an $\mathbb{A}^1$-homotopy fibre sequence
	\[\mathrm{K}(\pi_{d}^{\mathbb{A}^1}F_{d-1}, d)\to E'\xrightarrow{q}E\xrightarrow{\theta'}\mathrm{K}(\pi_{d}^{\mathbb{A}^1}F_{d-1}, d+1),\]
	this shows $q_*$ is surjective, and gives a homotopy fibre sequence in $\sset_*$:
	\[{\rm RMap}(X, E')\xrightarrow{q_*}{\rm RMap}(X, E)\xrightarrow{\theta'_*}{\rm RMap}(X, \mathrm{K}(\pi_{d}^{\mathbb{A}^1}F_{d-1}, d+1)).\]
	
	By \cite[Chapter I, Lemma 7.3]{GJ}, for any $\xi\in[X, \mathrm{BSL}]_{\mathbb{A}^1}$ which is represented by a rank $d-1$ vector bundle (or equivalently, $c_d(\xi)=0$), let $\xi_E\in[X, E]_{\mathbb{A}^1}$ be the unique lifting of $\xi$ as in (\ref{MPfact}) (so $\theta'_*(\xi_E)=0\in{\rm H}^{d+1}(X; \pi_{d}^{\mathbb{A}^1}F_{d-1})$), to show the injectivity of $q_*$, it suffices to show the surjectivity of the map
	\[\theta'_*: \pi_{1}({\rm RMap}(X, E), \xi_E)\to\pi_{1}({\rm RMap}(X, \mathrm{K}(\pi_{d}^{\mathbb{A}^1}F_{d-1}, d+1)), 0)={\rm H}^d(X; \pi_{d}^{\mathbb{A}^1}F_{d-1}).\]
	
	Consider now the comparison diagram of Moore-Postnikov towers
	\[
	\xymatrix{F_{d-1}\ar[r] \ar[d] & {\rm BSL}_{d-1}\ar[r] \ar[d]& E'\ar[r]^{q} \ar[d]&  E \ar[r]^{p}\ar[d]^{p} &{\rm BSL} \ar@{=}[d]\\
		F_d\ar[r] &{\rm BSL}_{d} \ar[r] & \tilde{E} \ar[r] &  {\rm BSL} \ar@{=}[r] & {\rm BSL},
	}
	\]
	where $\tilde{E}$ is the first stage in the Moore-Postnikov tower of the map ${\rm BSL}_{d}\to{\rm BSL}$. By functoriality, the first stage $k$-invariants (that of the column of $E$) gives a commutative square
	\begin{equation}
	\begin{tikzcd}
	E \arrow[r, "\theta'"] \arrow[d, "p"']  & \mathrm{K}(\pi_{d}^{\mathbb{A}^1}F_{d-1}, d+1)  \arrow[d] \\
	{\rm BSL} \arrow[r, "\tilde{\theta}"]          & \mathrm{K}(\pi_{d}^{\mathbb{A}^1}F_d, d+1),
	\end{tikzcd}
	\label{compkinv}
	\end{equation}
	where $\tilde{\theta}$ is the $k$-invariant ``$\theta$'' in the rank $d$ case, we write it as $\tilde{\theta}$ to distinguish it from the $k$-invariant ``$\theta$'' in the rank $d-1$ case here) and the right vertical map is induced by the map $F_{d-1}\to F_d$. Applying ${\rm RMap}(X, -)$ we obtain a commutative diagram
	\begin{equation}
	\begin{tikzcd}
	{\rm RMap}(X, E) \arrow[r, "\theta'"] \arrow[d, "p"']  & {\rm RMap}(X, \mathrm{K}(\pi_{d}^{\mathbb{A}^1}F_{d-1}, d+1))  \arrow[d] \\
	{\rm RMap}(X, {\rm BSL}) \arrow[r, "\tilde{\theta}"]          & {\rm RMap}(X, \mathrm{K}(\pi_{d}^{\mathbb{A}^1}F_d, d+1)),
	\end{tikzcd}
	\end{equation}
	and hence
	\begin{equation}\label{maindiagram}
	\begin{tikzcd}
	c''\in\pi_{1}({\rm RMap}(X, E), \xi_E) \arrow[r, "\theta'_*"] \arrow[d, "p_*"'] & \pi_{1}({\rm RMap}(X, \mathrm{K}(\pi_{d}^{\mathbb{A}^1}F_{d-1}, d+1)), 0)\ni c  \arrow[d, "\cong"] \\
	T_{\xi}[\beta_{d+1}(a, b)]\in\pi_{1}({\rm RMap}(X, {\rm BSL}), \xi) \arrow[r, "\tilde{\theta}_*"]          & \pi_{1}({\rm RMap}(X, \mathrm{K}(\pi_{d}^{\mathbb{A}^1}F_d, d+1)), 0)\ni c \ar[d, equal]\\
	{[\beta_{d+1}(a, b)]\in\pi_1({\rm RMap}(X, \mathrm{BSL}), 0)}\ar[u, "T_{\xi}"]\ar[r, "{\Delta(c_{d+1}, \xi)}"]  & {{\rm H}^{d}(X; \mathbf{K}^{\mathrm{M}}_{d+1})\ni c}\\
	{[X, Q_{2d+1}]_{\mathbb{A}^1}\cong{\rm H}^{d}(X; \mathbf{K}^{\mathrm{MW}}_{d+1})\ni c'=[a, b]} \ar[u, "\beta_{d+1}"] \ar[ur, "\pm d!"'],
	\end{tikzcd}
	\end{equation}
	where the arrow $\tilde{\theta}_*$ is surjective by \Cref{divisability-cancel,divisability-cancel-even}, and the right vertical maps are isomorphisms by (\ref{cohodiffib}); the middle square commutes by \Cref{chern-delta}, and the lower triangle is given by (\ref{chern-sus-d!}), the arrow labeled by $\pm d!$ is surjective (see \Cref{d!-div}) . So for any $c\in\pi_{1}({\rm RMap}(X, \mathrm{K}(\pi_{d}^{\mathbb{A}^1}F_{d-1}, d+1)), 0)={\rm H}^{d}(X; \mathbf{K}^{\mathrm{M}}_{d+1})$, we can find $c'=[a, b]\in[X, Q_{2d+1}]_{\mathbb{A}^1}$ with $c=\pm d!\cdot[a, b]=\Delta(c_{d+1}, \xi)([\beta_{d+1}(a, b)])$.
	
	\Cref{chern-delta} (again) and the last statement of \Cref{Delt-d-1} tell that
	\[\theta_*(T_{\xi}[\beta_{d+1}(a, b)])=(c_d)_*(T_{\xi}[\beta_{d+1}(a, b)])=0.\]
	
	On the other hand, the fiber sequence
	\[{\rm RMap}(X, E)_{\xi_E}\xrightarrow{p}{\rm RMap}(X, {\rm BSL})_{\xi}\xrightarrow{\theta}{\rm RMap}(X, \mathrm{K}(\mathbf{K}^{\mathrm{M(W)}}_d, d))_0\]
	in $\sset_*$, where the subscripts refer the corresponding components, induces another
	\[\Omega_{\xi_E}{\rm RMap}(X, E)\xrightarrow{p}\Omega_{\xi}{\rm RMap}(X, {\rm BSL})\xrightarrow{\theta}{\rm RMap}(X, \mathrm{K}(\mathbf{K}^{\mathrm{M(W)}}_d, d-1)).\]
	Thus $T_{\xi}[\beta_{d+1}(a, b)]\in\ker(\theta_*)={\rm im}(p_*)$, we see there exists $c''\in\pi_{1}({\rm RMap}(X, E), \xi_E)$ with $p_*(c'')=T_{\xi}[\beta_{d+1}(a, b)]$, which then satisfies $\theta_*'(c'')=c$, proving that the map
	\[\theta'_*: \pi_{1}({\rm RMap}(X, E), \xi_E)\to\pi_{1}({\rm RMap}(X, \mathrm{K}(\pi_{d}^{\mathbb{A}^1}F_{d-1}, d+1)), 0)={\rm H}^d(X; \pi_{d}^{\mathbb{A}^1}F_{d-1})\]
	is surjective. Hence $q_*$ is injective as well.
\end{proof}

We finally arrive at the following cancellation result for (oriented) rank $d-1$ vector bundles over a smooth affine variety of dimension $d$, admitting Asok-Fasel conjecture. (As before, we let $\varphi=\varphi_{d-1}: \mathrm{BSL}_{d-1}\to\mathrm{BSL}$ be the stabilizing map.)

\begin{theorem}\label{cancel-d-1}
	Assume that the base field $k$ is algebraically closed and ${\rm char} (k)\neq 2$. Let $A$ be a smooth $k$-algebra of Krull dimension $d\geqslant 3$. Let $X={\rm Spec}\ A$, $\xi\in[X, \mathrm{BSL}]_{\mathbb{A}^1}$ which is represented by a rank $d-1$ vector bundle (or equivalently, $c_d(\xi)=0$). Let $\varphi_*: [X, \mathrm{BSL}_{d-1}]_{\mathbb{A}^1}\to[X, \mathrm{BSL}]_{\mathbb{A}^1}$ be the stabilizing map. If ${\rm char}(k)=0$ or ${\rm char}(k)\geqslant d$, then the lifting set $\varphi_*^{-1}(\xi)\subset[X, \mathrm{BSL}_{d-1}]_{\mathbb{A}^1}$ is a singleton, provided \Cref{afcj} holds. In other words, every oriented rank $d-1$ vector bundle over $X$ is cancellative.
\end{theorem}

\section{Application to cancellation properties of symplectic vector bundles}

In this last section, we give enumeration results and in particular study cancellation properties of symplectic vector bundles, using the ideas and methods of the previous sections.

By \cite[Theorems 2.3.5, 3.3.3, and 4.1.2]{AHW2} (see also \cite{PW10}), the set of classes of rank $2n$ symplectic vector bundles over a smooth affine scheme over a perfect field is represented by the motivic space $\mathrm{BSp}_{2n}$.

The map $\mathrm{Sp}_{2n}\to\mathrm{Sp}_{2n+2}, A\mapsto\left( \begin{array} { c c } { A } & \\  & { \rm I_2 }  \end{array} \right)$ induces an $\mathbb{A}^1$-homotopy fiber sequence
\[\mathbb{A}^{2n+2}\setminus{0}\simeq\mathrm{Sp}_{2n+2}/\mathrm{Sp}_{2n}\to\mathrm{BSp}_{2n}\to\mathrm{BSp}_{2n+2}.\]
Since
\[\pi_i^{\mathbb{A}^1}(\mathbb{A}^{2n+2}\setminus 0)=\begin{cases}
0,  & i<2n+1, \\
\mathbf{K}^{\mathrm{MW}}_{2n+2},  & i=2n+1,
\end{cases}\]
we find that the map $\pi_i^{\mathbb{A}^1}\mathrm{BSp}_{2n}\to\pi_i^{\mathbb{A}^1}\mathrm{BSp}_{2n+2}$ is an isomorphism if $i<2n+1$ and a surjection if $i=2n+1$. Thus also,  the induced map $\pi_i^{\mathbb{A}^1}\mathrm{BSp}_{2n}\to\pi_i^{\mathbb{A}^1}\mathrm{BSp}$ is an isomorphism if $i<2n+1$ and a surjection if $i=2n+1$ (see also \cite[Theorems 2.6 and 3.3]{AF14}).

Let $F'_n$ be the $\mathbb{A}^1$-homotopy fiber of the canonical map $\varphi: \mathrm{BSp}_{2n}\rightarrow\mathrm{BSp}$. We then have an $\mathbb{A}^1$-homotopy fiber sequence
\[F'_n\to\mathrm{BSp}_{2n}\xrightarrow{\varphi}\mathrm{BSp},\]
and from the above results we see that 
\[\pi_j^{\mathbb{A}^1}F'_n=0, j<2n+1.\]
As before, we have an $\mathbb{A}^1$-homotopy fiber sequence
\[\mathbb{A}^{2n+2}\setminus{0}\to F'_n\to F'_{n+1},\]
which then gives
$\pi_{2n+1}^{\mathbb{A}^1}(\mathbb{A}^{2n+2}\setminus{0})\xrightarrow{\cong}\pi_{2n+1}^{\mathbb{A}^1}F'_n$ and we obtain
\[\pi_i^{\mathbb{A}^1}F'_n=\begin{cases}
0,  & i<2n+1, \\
\mathbf{K}^{\mathrm{MW}}_{2n+2},  & i=2n+1.
\end{cases}\]

The motivic space $\mathrm{BSp}\times\mathbb{Z}$ represents symplectic K-theory in the motivic homotopy category $\EuScript{H}^{\mathbb{A}^1}(k)$, and so $\pi_i^{\mathbb{A}^1}\mathrm{BSp}\cong\mathbf{K}^{\mathrm{Sp}}_i, i\geqslant 1$ and $\pi_1^{\mathbb{A}^1}\mathrm{BSp}\cong\mathbf{K}^{\mathrm{Sp}}_1=0$. Moreover, there is a motivic $T^{\wedge4}$-spectrum with term $\mathrm{BSp}\times\mathbb{Z}$ at each level (see e.g. \cite[Theorem 3]{SchT}), yielding as before the following result.
\begin{prop}\label{bsp-comp}
	If $k$ is a perfect field, then $\mathrm{BSp}$ is an abelian group object in the pointed $\mathbb{A}^1$-homotopy category $\EuScript{H}^{\mathbb{A}^1}_*(k)$ and for any $X\in\EuScript{S}\mathsf{m}_k$, all the components of the derived mapping space ${\rm RMap}(X, \mathrm{BSp})$ are weakly equivalent.
\end{prop}
Thus there is a canonical map $m: \mathrm{BSp}\times\mathrm{BSp}\to\mathrm{BSp}$ in the pointed $\mathbb{A}^1$-homotopy category $\EuScript{H}^{\mathbb{A}^1}_*(k)$ giving the abelian group object structure of $\mathrm{BSp}$. Given $\xi\in[X, \mathrm{BSp}]_{\mathbb{A}^1}$, let \[T_{\xi}:=m(-, \xi)_*: \pi_1({\rm RMap}(X, \mathrm{BSp}), 0)\to\pi_1({\rm RMap}(X, \mathrm{BSp}), \xi)\]
be the isomorphism on fundamental groups induced by $m(-, \xi): {\rm RMap}(X, \mathrm{BSp})_0\to{\rm RMap}(X, \mathrm{BSp})_{\xi}$ as introduced in \Cref{abhspcomp}.

\vspace{4mm}

Let $E=E_{2n}$ be the first non-trivial stage of the Moore-Postnikov tower of the canonical map $\mathrm{BSp}_{2n}\rightarrow\mathrm{BSp}$. Then we obtain a principal $\mathbb{A}^1$-homotopy fiber sequence
\[\mathrm{K}(\mathbf{K}^{\mathrm{MW}}_{2n+2}, 2n+1)\to E\to\mathrm{BSp}\]
classified by a $k$-invariant $b_{n+1}\in\widetilde{\rm CH}^{2n+2}(\mathrm{BSp})={\rm H}^{2n+2}(\mathrm{BSp}; \mathbf{K}^{\mathrm{MW}}_{2n+2})$ (which is given by a map $b_{n+1}: \mathrm{BSp}\to\mathrm{K}(\mathbf{K}^{\mathrm{MW}}_{2n+2}, 2n+2)$ in $\EuScript{H}^{\mathbb{A}^1}(k)$).

We have a commutative (indeed, cartesian) square
\[
\begin{tikzcd}
\mathrm{Sp}_{2n} \ar[r]\ar[d] &\mathrm{Sp}_{2n+2}\ar[d]\\
\mathrm{SL}_{2n+1} \ar[r] &\mathrm{SL}_{2n+2},
\end{tikzcd}\]
giving a commutative diagram
\[
\begin{tikzcd}
\mathrm{BSp}_{2n} \ar[r]\ar[d] &\mathrm{BSp}_{2n+2}\ar[d]\\
\mathrm{BSL}_{2n+1} \ar[r] &\mathrm{BSL}_{2n+2},
\end{tikzcd}\]
which in turn, by naturality of $k$-invariants, yields a commutative diagram
\begin{equation}
\begin{tikzcd}\label{bsp-kinv}
\mathrm{BSp}_{2n+2} \ar[r, "b_{n+1}"]\ar[d] &\mathrm{K}(\mathbf{K}^{\mathrm{MW}}_{2n+2}, 2n+2)\ar[d, equal]\\
\mathrm{BSL}_{2n+2} \ar[r, "e_{2n+2}"] &\mathrm{K}(\mathbf{K}^{\mathrm{MW}}_{2n+2}, 2n+2),
\end{tikzcd}
\end{equation}
where $e_{2n+2}\in\widetilde{\rm CH}^{2n+2}(\mathrm{BSL}_{2n+2})={\rm H}^{2n+2}(\mathrm{BSL}_{2n+2}; \mathbf{K}^{\mathrm{MW}}_{2n+2})$ is the universal Euler class. The universal Euler class restricts to the universal \emph{Borel class}, see e.g. \cite[Proposition 4.3]{HoWt}, where it is called the \emph{Pontryagin class} and denoted ${\rm p}_{n+1}(\mathrm{Sp}_{2n+2})$; we follow the nowadays common nomenclature---the Borel class, as in the newest version of \cite{PW10}. Thus the $k$-invariant $b_{n+1}\in\widetilde{\rm CH}^{2n+2}(\mathrm{BSp})={\rm H}^{2n+2}(\mathrm{BSp}; \mathbf{K}^{\mathrm{MW}}_{2n+2})$ is exactly the $n+1$-st Borel class.

The Borel classes $b_i$'s satisfy a Whitney formula (or rather, Cartan sum formula, see \cite[Theorem 4.10]{HoWt} or \cite{PW10}):
\[b_n(\xi+\xi')=\sum_{r=0}^{n}b_r(\xi)\cdot b_{n-r}(\xi')\]
for any $U\in\EuScript{S}\mathsf{m}_k$ and $(\xi, \xi')\in[U, \mathrm{BSp}]_{\mathbb{A}^1}\times[U, \mathrm{BSp}]_{\mathbb{A}^1}$. As in \Cref{comulp}, we obtain
\begin{equation}
m^*b_n=\sum_{r=0}^{n}(b_r\times b_{n-r})\in[\mathrm{BSp}\times\mathrm{BSp}, \mathrm{K}(\mathbf{K}^{\mathrm{MW}}_{2n}, 2n)]_{\mathbb{A}^1},
\end{equation}
where $b_r\times b_{n-r}:=\mu(b_r\otimes b_{n-r})$, $\mu$ being the obvious map
\[[\mathrm{BSp}, \mathrm{K}(\mathbf{K}^{\mathrm{MW}}_{2r}, 2r)]_{\mathbb{A}^1}\otimes[\mathrm{BSp}, \mathrm{K}(\mathbf{K}^{\mathrm{MW}}_{2n-2r}, 2n-2r)]_{\mathbb{A}^1}\to[\mathrm{BSp}\times\mathrm{BSp}, \mathrm{K}(\mathbf{K}^{\mathrm{MW}}_{2n}, 2n)]_{\mathbb{A}^1}\]
induced by the multiplication in the graded sheaves of Milnor-Witt K-groups $\mathbf{K}^{\mathrm{MW}}_*$.

\vspace{4mm}

If $k$ is a perfect field and $X={\rm Spec}(A)$ is a smooth affine $k$-scheme with $\dim(X)=d\leqslant 2n+1$, then an easy argument using Moore-Postnikov tower shows that any stable symplectic vector bundle $\xi$ over $X$ is represented by a rank $2n$ symplectic vector bundle $\xi_{2n}\in[X, \mathrm{BSp}_{2n}]_{\mathbb{A}^1}$, i.e. $\xi: X\to\mathrm{BSp}$ factors through the canonical map $\varphi: \mathrm{BSp}_{2n}\rightarrow\mathrm{BSp}$ (in the $\mathbb{A}^1$-homotopy category $\EuScript{H}^{\mathbb{A}^1}(k)$). If $d<2n+1$, then another Moore-Postnikov tower argument shows that the representative symplectic vector bundle $\xi_{2n}\in[X, \mathrm{BSp}_{2n}]_{\mathbb{A}^1}$ is unique. So the main problem is to consider in the case $d=2n+1$, the uniqueness of the lifting, or the cancellation property of the symplectic vector bundle $\xi_{2n}$.

As in \Cref{enu-oddrk} and \Cref{chern-delta}, we obtain the following enumeration result on symplectic vector bundles near critical rank (of course, only the $d=2n+1$ case is interesting; the $d<2n+1$ cases are almost trivial by the Moore-Postnikov tower argument), where we denote by $\EuScript{V}_{2n}^{\rm Sp}(X, \xi):=\varphi_*^{-1}(\xi)$ the set of isomorphism classes of rank-$2n$ symplectic vector bundles which represent $\xi$.
\begin{theorem}\label{borel-delta}
	Let $k$ be a perfect field with ${\rm char}(k)\neq 2$ and let $X={\rm Spec}(A)$ be a smooth affine $k$-scheme of dimension $d$ with $3\leqslant d\leqslant 2n+1$. Let $\xi$ be a stable symplectic vector bundle over $X$, whose classifying map is still denoted by $\xi: X\to\mathrm{BSp}$. Denote by
	\[\Delta(b_{n+1}, \xi):=b_{n+1}\circ T_{\xi}: {\rm KSp}_1(X)\to{\rm H}^{2n+1}(X; \mathbf{K}^{\mathrm{MW}}_{2n+2})\]
	the induced homomorphism on fundamental groups. Then there is a bijection
	\[\EuScript{V}_{2n}^{\rm Sp}(X, \xi)\longleftrightarrow{\rm coker}(\Delta(b_{n+1}, \xi)).\]
	
	Moreover, the map $\Delta(b_{n+1}, \xi)$ is given by
	\begin{equation}
	\Delta(b_{n+1}, \xi)(\beta)=(\Omega b_{n+1})(\beta)+\sum_{r=1}^{n}((\Omega b_{r})(\beta))\cdot b_{n+1-r}(\xi).
	\end{equation}
\end{theorem}

We again turn to more refined computations in order to give cancellation results on symplectic vector bundles near critical rank. As discussed above, we only need to consider the case $d=2n+1$ and we do assume this in the sequel.

Consider the following commutative diagram in $\EuScript{H}^{\mathbb{A}^1}(k)$:
\[\begin{tikzcd}[row sep = 3em, column sep = 3em]
&  & {\rm SL}_{d+1} \arrow[d, hookrightarrow] \ar[r, "H"] & {\rm Sp}_{2(d+1)}\subset{\rm Sp}\ar[r, "\Omega b_r"] &\mathrm{K}(\mathbf{K}^{\mathrm{MW}}_{2r}, 2r-1) \\
X\arrow[r, "{(a, b)}"]  &  Q_{2d+1} \arrow[r, "\alpha_{d+1}"] \arrow[ru, "\beta_{d+1}"]          & {\rm SL}_{2^d},
\end{tikzcd}\]
where $H: {\rm SL}_r\to{\rm Sp}_{2r}, A\mapsto\left( \begin{array} { c c } { A } & \\  & { (A^T)^{-1} }  \end{array} \right)$. So
\begin{equation*}
(\Omega b_r)(H\beta_{d+1})\in\widetilde{{\rm H}}^{2r-1}(Q_{2d+1}; \mathbf{K}^{\mathrm{MW}}_{2r})\cong \begin{cases}
0,  & 2r\neq d+1; \\
\mathbf{K}^{\mathrm{MW}}_{0}(k)={\rm GW}(k), & 2r= d+1\ (r=n+1).
\end{cases}
\end{equation*}
It follows that $ (\Omega b_r)(H\beta_{d+1})=0, \forall r\neq n+1$  and that
\begin{equation}
\Delta(b_{n+1}, \xi)(H\beta_{d+1})=(\Omega b_{n+1})(H\beta_{d+1}).
\end{equation}

Let $F: {\rm Sp}\hookrightarrow{\rm SL}$ be the ``forgetful'' map induced by the inclusions ${\rm Sp}_{2r}\hookrightarrow{\rm SL}_{2r}$, then we have the induced maps
\[{\rm SK}_1(X)\xrightarrow{H}{\rm KSp}_1(X)\xrightarrow{F}{\rm SK}_1(X).\]
By \cite[Lemma 5.3]{Sus77}, $[\alpha_{d+1}(a, b)]+[\alpha_{d+1}(a, b)^T]=0$ in ${\rm SK}_1(X)$, for any $A$-point $(a, b)$ of $Q_{2d+1}$ (since the matrix $I_r$ in \cite[Lemma 5.3]{Sus77} is in ${\rm SL}_{2^r}(\mathbb{Z})$, hence elementary). Thus
\[FH([\alpha_{d+1}(a, b)])=\left[\left( \begin{array} { c c } { \alpha_{d+1}(a, b) } & \\  & { (\alpha_{d+1}(a, b)^T)^{-1} }  \end{array} \right)\right]=[\alpha_{d+1}(a, b)]-[\alpha_{d+1}(a, b)^T]=2[\alpha_{d+1}(a, b)].\]
As $[\alpha_{d+1}(a, b)]=[\beta_{d+1}(a, b)]\in{\rm SK}_1(X)$, we also get
\[FH([\beta_{d+1}(a, b)])=2[\beta_{d+1}(a, b)].\]

The diagram (\ref{bsp-kinv}) then yields the following commutative diagrams
\begin{equation}\label{bor-ch}
\begin{tikzcd}
\mathrm{BSp} \ar[r, "b_{n+1}"]\ar[d, "F"'] &\mathrm{K}(\mathbf{K}^{\mathrm{MW}}_{d+1}, d+1)\ar[d, "\tau"]\\
\mathrm{BSL} \ar[r, "c_{d+1}"] &\mathrm{K}(\mathbf{K}^{\mathrm{M}}_{d+1}, d+1),
\end{tikzcd}
\qquad
\text{and hence}
\qquad
\begin{tikzcd}
\mathrm{Sp} \ar[r, "\Omega b_{n+1}"]\ar[d, "F"'] &\mathrm{K}(\mathbf{K}^{\mathrm{MW}}_{d+1}, d)\ar[d, "\tau"]\\
\mathrm{BSL} \ar[r, "\Omega c_{d+1}"] &\mathrm{K}(\mathbf{K}^{\mathrm{M}}_{d+1}, d).
\end{tikzcd}
\end{equation}
Composing the latter with the map $H: {\rm SL}\to{\rm Sp}$ and applying $[X, -]_{\mathbb{A}^1}$ we see that
\[\tau(\Omega b_{n+1})H=(\Omega c_{d+1})FH: {\rm SK}_1(X)\to{\rm H}^d(X; \mathbf{K}^{\mathrm{M}}_{d+1}).\]

Considering the effect on $[\beta_{d+1}(a, b)]\in{\rm SK}_1(X)$ for any $A$-point $(a, b)$ of $Q_{2d+1}$, we find that
\begin{align*}
\tau(\Omega b_{n+1})H([\beta_{d+1}(a, b)])&=2(\Omega c_{d+1})([\beta_{d+1}(a, b)])=\pm2\cdot\tau([a_0^{d!},a_1, \cdots, a_d])\\
&=\tau(\pm2\cdot \frac{d!}{2}h\cdot[(a, b)])=\tau(\pm d!\cdot h\cdot[(a, b)])
\end{align*}
by \Cref{factoresu} and \cref{chern-sus}, where we write $a=(a_0,a_1, \cdots, a_d)$ and $[(a, b)]$ denotes the image of $(a, b)\in Q_{2d+1}(A)$ in $[X, Q_{2d+1}]_{\mathbb{A}^1}\cong{\rm H}^d(X; \mathbf{K}^{\mathrm{MW}}_{d+1})$.

As noted before, if ${\rm char}(k)\neq 2$ and ${\rm c.d.}_2(k)\leqslant 2$, then ${\rm H}^d(X; \mathbf{I}^{d+2})=0$ and the reduction map $\tau: {\rm H}^d(X; \mathbf{K}^{\mathrm{MW}}_{d+1})\to{\rm H}^d(X; \mathbf{K}^{\mathrm{M}}_{d+1})$ is an isomorphism (if ${\rm c.d.}(k)\leqslant 1$, one can also use the Arason-Pfister Hauptsatz and the motivic Bloch-Kato conjecture to show that $\mathbf{I}^{d+2}|_X=0$). Then we must have
\[\Delta(b_{n+1}, \xi)H([\beta_{d+1}(a, b)])=(\Omega b_{n+1})H([\beta_{d+1}(a, b)])=\pm d!\cdot h\cdot[(a, b)].\]
This says that $d!\cdot h\cdot{\rm H}^d(X; \mathbf{K}^{\mathrm{MW}}_{d+1})\subset{\rm im}\ \Delta(b_{n+1}, \xi), (2\cdot d!)\cdot{\rm H}^d(X; \mathbf{K}^{\mathrm{M}}_{d+1})\subset{\rm im}(\tau\Delta(b_{d+1}, \xi))$. Since $h$ goes to $2$ under $\tau$ and $\tau: {\rm H}^d(X; \mathbf{K}^{\mathrm{MW}}_{d+1})\to{\rm H}^d(X; \mathbf{K}^{\mathrm{M}}_{d+1})$ is an isomorphism, in fact we have $(d!\cdot h)\cdot{\rm H}^d(X; \mathbf{K}^{\mathrm{MW}}_{d+1})=(2\cdot d!)\cdot{\rm H}^d(X; \mathbf{K}^{\mathrm{MW}}_{d+1})\subset{\rm im}\ \Delta(b_{d+1}, \xi)$. We summarize the discussion as follows.

\begin{theorem}\label{enu-spbdl}
	Let  $k$ be a perfect field with ${\rm char}(k)\neq 2$ and ${\rm c.d.}_2(k)\leqslant 2$. Let $X={\rm Spec}(A)$ be a smooth affine $k$-scheme of dimension $d=2n+1\geqslant 3$. Let $\xi$ be a stable symplectic vector bundle over $X$, whose classifying map is still denoted by $\xi: X\to\mathrm{BSp}$.
	\begin{enumerate}[label=\emph{(\arabic*)}]
		\item We have $(2\cdot d!)\cdot{\rm H}^d(X; \mathbf{K}^{\mathrm{M}}_{d+1})\subset{\rm im}(\tau\Delta(b_{n+1}, \xi))$.
		\item If ${\rm H}^d(X; \mathbf{K}^{\mathrm{M}}_{d+1})$ is $d!$-divisible, then ${\rm coker}\Delta(b_{n+1}, \xi)=0$. In this case, any rank $2n=d-1$ symplectic vector bundle is cancellative. Moreover,  the map
		\[(b_{n+1})_*: \pi_1({\rm RMap}(X, \mathrm{BSp}), \xi)\to\pi_1({\rm RMap}(X, \mathrm{K}(\mathbf{K}^{\mathrm{MW}}_{d+1}, d+1)), 0)={\rm H}^{d}(X; \mathbf{K}^{\mathrm{MW}}_{d+1})\]
		is surjective for every $\xi\in[X, \mathrm{BSp}]_{\mathbb{A}^1}$.
		\item If ${\rm c.d.}(k)\leqslant 1$, then $\xi$ is represented by a unique symplectic vector bundle over $X$ of rank $2n=d-1$.
	\end{enumerate}
\end{theorem}
Note that for statement (2) above, if an abelian group is $d!$-divisible, then it is also $2\cdot d!$-divisible. The last statement above is again guaranteed by Voevodsky's confirmation of the motivic Bloch-Kato conjecture as in \Cref{d!-div}.

\vspace{5mm}

{\noindent\bf Acknowledgements.}
First of all, I will express my deep indebtedness to my Ph.D. advisor, professor Jean Fasel, for many discussions on mathematics and for sharing numerous mathematical ideas with me, many of which are used in this work, e.g. the idea of using Suslin matrices. I also thank Tariq Syed for discussing many things on $\mathbb{A}^1$-homotopy theory and on vector bundles, and Aravind Asok for pointing out the paper of James and Thomas and his comment that our result extends to general base fields. Special thanks go to Daniel R. Grayson for useful discussions and to Burt Totaro for helpful communication regarding the Chow K\"unneth formula in \cite{TotCw97,Tot14} and directing me to \cite{Tot16}. I'm also grateful to the anonymous referees for their useful comments, pointing out gaps and giving revision suggestions, which improved the exposition considerably.

\providecommand{\bysame}{\leavevmode\hbox to3em{\hrulefill}\thinspace}
\providecommand{\MR}{\relax\ifhmode\unskip\space\fi MR }
\providecommand{\MRhref}[2]{%
  \href{http://www.ams.org/mathscinet-getitem?mr=#1}{#2}
}
\providecommand{\href}[2]{#2}

\vspace{8mm}

\noindent{\textsc{School of Mathematical Sciences,\\ The University of Nottingham,\\ University Park, Nottingham, NG7 2RD}}

\vspace{2mm}

\noindent{Email: pengdudp@gmail.com}

\end{document}